\documentclass[12pt, reqno]{amsart}
\usepackage{amsmath}
\usepackage{mathrsfs}
\usepackage{amstext}
\usepackage{amsxtra}
\usepackage{amssymb}
\usepackage{amsgen}
\usepackage{amsfonts}
\usepackage{amsthm}
\usepackage{amscd}
\usepackage{latexsym}
\usepackage{fullpage}
\usepackage{cite}

\newtheorem{theorem}{Theorem}[section]
\newtheorem*{definition}{Definition}
\newtheorem{lemma}{Lemma}[section]
\newtheorem{proposition}{Proposition}[section]
\newtheorem{corollary}{Corollary}[section]
\newtheorem*{assumption}{A Priori Assumption}

\newcommand{\oht}{\mathfrak{H}}
\newcommand{\ohtb}{\overline{\mathfrak{H}}}
\newcommand{\nht}{\mathcal{H}}
\newcommand{\nhtb}{\overline{\mathcal{H}}}
\newcommand{\pv}{\,\text{p.v.}}
\newcommand{\sgn}{\,\text{sgn}}
\newcommand{\eqnlineup}{\!\!\!\!}
\newcommand{\cc}{\text{c.c.}}

\numberwithin{equation}{section}

\title{A Rigorous Justification of the Modulation Approximation to the 2D Full Water Wave Problem}
\author{Nathan Totz}
\author{Sijue Wu}
\address{Department of Mathematics \\
University of Michigan \\
Ann Arbor, MI, 48105}
\thanks{The authors are supported in part by NSF grant DMS-0800194.}

\begin{document}

\begin{abstract}
We consider the 2D inviscid incompressible irrotational infinite depth water wave problem neglecting surface tension.  Given wave packet initial data of the form $\epsilon B(\epsilon \alpha)e^{ik\alpha}$ for $k > 0$, we show that the modulation of the solution is a profile traveling at group velocity and governed by a focusing cubic nonlinear Schr\"odinger equation, with rigorous error estimates in Sobolev spaces.  As a consequence, we establish existence of solutions of the water wave problem in Sobolev spaces for times of order $O(\epsilon^{-2})$ provided the initial data differs from the wave packet by at most $O(\epsilon^{3/2})$ in Sobolev spaces.  These results are obtained by directly applying modulational analysis to the evolution equation with no quadratic nonlinearity constructed in \cite{WuAlmostGlobal2D} and by the energy method.
\end{abstract}

\maketitle

\section{Introduction}

The mathematical problem of two dimensional water waves concerns the evolution of an interface separating an inviscid, incompressible, irrotational fluid, under the influence of gravity, from a region of zero density (e.g., air) in two dimensional space.  It is assumed that the fluid region lies below the air region.  Assume the fluid is infinitely deep and has density 1, and that the gravitational field is $g = (0, -1)$.  At $t \geq 0$, denote the fluid interface by $\Sigma(t)$ and the fluid region by $\Omega(t)$.  If surface tension is neglected, then the motion of the fluid is described by

$$\begin{matrix}
\begin{cases} & \mathbf{v}_t + \mathbf{v} \cdot \nabla \mathbf{v} = g  - \nabla p \cr
&\text{div}\,\mathbf{v} = 0 , \qquad \text{curl}\,\mathbf{v} = 0
\end{cases} & \qquad\; \text{on } \Omega(t), \, t \geq 0 \cr
p = 0 & \text{on } \Sigma(t)
\end{matrix}$$
\begin{equation}\label{EulerVelocityField}
(\mathbf{v}, 1) \text{ 
is tangent to   
the free surface } (\Sigma(t), t)
\end{equation} where $\mathbf{v}$ is the fluid velocity, $p$ is the fluid pressure.

Assume further that the interface $\Sigma(t)$ is parametrized by $z = z(\alpha, t)$, where $\alpha \in \mathbb{R}$ is the Lagrangian coordinate, i.e., $z_t(\alpha, t) = \mathbf{v}(z(\alpha, t), t)$.  Let $\mathfrak{a} = -\frac{\partial p}{\partial \mathbf{n}} \frac{1}{|z_\alpha|}$, where $\mathbf{n}=\frac{iz_\alpha}{|z_\alpha|}$ is the unit outward normal of $\Omega(t)$.  We know from \cite{WuLocal2DWellPosed} that \eqref{EulerVelocityField} is equivalent to the following complex system on the interface:
\begin{equation}\label{OldEuler} z_{tt} - i\mathfrak{a}z_\alpha = -i\end{equation}
\begin{equation}\label{ztIsAntihol} (I - \ohtb)z_t = 0,\end{equation}
where  $\oht$ is the Hilbert transform associated to the fluid region $\Omega(t)$: $$\oht f(\alpha, t) = \frac{1}{\pi i} \pv \int_{-\infty}^\infty \frac{f(\beta, t)z_\beta(\beta, t)}{z(\alpha, t) - z(\beta, t)} d\beta$$

In this paper we consider the modulation approximation to the infinite depth water wave equations \eqref{OldEuler}-\eqref{ztIsAntihol}, i.e., a solution which is to the leading order a wave packet of the form 
\begin{equation}\label{WavePacket}
\epsilon B(\epsilon \alpha, \epsilon t, \epsilon^2 t)e^{i(k\alpha + \omega t)}
\end{equation}  
It is well-known  (c.f. \cite{PeterMillerAsymptoticAnalysis}, \cite{JohnsonWaterWaves}) that if one performs a multiscale analysis to determine modulation approximations to the finite or infinite depth 2D water wave equations, one should expect to find that the amplitude $B$ is a profile that travels at the group velocity determined by the dispersion relation of the water wave equations over time intervals of length $O(\epsilon^{-1})$, and evolves according to a  nonlinear Schr\"odinger equation (NLS) over time intervals of length $O(\epsilon^{-2})$.  
The first formal derivations of the NLS from the 2D water wave equations was obtained by Zakharov \cite{ZakharovInfiniteDepth} for the infinite depth case, and by Hasimoto and Ono \cite{HashimotoOno} for the finite depth case. In \cite{CraigSulemSulemNLSFiniteDepth}, Craig, Sulem and Sulem applied modulation analysis to the finite depth 2D water wave equation,  derived an approximate solution of the form of a wave packet and showed that the modulation approximation satisfies the 2D finite depth water wave equation to leading order.

A rigorous justification of the NLS from the full water wave equations would bring us one step closer to understanding qualitative properties for wave packet-like solutions of the water wave equations from  that of solutions to NLS on the appropriate time scales.  
Rigorous justifications of the KdV, KP, Boussinesq, shallow water and various other asymptotic models from the full water wave equations have been done in \cite{CraigWWExistenceBousKdV}, \cite{SchneiderWayveJustifyKdV}, \cite{LannesJustifyWWModel3D}.  
As was noted in \cite{CraigSulemSulemNLSFiniteDepth}, the reason that a justification for NLS has not been given is that the longest existence time in Sobolev spaces for the water waves equation demonstrated thus far have been on time scales of the order $O(\epsilon^{-1})$, for data with Sobolev norms of the order $O(\epsilon)$.  However these times are too short to distinguish the NLS behavior of the wave packet from simple translation of the initial wave packet at group velocity.  Since there is no existence result in Sobolev spaces on the necessary time scales, an attempt to justify NLS as a rigorous modulation approximation to the water wave system on that scale has not been made.

 Let $U_g f = f \circ g$, and for $\kappa: \mathbb R\to \mathbb R$ a diffeomorphism we introduce the notation
\begin{equation*}\zeta := z \circ \kappa^{-1}, \qquad U_\kappa^{-1} D_t := \partial_t U_\kappa^{-1}, \qquad U_\kappa^{-1} \mathcal{P} := (\partial_t^2 - i\mathfrak{a}\partial_\alpha)U_\kappa^{-1}\end{equation*}

$$b := \kappa_t \circ \kappa^{-1}, \qquad U_\kappa^{-1} \mathcal{A} \partial_\alpha := \mathfrak{a} \partial_\alpha U_\kappa^{-1}$$

\begin{equation}\label{NewVariableNotation}
 D_t = (\partial_t + b \partial_\alpha), \qquad U_\kappa^{-1} \nht = \oht U_\kappa^{-1}, \qquad \mathcal{P} = D_t^2 - i\mathcal{A}\partial_\alpha
\end{equation}
In \cite{WuAlmostGlobal2D},  Wu showed that for any solution $z$ of \eqref{OldEuler}-\eqref{ztIsAntihol}, the quantity $\Pi := (I - \oht)(z - \overline{z})$  satisfies the equation 
\begin{align}\label{ChiEquation}
\mathcal{P}(\Pi \circ \kappa^{-1}) & = -2\left[D_t \zeta, \nht\frac{1}{\zeta_\alpha} + \nhtb\frac{1}{\overline{\zeta}_\alpha} \right]\partial_\alpha D_t \zeta + \frac{1}{\pi i} \int \left(\frac{D_t\zeta(\alpha,t) - D_t\zeta(\beta,t)}{\zeta(\alpha,t) - \zeta(\beta,t)}\right)^2 \partial_\beta(\zeta - \overline{\zeta}) d\beta \notag\\
& = \frac{4}{\pi} \int \frac{(D_t\zeta(\alpha,t) - D_t\zeta(\beta,t))(\Im \zeta(\alpha,t) - \Im \zeta(\beta,t))}{|\zeta(\alpha,t) - \zeta(\beta,t)|^2} \partial_\beta D_t\zeta(\beta,t) d\beta\notag \\
& \quad + \frac{2}{\pi} \int \left(\frac{D_t\zeta(\alpha,t) - D_t\zeta(\beta,t)}{\zeta(\alpha,t) - \zeta(\beta,t)}\right)^2 \partial_\beta \Im \zeta(\beta,t) d\beta
\end{align}
and furthermore there is a coordinate change $\kappa$, such that in this coordinate system, the equation \eqref{ChiEquation} contains no quadratic nonlinear terms. 
Using this favorable structure and the method of vector fields, Wu further proved the almost global well-posedness for the full water wave system \eqref{OldEuler}-\eqref{ztIsAntihol} for data small in the generalized $L^2$ Sobolev spaces defined by the invariant vector fields.  However, the wave packet data $\epsilon B(\epsilon \alpha)e^{ik\alpha}$  (for $B$ sufficiently smooth and localized) has slow decay at infinity, and in terms of the generalized Sobolev norms used in \cite{WuAlmostGlobal2D} these are at least of size $O(\epsilon^{-1/2})$. In terms of the standard Sobolev norms they are of size $O(\epsilon^{1/2})$.  Therefore  the standard 
$L^2$ Sobolev spaces suits our purposes better.

As is suggested by the work of \cite{KSMCubicNonlinearityLongtimeRemainder}, in justifying the modulation approximation for a nonlinear system it is advantageous if the nonlinear system contains no quadratic nonlinear terms. We therefore use the equation \eqref{ChiEquation} to  perform the multiscale analysis. In fact, we will use a slightly different change of variables $\kappa$ than that given in \cite{WuAlmostGlobal2D}.  Upon performing this multiscale analysis, we derive an approximate wave packet-like solution $\tilde{\zeta}$ satisfies the transformed equations (see \eqref{NewEuler}-\eqref{XiIsAntihol}
below) with a residual of size $O(\epsilon^4)$.  The special structure of \eqref{ChiEquation}  then allows us to obtain bounds for the error $r = \zeta - \tilde{\zeta}$ between the true solution and the approximate solution on the appropriate time scale in Sobolev spaces.  

We will see in the course of the multiscale analysis that the envelope of the leading term of $\tilde{\zeta} - \alpha$ obeys a focusing cubic nonlinear Schr\"odinger equation which is globally well-posed in sufficiently regular Sobolev spaces.  This implies that the approximate solution $\tilde{\zeta}$ is eternal.  This fact, along with the a priori bounds on the remainder $r$, allows us to show existence and uniqueness of solutions of the system \eqref{OldEuler}-\eqref{ztIsAntihol} on the proper $O(\epsilon^{-2})$ time scales, for initial data which is no more than $O(\epsilon^{3/2})$ away from a wave packet  
$\epsilon B(\epsilon \alpha)e^{ik\alpha}$ in Sobolev spaces.  A rigorous justification of wave packet approximations to  
 solutions of the water wave system is then obtained in this special coordinate system $\kappa$.  Upon changing variables, we obtain appropriate wave packet approximations to water waves in Lagrangian coordinates.  Finally, by introducing some further restrictions on the initial data, we justify an Eulerian version of the asymptotics.

\section{Derivation of the Main Equations}

In this section we introduce our notation as well as collect for future reference the main equations and formulas from \cite{WuAlmostGlobal2D} that we will use.  We first recall the definition of the Hilbert transform $\nht_\gamma$ associated to the interface determined by a curve parametrization $\gamma(\alpha) : \mathbb{R} \to \mathbb{C}\;$:
\begin{equation}\label{HilbertTransformDefinition}
\nht_\gamma f(\alpha) := \frac{1}{\pi i} \pv \int_{-\infty}^\infty \frac{\gamma_\beta(\beta)}{\gamma(\alpha) - \gamma(\beta)} f(\beta) d\beta
\end{equation}  We adopt the following notations for Hilbert transforms associated to specific curves: $\oht$ is the Hilbert transform associated to $z$ already defined, $\nht$ is the Hilbert transform associated to $\zeta$, and $\nht_0$ is the flat Hilbert transform associated to the line $\gamma(\alpha) = \alpha$.  In general, the Hilbert transform $\nht_\gamma$ satisfies the convention $\nht_\gamma 1 = 0$ and the identity $\nht_\gamma^2 = I$ in $L^2$.  Let $\Omega$ be a domain in $\mathbb R^2$, with $\partial\Omega$ parametrized by $\gamma(\alpha)$, $\alpha\in \mathbb R$, oriented clock-wisely. 
We know $f(\cdot) = F(\gamma(\cdot))\in L^2(\mathbb R)$ is the trace of a holomorphic function $F$ in $\Omega$ if and only if $(I - \nht_\gamma)f = 0$.
The celebrated result of \cite{CoifmanMeyerMcintoshL2Bounds} (see Theorem \ref{CMMWEstimates}) states that $\nht_\gamma$ is bounded on $L^2$ provided that $\gamma$ satisfies the chord-arc condition:  There exist constants $\nu, N > 0$ so that
\begin{equation}\label{ChordArcCondition}
\nu |\alpha - \beta| \leq |\gamma(\alpha) - \gamma(\beta)| \leq N |\alpha - \beta| \qquad \text{for all } \alpha, \beta \in \mathbb{R}.
\end{equation}

We will frequently use the properties of the Hilbert transform given in Lemmas 2.1 and 2.2 of \cite{WuAlmostGlobal2D} which for convenience are recorded here.  Note that in the sequel we will often be suppressing the dependence on $t$.

\begin{proposition}[c.f. Lemma 2.1 of \cite{WuAlmostGlobal2D}]\label{HilbertCommutatorIdentities}
Suppose that $z(\alpha, t)$ has no self-intersections at time $t \in [0, T_0]$ and satisfies $z_t, z_\alpha - 1 \in C^1([0, T_0]; H^1)$.  Then for all functions $f \in C^1(\mathbb{R} \times [0, T_0])$ having the property that $f_\alpha(\alpha, t) \to 0$ as $|\alpha| \to \infty$ we have the identities
$$[\partial_t, \oht]f = [z_t, \oht]\frac{f_\alpha}{z_\alpha}, \qquad  [\mathfrak{a}\partial_\alpha, \oht]f = [\mathfrak{a}z_\alpha, \oht]\frac{f_\alpha}{z_\alpha}, \qquad [\oht, \partial_\alpha/z_\alpha] = 0$$
$$[\partial_t^2, \oht]f = [z_{tt}, \oht]\frac{f_\alpha}{z_\alpha} + 2[z_t, \oht]\frac{f_{t\alpha }}{z_\alpha} - \frac{1}{\pi i}\int\left(\frac{z_t(\alpha) - z_t(\beta)}{z(\alpha) - z(\beta)}\right)^2 f_\beta(\beta) d\beta$$
$$[\partial_t^2 - i\mathfrak{a}\partial_\alpha, \oht]f = 2[z_t, \oht]\frac{f_{t\alpha }}{z_\alpha} - \frac{1}{\pi i}\int\left(\frac{z_t(\alpha) - z_t(\beta)}{z(\alpha) - z(\beta)}\right)^2 f_\beta(\beta) d\beta$$
$$(I - \oht)(-i\mathfrak{a}_t \overline{z}_\alpha) = 2[z_{tt}, \oht]\frac{\overline{z}_{t\alpha }}{z_\alpha} + 2[z_t, \oht]\frac{\overline{z}_{tt\alpha }}{z_\alpha} - \frac{1}{\pi i}\int\left(\frac{z_t(\alpha) - z_t(\beta)}{z(\alpha) - z(\beta)}\right)^2 \overline{z}_{t\beta }(\beta) d\beta$$
\end{proposition}

\textbf{Remark.} Observe that if we change variables via $\kappa$ each formula above has a corresponding formula in which $z$ is replaced by $\zeta$, $\partial_t$ is replaced by $D_t$, $\oht$ is replaced by $\nht$, etc.

\begin{proposition}[c.f. Lemma 2.2 of \cite{WuAlmostGlobal2D}]\label{HoloProperties}
Let $\Omega \subset \mathbb{C}$ be a region whose boundary $\partial\Omega$ is parametrized by $\gamma(\alpha)$, oriented clockwise.  
Then the following hold:
\begin{enumerate}
\item{If $f = \nht_\gamma f$ and $g = \nht_\gamma g$, then $[f, \nht_\gamma]g = 0$.}
\item{For all $f, g \in L^2(\partial\Omega)$, $[f, \nht_\gamma]\nht_\gamma g = -[\nht_\gamma f, \nht_\gamma]g$.}
\end{enumerate}
\end{proposition}

With these preparations, we give the change of variables used to convert \eqref{OldEuler}-\eqref{ztIsAntihol} into a more suitable equation for our purposes.  Originally, in \cite{WuAlmostGlobal2D}, the change of variables $\kappa$ was introduced using a Riemann map $\mathbf{\Phi}(z, t) : \Omega(t) \to P_-$ which for each $t$ mapped the fluid region $\Omega(t)$ to the lower half plane, and then defined by $\alpha \mapsto z(\alpha, t) + \overline{z}(\alpha, t) - h(\alpha, t)$, where $h$ was taken to be $\alpha \mapsto \mathbf{\Phi}(z(\alpha, t), t)$.   

However, the only property of $h$ that was used was that it was a real-valued trace of a holomorphic function defined on $\Omega(t)$.  This idea was already used in the 3D setting to prove global existence of solutions to the 3D water wave problem \cite{WuGlobal3D}.  We use it here by choosing to set $$h(\alpha, t) = z(\alpha, t) - \frac{1}{2}(I + \oht)(I + \mathfrak{K})^{-1}\left(z(\alpha, t) - \overline{z}(\alpha, t)\right),$$ where $\mathfrak{K} = \Re \oht$ is the double layer potential operator associated to the curve $z$.  It is easy to see from the definition that $h$ is a real-valued trace of a holomorphic function in $\Omega(t)$.  Then the change of variables is defined by
\begin{align}\label{ChangeOfVariables}
\kappa(\alpha, t) & = z(\alpha,t) + \overline{z}(\alpha,t) - h(\alpha, t) \notag \\
& = \overline{z}(\alpha, t) + \frac{1}{2}(I + \oht)(I + \mathfrak{K})^{-1}(z(\alpha, t) - \overline{z}(\alpha, t))
\end{align}  Our choice of $\kappa$ then gives us the crucial identity \begin{equation}\label{OldXiIsAntihol}
(I - \oht)(\overline{z} - \kappa) =- (I - \oht)\left(\frac{1}{2}(I + \oht)(I + \mathfrak{K})^{-1}(z - \overline{z})\right) = 0,
\end{equation} and from this it follows immediately in the new coordinates that
\begin{equation}
(I - \nht)(\overline{\zeta} - \alpha) = 0
\end{equation} and
\begin{equation}\label{ChiAndXiAreEqual}
\Pi \circ \kappa^{-1} = (I - \nht)(\zeta - \overline{\zeta}) = (I - \nht)(\zeta - \alpha)
\end{equation}
We denote $$\xi := \zeta - \alpha,$$ the perturbation of $\zeta$ from the rest state $\alpha$.  Then from \eqref{ChiEquation} and \eqref{OldXiIsAntihol} we have that solutions $z$ also satisfy the system
\begin{equation}\label{NewEuler}\mathcal{P}(I - \nht)\xi = G \end{equation} \begin{equation}\label{XiIsAntihol}(I - \nhtb)\xi = 0\end{equation} 
where as in \eqref{ChiEquation} the cubic nonlinearity $G$ is
\begin{equation}\label{GFormula} G := -2\left[D_t\zeta, \nht\frac{1}{\zeta_\alpha} + \nhtb\frac{1}{\overline{\zeta}_\alpha}\right]\partial_\alpha D_t\zeta + \frac{1}{\pi i}\int \left(\frac{D_t\zeta(\alpha) - D_t\zeta(\beta)}{\zeta(\alpha) - \zeta(\beta)}\right)^2 (\zeta_\beta(\beta) - \overline{\zeta}_\beta(\beta))\,d\beta \end{equation}
We will also need the equations corresponding to the time derivative, which by virtue of  \eqref{ztIsAntihol} and a derivative $D_t$ to \eqref{NewEuler}  are given by
\begin{equation}\label{DtNewEuler}(D_t^2 - i\mathcal{A}\partial_\alpha)D_t(I - \nht)\xi = D_t G + [\mathcal{P}, D_t](I - \nht)\xi \end{equation} \begin{equation}\label{DtXiIsAntihol}(I - \nhtb)D_t\zeta = 0\end{equation}
An explicit formula for $D_t G$ is
\begin{align}\label{DtGFormula}
D_t G & = -2\left[D_t^2 \zeta, \nht\frac{1}{\zeta_\alpha} + \nhtb\frac{1}{\overline{\zeta}_\alpha}\right]\partial_\alpha D_t \zeta - 2\left[D_t \zeta, \nht\frac{1}{\zeta_\alpha} + \nhtb\frac{1}{\overline{\zeta}_\alpha}\right]\partial_\alpha D_t^2 \zeta \notag\\
& + \frac{2}{\pi i}\int \left(\frac{D_t\zeta(\alpha) - D_t\zeta(\beta)}{\zeta(\alpha) - \zeta(\beta)}\right)^2 \partial_\beta D_t\zeta(\beta) \, d\beta - \frac{2}{\pi i}\int \frac{\left|D_t\zeta(\alpha) - D_t\zeta(\beta)\right|^2}{(\overline{\zeta}(\alpha) - \overline{\zeta}(\beta))^2} \partial_\beta D_t \zeta(\beta) \notag\\
& \quad + \frac{4}{\pi}\int \frac{(D_t\zeta(\alpha) - D_t\zeta(\beta))(D_t^2\zeta(\alpha) - D_t^2 \zeta(\beta))}{(\zeta(\alpha) - \zeta(\beta))^2} \partial_\beta \Im \zeta(\beta) d\beta \notag \\
& \quad + \frac{2}{\pi} \int \left(\frac{D_t\zeta(\alpha) - D_t\zeta(\beta)}{\zeta(\alpha) - \zeta(\beta)}\right)^2 \partial_\beta \Im D_t \zeta(\beta) d\beta \notag\\
& - \frac{4}{\pi} \int \left(\frac{D_t\zeta(\alpha) - D_t\zeta(\beta)}{\zeta(\alpha) - \zeta(\beta)}\right)^3 \partial_\beta \Im \zeta(\beta) \, d\beta
\end{align}
We also have the following formulas for $b$ and $\mathcal{A}$ in terms of $\zeta$ (c.f. Proposition 2.4 of \cite{WuAlmostGlobal2D} for a proof. From the proof, it is clear that \eqref{XiIsAntihol} and \eqref{DtXiIsAntihol} together implies \eqref{bFormula} and \eqref{AFormula}.):
\begin{equation}\label{bFormula}(I - \nht)b = -[D_t \zeta, \nht]\frac{\overline{\zeta}_\alpha - 1}{\zeta_\alpha},\end{equation}
\begin{equation}\label{AFormula}(I - \nht)\mathcal{A} = 1 + i[D_t^2 \zeta, \nht]\frac{\overline{\zeta}_\alpha - 1}{\zeta_\alpha} + i[D_t \zeta, \nht]\frac{\partial_\alpha D_t \overline{\zeta}}{\zeta_\alpha}\end{equation} The commutator in the right hand side of \eqref{DtNewEuler} can be rewritten using
\begin{equation}\label{PCommutatorDtFormula}[\mathcal{P}, D_t](I - \nht)\xi = U_{\kappa^{-1}}\left(\frac{\mathfrak{a}_t}{\mathfrak{a}}\right) i\mathcal{A}\partial_\alpha(I - \nht)\xi,\end{equation}
and is controlled using the following formula (c.f. (1.9) and (2.32) of \cite{WuAlmostGlobal2D} for a derivation):
\begin{align}\label{atOveraFormula}
(I - \nht)\biggl(\mathcal{A}\overline{\zeta}_\alpha U_{\kappa}^{-1}\left(\frac{\mathfrak{a}_t}{\mathfrak{a}}\right)\biggr) & = 2i[D_t^2 \zeta, \nht]\frac{\partial_\alpha D_t \overline{\zeta}}{\zeta_\alpha} + 2i[D_t \zeta, \nht]\frac{\partial_\alpha D_t^2 \overline{\zeta}}{\zeta_\alpha} \notag\\
& \quad -\; \frac{1}{\pi} \int \biggl(\frac{D_t\zeta(\alpha) - D_t\zeta(\beta)}{\zeta(\alpha) - \zeta(\beta)} \biggr)^2 \partial_\beta D_t \overline{\zeta}(\beta) d\beta
\end{align}
We also record Proposition 2.7 of \cite{WuAlmostGlobal2D}:
\begin{align}\label{DtbFormula}
(I - \nht)D_t b & = [D_t\zeta, \nht]\frac{\partial_\alpha(2b - D_t\overline{\zeta})}{\zeta_\alpha} - [D_t^2\zeta, \nht]\frac{\overline{\zeta}_\alpha - 1}{\zeta_\alpha} \notag \\ 
& \qquad + \frac{1}{\pi i} \int \left(\frac{D_t\zeta(\alpha) - D_t\zeta(\beta)}{\zeta(\alpha) - \zeta(\beta)}\right)^2 (\overline{\zeta}_\beta(\beta) - 1) d\beta
\end{align}
To estimate terms involving time derivatives of singular integral operators we record the following
\begin{lemma}\label{SingIntCommuteWithDt}
Suppose that $\mathcal{T} f = \int K(\alpha, \beta) \partial_\beta f(\beta)\, d\beta$.  Then $$[D_t, \mathcal{T}]f = \int (\partial_t + b(\alpha)\partial_\alpha + b(\beta)\partial_\beta)K(\alpha, \beta) \; \partial_\beta f(\beta) \, d\beta$$
\end{lemma}
\begin{proof}
We have
\begin{align*}
[D_t, \mathcal{T}]f & = (\partial_t + b(\alpha) \partial_\alpha)\int K(\alpha, \beta) f_\beta(\beta) \, d\beta - \int K(\alpha, \beta) \partial_\beta D_t f(\beta) \, d\beta \\
& = \int (\partial_t + b(\alpha)\partial_\alpha + b(\beta)\partial_\beta)K(\alpha, \beta) f_\beta(\beta) \, d\beta \\
& \quad + \int K(\alpha, \beta)\Bigl(b_\beta(\beta) f_\beta(\beta) + D_t f_\beta(\beta) - \partial_\beta D_t f(\beta)\Bigr) \, d\beta \\
& = \int (\partial_t + b(\alpha)\partial_\alpha + b(\beta)\partial_\beta)K(\alpha, \beta) f_\beta(\beta) \, d\beta
\end{align*} as desired.
\end{proof}

Denote the Fourier transform on $\mathbb{R}$ by $$\hat{f}(x) = \frac{1}{2\pi}\int_{-\infty}^\infty f(\alpha)e^{-ix\alpha} d\alpha$$ For $s \in \mathbb{R}$ we have the usual Sobolev spaces $$H^s = \{f \in L^2(\mathbb{R}) : \|f\|_{H^s} := \|(1 + |\cdot|^2)^{s/2}\hat{f}(\cdot)\|_{L^2} < \infty\}$$ and the homogeneous Sobolev spaces $$\dot{H}^s = \{f \in L^2(\mathbb{R}) : \|f\|_{\dot{H}^s} := \|\,|\cdot|^s\hat{f}(\cdot)\|_{L^2} < \infty\}$$ Also for $s \in \mathbb{N}$ we define $W^{s, \infty} = \{f \in L^\infty : \partial_\alpha^j f \in L^\infty, \, j = 1, \ldots, s\}$, with $\|f\|_{W^{s,\infty}}:=\sum_{j=0}^s\|\partial_\alpha^jf\|_{L^\infty}$.  A well-known consequence of the Sobolev embedding theorem is that  $H^s$ is continuously embedded in $W^{s - 1, \infty}$ for  $s \geq 1$.  Given a Banach space $X$, let $C([0, T]; X)$ be the space of all $f \in \mathbb R \times [0, T]$ so that $t \mapsto \|f(t)\|_X$ is continuous on $[0, T]$; equip $C([0, T]; X)$ with the norm $$\|f\|_{C([0, T]; X)} := \max_{t \in [0, T]} \|f(t)\|_X < \infty.$$

In the rest of the paper, we make the following 
\begin{assumption}
Let $s \geq 6$, and let $\zeta$ be a solution to the water wave system \eqref{NewEuler}-\eqref{XiIsAntihol}-\eqref{DtXiIsAntihol} on some time interval $[0, T_0]$ satisfying for $0 \leq t \leq T_0$ the bounds
\begin{equation}\label{ZetaLocalAPrioriBound}
\mathfrak{S}(T_0) := \|\zeta_\alpha - 1\|_{C([0, T_0]; H^s)} + \|D_t\zeta\|_{C([0, T_0]; H^{s })} \leq \delta.
\end{equation} First we choose $\delta > 0$ sufficiently small so that $\zeta$ satisfies the chord-arc condition \eqref{ChordArcCondition} and $\mathcal{A} \geq 1/2$ (c.f. \cite{WuAlmostGlobal2D}).  In the course of the paper we will need to choose $\delta$ smaller still.
\end{assumption}

In order to use the formulas \eqref{bFormula}, \eqref{AFormula}, \eqref{atOveraFormula} to get estimates for $b$, $\mathcal{A}$ and $U_\kappa^{-1} (\mathfrak{a}_t/\mathfrak{a})$ in $H^s$ we use the following lemma, whose proof is essentially that of Lemma 3.8 and Lemma 3.15 of \cite{WuAlmostGlobal2D}:
\begin{lemma}\label{DoubleLayerPotentialArgument}
Let $s \geq 4$, and suppose that $\zeta$ satisfies \eqref{ZetaLocalAPrioriBound}.  Then there exists a constant $C$ depending on $\mathfrak{S}(T_0)$, so that for all real-valued $f$ we have the following estimates:
\begin{enumerate}
\item{$\|f\|_{H^s} \leq C\|(I - \nht)f\|_{H^s}$}
\item{$\|f\|_{H^s} \leq C\|(I - \nht)\left(f\mathcal{A}\overline{\zeta}_\alpha\right)\|_{H^s}$}
\end{enumerate}
\end{lemma}

\section{The Formal Multiscale Calculation.}

The goal of this section is to derive a formal solution to the system \eqref{NewEuler}-\eqref{XiIsAntihol} which is to leading order a wave packet.  Since we want our approximation to remain bounded for times on the order $O(\epsilon^{-2})$, we calculate this formal solution using a multiscale analysis.  As mentioned in the introduction, we expect from similar formal derivations of modulation approximations to the water wave equations that the amplitude of the wave packet is a profile which travels at the group velocity of the water wave operator, and evolves according to a  nonlinear Schr\"odinger equation.

To effect this multiscale analysis, we must first formally expand the Hilbert transform $\nht$ appearing in the water wave equations.  In particular, we must intepret how the flat Hilbert transform $\nht_0$ acts on multiple scale functions of the form $F(\epsilon \alpha) e^{ik\alpha}$ for $k \neq 0$.

\subsection{Formal Expansion of the Hilbert Transform}

Understanding the system \eqref{NewEuler}, \eqref{XiIsAntihol} depends on understanding the Hilbert Transform $\nht$.  Since our first goal is to seek a perturbation expansion $$\zeta(\alpha, t) = \alpha + \xi = \alpha + \sum_{n = 1}^\infty \epsilon^n \zeta^{(n)}(\alpha, t, \epsilon),$$ we must find a corresponding development of $\nht$ into a formal power series $$\nht = \nht_0 + \epsilon\nht_1 + \epsilon^2 \nht_2 + \cdots$$ To predict what the terms of this series ought to be, we heuristically expand the kernel of $\nht$ in a formal power series as follows:
\begin{equation}
\nht f = \nht_0 f + \sum_{n = 1}^\infty \frac{(-1)^{n + 1}}{n \pi i}  \int f_\beta(\beta) \left(\frac{\xi(\alpha) - \xi(\beta)}{\alpha - \beta}\right)^n d\beta
\end{equation}  Equating like powers of $\epsilon$ on the right hand side of this last expression suggests the following formulas for $\nht_1$:
\begin{align*}
\nht_1 f & := \frac{1}{\pi i} \int f_\beta \left(\frac{\zeta^{(1)}(\alpha) - \zeta^{(1)}(\beta)}{\alpha - \beta}\right) d\beta \\
& = [\zeta^{(1)}, \nht_0]f_\alpha
\end{align*} and for $\nht_2$:
\begin{align}\label{HilbertFormulas}
\nht_2 f & := \frac{1}{\pi i}  \int f_\beta(\beta) \left(\frac{\zeta^{(2)}(\alpha) - \zeta^{(2)}(\beta)}{\alpha - \beta}\right) d\beta - \frac{1}{2\pi i} \int f_\beta(\beta) \left(\frac{\zeta^{(1)}(\alpha) - \zeta^{(1)}(\beta)}{\alpha - \beta}\right)^2 d\beta \notag\\
& = \frac{1}{\pi i} \int f_\beta(\beta) \left(\frac{\zeta^{(2)}(\alpha) - \zeta^{(2)}(\beta)}{\alpha - \beta}\right) d\beta \notag\\
& - \frac{1}{\pi i} \int f_{\beta} \zeta^{(1)}_\beta \left(\frac{\zeta^{(1)}(\alpha) - \zeta^{(1)}(\beta)}{\alpha - \beta}\right) d\beta + \frac{1}{2\pi i} \int f_{\beta \beta}(\beta) \left(\frac{(\zeta^{(1)}(\alpha) - \zeta^{(1)}(\beta))^2}{\alpha - \beta}\right) d\beta \notag\\
& = [\zeta^{(2)}, \nht_0]f_\alpha - [\zeta^{(1)}, \nht_0](\zeta^{(1)}_\alpha f_\alpha) + \frac{1}{2}[\zeta^{(1)}, [\zeta^{(1)}, \nht_0]]f_{\alpha\alpha}
\end{align} and so we define the approximate Hilbert Transform $$\tilde{\nht} := \nht_0 + \epsilon \nht_1 + \epsilon^2 \nht_2$$  If $\tilde{\nht}$ acts on a multiple scale function $f(\alpha_0, \alpha_1) = f(\alpha, \epsilon\alpha)$, then we have the expansion $$\tilde{\nht} = \nht^{(0)} + \epsilon \nht^{(1)} + \epsilon^2 \nht^{(2)} + O(\epsilon^3),$$ where $$\nht^{(0)}f = \nht_0f, \qquad \nht^{(1)}f = [\zeta^{(1)}, \nht_0]\partial_{\alpha_0}f,$$ \begin{equation}\label{MultiscaleHilbertFormulas}\nht^{(2)}f = [\zeta^{(1)}, \nht_0]\partial_{\alpha_1}f + [\zeta^{(2)}, \nht_0]\partial_{\alpha_0}f - [\zeta^{(1)}, \nht_0]\zeta^{(1)}_{\alpha_0}\partial_{\alpha_0}f + \frac{1}{2}[\zeta^{(1)}, [\zeta^{(1)}, \nht_0]]\partial_{\alpha_0}^2f\end{equation}
Later we will need to estimate the operator $$\nht - \tilde{\nht} = (\nht - \nht_{\tilde{\zeta}}) + (\nht_{\tilde{\zeta}} - \tilde{\nht}),$$ where $\nht_{\tilde{\zeta}}$ is the Hilbert transform associated to the curve given by the approximation $\tilde{\zeta}$.  We will see that for our purposes it suffices to develop the approximate solution $\tilde{\zeta}$ to the third order: $$\tilde{\zeta}(\alpha, t) = \alpha + \epsilon \zeta^{(1)}(\alpha, t) + \epsilon^2 \zeta^{(2)}(\alpha, t) + \epsilon^3\zeta^{(3)}(\alpha, t)$$  Hence we record the following formula as a first step towards analyzing $\nht_{\tilde{\zeta}} - \tilde{\nht}$:
\begin{lemma}\label{DiffHilbertPart1Formula}
$(\nht_{\tilde{\zeta}} - \tilde{\nht})f$ can be written as the following finite sum of singular integrals:
\begin{align}\label{HTildeZetaMinusTildeHFormula}
& (\nht_{\tilde{\zeta}} - \tilde{\nht})f = -\frac{1}{\pi i}\int \frac{\left(\tilde{\xi}(\alpha) - \tilde{\xi}(\beta)\right)^3\tilde{\zeta}_\beta(\beta)}{(\alpha - \beta)^3\left(\tilde{\zeta}(\alpha) - \tilde{\zeta}(\beta)\right)} f(\beta) d\beta \\
& + \sum_S \frac{C_{p_1, p_2}\epsilon^{n_1p_1 + n_2p_2 + m}}{\pi i} \int \frac{\left(\zeta^{(n_1)}(\alpha) - \zeta^{(n_1)}(\beta)\right)^{p_1}\left(\zeta^{(n_2)}(\alpha) - \zeta^{(n_2)}(\beta)\right)^{p_2}}{(\alpha - \beta)^{p_1 + p_2 + 1}}\zeta^{(m)}_\beta(\beta) f(\beta) d\beta \notag
\end{align} where $S = \{(n_1, n_2, m, p_1, p_2) : n_1p_1 + n_2p_2 + m \geq 3,\, 0 \leq p_1 + p_2 \leq 2,\, 0 \leq n_1,\,n_2,\, m \leq 3\}$ and $C_{p_1, p_2}$ are constants depending only on $p_1, p_2$.
\end{lemma} 

\begin{proof}
First observe that with an integration by parts we have the formulas $$\nht_1 f = \frac{1}{\pi i} \pv \int f(\beta)\left(\frac{\zeta^{(1)}_\beta(\beta)}{\alpha - \beta} - \frac{\zeta^{(1)}(\alpha) - \zeta^{(1)}(\beta)}{(\alpha - \beta)^2}\right) d\beta$$ and $$\nht_2 f = \frac{1}{\pi i} \pv \int f(\beta)\left(\frac{\zeta^{(2)}_\beta(\beta)}{\alpha - \beta} - \frac{\zeta^{(2)}(\alpha) - \zeta^{(2)}(\beta)}{(\alpha - \beta)^2}\right) d\beta$$ $$-\frac{1}{\pi i}\int f(\beta) \left(\frac{\zeta^{(1)}(\alpha) - \zeta^{(1)}(\beta)}{\alpha - \beta}\right)\left(\frac{\zeta^{(1)}_\beta(\beta)}{\alpha - \beta} - \frac{\zeta^{(1)}(\alpha) - \zeta^{(1)}(\beta)}{(\alpha - \beta)^2}\right) d\beta$$  Now we repeatedly apply the identity $$\frac{1}{\tilde{\zeta}(\alpha) - \tilde{\zeta}(\beta)} = \frac{1}{\alpha - \beta} - \frac{\tilde{\xi}(\alpha) - \tilde{\xi}(\beta)}{(\alpha - \beta)\left(\tilde{\zeta}(\alpha) - \tilde{\zeta}(\beta)\right)}$$ so as to arrive at the identity
\begin{equation}
\frac{1}{\tilde{\zeta}(\alpha) - \tilde{\zeta}(\beta)} =  \frac{1}{\alpha - \beta} - \frac{\tilde{\xi}(\alpha) - \tilde{\xi}(\beta)}{(\alpha - \beta)^2} + \frac{\left(\tilde{\xi}(\alpha) - \tilde{\xi}(\beta)\right)^2}{(\alpha - \beta)^3} - \frac{\left(\tilde{\xi}(\alpha) - \tilde{\xi}(\beta)\right)^3}{(\alpha - \beta)^3\left(\tilde{\zeta}(\alpha) - \tilde{\zeta}(\beta)\right)}
\end{equation}  The last of these terms is of size $O(\epsilon^3)$.  As for the rest, if we arrange $\tilde{\zeta}_\beta(\beta)/\left(\tilde{\zeta}(\alpha) - \tilde{\zeta}(\beta)\right)$ in powers of $\epsilon$ up through $\epsilon^2$, we see that
\begin{eqnarray*}
\frac{\tilde{\zeta}_\beta(\beta)}{\tilde{\zeta}(\alpha) - \tilde{\zeta}(\beta)} & = & \frac{1}{\alpha - \beta} \cr
& & \eqnlineup +\; \epsilon \left(\frac{\zeta^{(1)}_\beta(\beta)}{\alpha - \beta} - \frac{\zeta^{(1)}(\alpha) - \zeta^{(1)}(\beta)}{(\alpha - \beta)^2}\right) \cr
& & \eqnlineup +\; \epsilon^2 \Biggl(\frac{\zeta^{(2)}_\beta(\beta)}{\alpha - \beta} - \frac{\zeta^{(2)}(\alpha) - \zeta^{(2)}(\beta)}{(\alpha - \beta)^2} \cr 
& & \quad -\; \frac{\zeta^{(1)}(\alpha) - \zeta^{(1)}(\beta)}{\alpha - \beta}\left(\frac{\zeta^{(1)}_\beta(\beta)}{\alpha - \beta} - \frac{\zeta^{(1)}(\alpha) - \zeta^{(1)}(\beta)}{(\alpha - \beta)^2}\right)\Biggr) \cr
& & \eqnlineup + \; O(\epsilon^3)
\end{eqnarray*}  All of the terms here up through order $O(\epsilon^2)$ precisely comprise $\tilde{\nht}$, and so vanish upon subtracting $\tilde{\nht}$.  The remaining $O(\epsilon^3)$ terms consists of a finite number of terms which can be written explicitly in the form $$\sum_S C_{p_1, p_2} \epsilon^{n_1 p_1 + n_2 p_2 + m} \frac{\left(\zeta^{(n_1)}(\alpha) - \zeta^{(n_1)}(\beta)\right)^{p_1}\left(\zeta^{(n_2)}(\alpha) - \zeta^{(n_2)}(\beta)\right)^{p_2}}{(\alpha - \beta)^{p_1+p_2 + 1}}\zeta^{(m)}_\beta(\beta)$$ where $S = \{(n_1, n_2, m, p_1, p_2) : n_1p_1 + n_2p_2 + m \geq 3,\, 0 \leq p_1 + p_2 \leq 2,\, 0 \leq n_1,\,n_2,\, m \leq 3\}$ and  $C_{p_1, p_2}$ are constants depend only on $p_1, p_2$.
\end{proof}

\subsection{The Action of $\nht_0$ on Multiscale Functions}

As we saw in the last section, the operators appearing in the power series expansion of the Hilbert Transform of the interface can be written in terms of the flat Hilbert transform $$\nht_0 f := \frac{1}{\pi i} \pv \int \frac{f(\beta)}{\alpha - \beta} \, d\beta$$  It is known that $\nht_0$ is a Fourier multiplier with Fourier symbol $\hat{\nht}_0(\xi) = -\sgn(\xi)$.  However, it still remains to be seen how to interpret the action of $\nht_0$ on a multiscale function $f = f(\alpha, \epsilon\alpha)$ as a multiscale function.

Since we are interested in the modulation approximation of the water wave problem, we will choose the leading order of our approximation to be a wave packet of the form $B(\epsilon\alpha)e^{ik\alpha}$ for $k > 0$.  Hence the formal calculation depends upon understanding the action of $\nht_0$ on such wave packets.  Since the amplitude of $B(\epsilon\alpha)e^{ik\alpha}$ is slowly varying for small $\epsilon$, we heuristically expect for $k \neq 0$ that $$\nhtb_0\left(B(\epsilon\alpha)e^{ik\alpha}\right) \sim B(\epsilon\alpha)\nhtb_0\left(e^{ik\alpha}\right) = B(\epsilon\alpha)\sgn(k)e^{ik\alpha},$$ where $\sim$ indicates an error depending on $\epsilon$.  The following result confirms this intuition.  We adopt the usual practice of assuming, unless otherwise stated, that a constant $C$ may denote different constants in the process of deriving an inequality.

\begin{proposition}\label{WavePacketAntiholProp}
Let $k \neq 0$ and $s, m \geq 0$ be given.  Assume $\epsilon\le 1$. Then if $f \in H^{s + m}$,  $$\|(\nhtb_0 - \sgn(k))f(\epsilon\alpha)e^{ik\alpha}\|_{H^s} \leq C \frac{\epsilon^{m - 1/2}}{k^m}\|f\|_{H^{s + m}}$$ where the constant depends only on $s$.
\end{proposition}

\textbf{Proof.}  It suffices to consider the case $k > 0$, since the case $k < 0$ follows by complex conjugation and the fact that $\nhtb_0 = -\nht_0$.  We first derive a bound for $\|\partial_\alpha^n(I - \nhtb_0)f(\epsilon\alpha)e^{ik\alpha}\|_{L^2}$.  We calculate that
\begin{eqnarray*}
\|\partial_\alpha^n(I - \nhtb_0)f(\epsilon\alpha)e^{ik\alpha}\|_{L^2} & = & \left( \int_{-\infty}^\infty \left|(i\xi)^n (1 - \sgn(\xi)) \frac{1}{\epsilon}\hat{f}\biggl(\frac{\xi - k}{\epsilon}\biggr)\right|^2 d\xi \right)^{1/2}\cr
& = & 2\left( \int_{-\infty}^{-k} \left|(\xi + k)^n \frac{1}{\epsilon}\hat{f}\left(\frac{\xi}{\epsilon}\right)\right|^2 d\xi \right)^{1/2}\cr
& \leq & 2\left( \int_{-\infty}^{-k} \epsilon^{2(n + m) - 1}|\xi|^{-2m}\left|\widehat{\partial_\alpha^{n + m}f}\left(\frac{\xi}{\epsilon}\right)\right|^2 \frac{d\xi}{\epsilon} \right)^{1/2} \cr
& \leq & 2\epsilon^{n + m - 1/2} \left(\sup_{\xi \leq -k} |\xi|^{-m}\right) \left(\int\left|\widehat{\partial_\alpha^{n + m}f}\left(\frac{\xi}{\epsilon}\right)\right|^2\frac{d\xi}{\epsilon}\right)^{1/2} \cr
& \leq & 2\frac{\epsilon^{n + m - 1/2}}{k^m}\|\partial_\alpha^{n + m}f\|_{L^2}.
\end{eqnarray*}  But since $\epsilon \leq 1$, we have for any $m \geq 0$ that
\begin{eqnarray*}
\|(I - \nhtb_0)f(\epsilon\alpha)e^{ik\alpha}\|_{H^s} & \leq & C \sum_{n = 0}^s \|\partial_\alpha^n(I - \nhtb_0)f(\epsilon\alpha)e^{ik\alpha}\|_{L^2} \cr
& \leq & C \sum_{n = 0}^s \frac{\epsilon^{n + m - 1/2}}{k^m}\|\partial_\alpha^{n + m}f\|_{L^2} \cr
& \leq & C \frac{\epsilon^{m - 1/2}}{k^m}\|\partial_\alpha^mf\|_{H^s} \cr
& \leq & C \frac{\epsilon^{m - 1/2}}{k^m}\|f\|_{H^{s + m}}.\Box
\end{eqnarray*}

As a consequence we may freely assume in the multiscale calculation that $\nht_0$ formally treats the amplitude of the wave packet $B(\epsilon\alpha)e^{ik\alpha}$ as a constant when $k \neq 0$.  However, note that in the case $k = 0$ we can at best say that $$\nhtb_0(f(\epsilon \cdot))(\alpha) = (\nhtb_0 f)(\epsilon\alpha)$$ and so these must be retained as functions of the slow variable $\alpha_1 = \epsilon\alpha$ whenever they occur in the multiscale calculation.

We record an immediate consequence of this result that will be used frequently in the multiscale calculation.

\begin{corollary}\label{CommutatorPhase}
Let $s\ge 1, m \geq 0$, $\epsilon\le 1$ and $f, g \in H^{s + m}(\mathbb{R})$ and suppose that $k, l$ are given so that $l \neq 0, -k$, and $sgn (l)=sgn (k+l)$. Then $$\|[f(\epsilon\alpha)e^{ik\alpha}, \nht_0]g(\epsilon\alpha)e^{il\alpha}\|_{H^s} \leq C\epsilon^{m - 1/2}\left(\frac{1}{(k + l)^m} + \frac{1}{k^m}\right)\|f\|_{H^{s + m}}\|g\|_{H^{s + m}}$$
\end{corollary}

\subsection{The Multiscale Calculation}

We are now prepared to find an approximate solution $\tilde{\zeta}$ to the four equations \eqref{NewEuler}-\eqref{DtXiIsAntihol} which is to leading order a wave packet, where $G$ is given by \eqref{GFormula}.  Our approach will be to derive an approximate solution to the system \eqref{NewEuler}-\eqref{XiIsAntihol} having residual $O(\epsilon^4)$ with a multiscale analysis and then verify that this approximate solution also satisfies \eqref{DtNewEuler}-\eqref{DtXiIsAntihol} up to a residual of size $O(\epsilon^4)$.  We begin by seeking a perturbative ansatz for \eqref{NewEuler}-\eqref{XiIsAntihol} $$\zeta(\alpha, t) = \alpha + \sum_{n = 1}^\infty \epsilon^n \zeta^{(n)}(\alpha, t, \epsilon)$$  In order to construct an expansion that is valid on times on the order $O(\epsilon^{-2})$, we introduce multiple scales $$t_0 = t, \quad t_1 = \epsilon t, \quad t_2 = \epsilon^2 t, \quad \alpha_0 = \alpha, \quad \alpha_1 = \epsilon \alpha$$ and so we seek a solution of the form  $$\zeta(\alpha, t) = \alpha + \sum_{n = 1}^\infty \epsilon^n \zeta^{(n)}(\alpha_0, \alpha_1, t_0, t_1,  t_2)$$ which formally satisfies the original equations up to terms of size $O(\epsilon^4)$.

Before we begin solving these equations, we expand the auxiliary quantities and operators in powers of $\epsilon$.  In particular we must determine the expansions in $\epsilon$ of the quantities
$$b = \sum_{n = 0}^\infty \epsilon^n b_n, \qquad \mathcal{A} = \sum_{n = 0}^\infty \epsilon^n \mathcal{A}_n, \qquad G = \sum_{n = 0}^\infty \epsilon^n G_n$$ Notice that since $b$ and $\mathcal{A} - 1$ are of quadratic order and $G$ is of cubic order, it follows that that $$b_0 = b_1 = \mathcal{A}_0= \mathcal{A}_1 = G_0 = G_1 = G_2 = 0.$$  We will also show in the sequel that $\mathcal{A}_2 = 0$ and $b_2 = b_2(\alpha_1, t_1, t_2)$; the linear operator associated to the water wave equation thus has the multiscale expansion
$$D_t^2 - i\mathcal{A}\partial_\alpha = (\partial_{t_0}^2 - i\partial_{\alpha_0}) + \epsilon(2\partial_{t_0}\partial_{t_1} - i\partial_{\alpha_1})$$
\begin{equation}\label{PowerExpansionOfP} \qquad\qquad +\; \epsilon^2(2\partial_{t_0}\partial_{t_2} + \partial_{t_1}^2 + 2b_2 \partial_{t_0}\partial_{\alpha_0}) + O(\epsilon^3)\end{equation}  Recall that we also have the formulas for the multiscale expansion of the Hilbert Transform given by \eqref{MultiscaleHilbertFormulas}

We will find explicit formulas for $b_2$, $b_3$ and $G_3$ in the course of the analysis.  In what follows we will repeatedly use the fact, justified by the last section, that \begin{equation}\label{WavePacketIsAlmostHol}
\nhtb_0(f(\alpha_1)e^{ik\alpha_0}) = \sgn(k)f(\alpha_1)e^{ik\alpha_0} + O(\epsilon^4), \qquad k \neq 0\end{equation} and hence that
\begin{equation}\label{CommutatorPhaseIdentity}
[f(\alpha_1)e^{ik\alpha_0}, \nhtb_0]g(\alpha_1)e^{il\alpha_0} = O(\epsilon^4) \quad \text{ whenever } \quad \sgn(l) = \sgn(k + l), \ l, l+k\neq 0\end{equation}

We are now ready to expand \eqref{NewEuler}-\eqref{XiIsAntihol} in powers of $\epsilon$.  Collecting like terms yields a hierarchy of systems that allow us to successively solve for the holomorphic trace $\frac{1}{2}(I + \nht_0)\zeta^{(n)}$ and the antiholomorphic trace $\frac{1}{2}(I - \nht_0)\zeta^{(n)}$ of the $\zeta^{(n)}$'s in the lower half plane.  The terms of order $O(\epsilon)$ in \eqref{NewEuler}-\eqref{XiIsAntihol} yield the system
\begin{equation}\label{Xi1NewEuler}(\partial_{t_0}^2 - i\partial_{\alpha_0})(I - \nht_0)\zeta^{(1)} = 0 \end{equation} \begin{equation}\label{Xi1IsAntihol}(I - \nhtb_0)\zeta^{(1)} = 0 \end{equation}  Because we are interested in solutions which to leading order are given by wave packets, we assume an ansatz concentrated in Fourier space about the fixed wave number $k > 0$: $$\zeta^{(1)} = B_+(\alpha_1, t_0, t_1, t_2) e^{ik\alpha} + B_-(\alpha_1, t_0, t_1, t_2)e^{-ik\alpha}$$  Injecting the above ansatz into \eqref{Xi1IsAntihol} forces $B_- = 0$ by \eqref{WavePacketIsAlmostHol}.  Similarly substituting this ansatz into \eqref{Xi1NewEuler} yields the condition $(\partial_{t_0}^2 + k)B_+ = 0$, which implies that $B_+(\alpha_1, t_0, t_1, t_2) = B(\alpha_1, t_1, t_2)e^{i\omega t_0}$, where we have introduced the wave frequency $\omega$ which satisfies the water wave dispersion relation \begin{equation}\label{WWDispersionRelation}\omega^2 = k\end{equation} Thus we take as our solution \begin{equation}\label{Zeta1Formula}\zeta^{(1)} = B(\alpha_1, t_1, t_2)e^{i\phi},\end{equation} where we have introduced the phase $\phi := k\alpha_0 + \omega t_0$.

Moving to the $O(\epsilon^2)$ terms from \eqref{NewEuler}, we have by \eqref{Zeta1Formula} and using \eqref{CommutatorPhaseIdentity} that
\begin{align}\label{Xi2NewEuler}
(\partial_{t_0}^2 - i\partial_{\alpha_0})(I - \nht_0)\zeta^{(2)} & = -(\partial_{t_0}^2 - i\partial_{\alpha_0})(-\nht^{(1)})\zeta^{(1)} \notag\\
& \quad -(2\partial_{t_0}\partial_{t_1} - i\partial_{\alpha_1})(I - \nht_0)\zeta^{(1)} \\
& = -4i\omega(B_{t_1} - \omega^\prime B_{\alpha_1})e^{i\phi} \notag 
\end{align} 
where $\omega^\prime = d\omega/dk$ is the group velocity of the wave packet.  If we want $(I - \nht_0)\zeta^{(2)}$ to be uniformly bounded for all time we must insist that the right hand side of \eqref{Xi2NewEuler} be equal to zero in order to avoid secular terms.  Therefore we choose \begin{equation}\label{BTravelsAtGroupVelocity}B(\alpha_1, t_1, t_2) = B(\alpha_1 + \omega^\prime t_1, t_2) := B(X, T)\end{equation} where $\omega^\prime = d\omega/dk$ is the group velocity.  The $O(\epsilon^2)$ terms from \eqref{XiIsAntihol} yield the equation
\begin{align}\label{Zeta2XiIsAntihol}
(I - \nhtb_0)\zeta^{(2)} & = \nhtb^{(1)} \zeta^{(1)} \notag\\
& = [\overline{\zeta}^{(1)}, \nhtb_0]\zeta^{(1)}_{\alpha_0} \\
& = ik(I - \nhtb_0)|B|^2 \notag
\end{align}  
An obvious choice seems to be $\zeta^{(2)}=ik|B|^2+B_2(\alpha_1,t_1,t_2) e^{i\phi}$. However such choice leads to unavoidable secular growth in the $O(\epsilon^3)$ level. Instead, we find that taking $\zeta^{(2)}$ so that $(I - \nht_0)\zeta^{(2)} = 0$ avoids such secular growth.  
Hence we take \begin{equation}\label{Zeta2Formula}\zeta^{(2)} = \frac{1}{2}ik(I - \nhtb_0)|B|^2\end{equation}

Before we move on to the $O(\epsilon^3)$ system, we must first derive formulas for $b_2$, $\mathcal{A}_2$ and $G_3$.  Substituting the expansion of $\zeta$ into the formula \eqref{bFormula} we see immediately upon collecting like powers of $\epsilon$ that $b_0 = b_1 = 0$.  Therefore we have
$$(I - \nht_0)b_2 = -[\zeta^{(1)}_{t_0}, \nht_0]\overline{\zeta}^{(1)}_{\alpha_0} = -k\omega(I - \nht_0)|B|^2,$$ and so since $b_2$ is real-valued we conclude that \begin{equation}\label{b2Formula}b_2 = -k\omega|B|^2\end{equation} Similarly, using \eqref{AFormula} we have immediately that $\mathcal{A}_1 = 0$ and that
$$(I - \nht_0)\mathcal{A}_2 = i[\partial_{t_0} \zeta^{(1)}_{t_0}, \nht_0]\overline{\zeta}^{(1)}_{\alpha_0} + i[\zeta^{(1)}_{t_0}, \nht_0]\partial_{t_0}\overline{\zeta}^{(1)}_{\alpha_0} = -ik\omega\partial_{t_0}(I - \nht_0)|B|^2 = 0$$ whence $\mathcal{A}_2 = 0$ as claimed.

Finally we derive from \eqref{GFormula} a formula for $G_3$:
\begin{align*}
G_3 & = \frac{4}{\pi}\int\frac{\left(\zeta^{(1)}_{t_0}(\alpha) - \zeta^{(1)}_{t_0}(\beta)\right)\left(\Im \zeta^{(1)}(\alpha) - \Im \zeta^{(1)}(\beta)\right)}{(\alpha - \beta)^2} \zeta^{(1)}_{t_0 \beta_0}(\beta) \, d\beta \notag\\
& \quad + \frac{2}{\pi}\int\frac{(\zeta^{(1)}_{t_0}(\alpha) - \zeta^{(1)}_{t_0}(\beta))^2}{(\alpha - \beta)^2} \Im \zeta^{(1)}_{\beta_0} (\beta) d\beta \notag\\
& := I_1 + I_2
\end{align*}  Using \eqref{WavePacketIsAlmostHol} and \eqref{CommutatorPhaseIdentity} yields
\begin{align*}
I_1 & = -\frac{2}{\pi i}\int\frac{\left(\zeta^{(1)}_{t_0}(\alpha) - \zeta^{(1)}_{t_0}(\beta)\right)\left(\overline{\zeta}^{(1)}(\alpha) - \overline{\zeta}^{(1)}(\beta)\right)}{(\alpha - \beta)^2} \zeta^{(1)}_{t_0 \beta_0}(\beta) \, d\beta \\
& = \frac{2}{\pi i} \int \frac{\left(\zeta^{(1)}_{t_0}(\alpha) - \zeta^{(1)}_{t_0}(\beta)\right)\left(\overline{\zeta}^{(1)}(\alpha) - \overline{\zeta}^{(1)}(\beta)\right)}{(\alpha - \beta)} \zeta^{(1)}_{t_0 \beta_0 \beta_0}(\beta) \, d\beta \\
& \quad - \frac{2}{\pi i} \int \frac{\left(\zeta^{(1)}_{t_0}(\alpha) - \zeta^{(1)}_{t_0}(\beta)\right)\overline{\zeta}^{(1)}_{\beta_0}(\beta)}{(\alpha - \beta)} \zeta^{(1)}_{t_0 \beta_0}(\beta) \, d\beta \notag\\
& \quad -\frac{2}{\pi i} \int \frac{\zeta^{(1)}_{t_0 \beta_0}(\beta)\left(\overline{\zeta}^{(1)}(\alpha) - \overline{\zeta}^{(1)}(\beta)\right)}{(\alpha - \beta)} \zeta^{(1)}_{t_0 \beta_0}(\beta) \, d\beta \\
& = \frac{2k^3}{\pi i}\overline{\zeta}^{(1)}(\alpha) \int \frac{\left(\zeta^{(1)}(\alpha) - \zeta^{(1)}(\beta)\right)}{(\alpha - \beta)} \zeta^{(1)}(\beta) \, d\beta \\
& \quad -2k^3[\overline{\zeta}^{(1)}, \nht_0]\left((\zeta^{(1)})^2\right) \\
& = 0
\end{align*}  Similarly, we simplify
\begin{align*}
I_2 & = \frac{2}{\pi}\int\frac{(\zeta^{(1)}_{t_0}(\alpha) - \zeta^{(1)}_{t_0}(\beta))^2}{(\alpha - \beta)^2} \Im \zeta^{(1)}_{\beta_0} (\beta) d\beta \\
& = 2i\biggl(2[\zeta^{(1)}_{t_0}, \nht_0](\zeta^{(1)}_{t_0 \alpha_0}\Im\zeta^{(1)}_{\alpha_0}) - [\zeta^{(1)}_{t_0}, [\zeta^{(1)}_{t_0}, \nht_0]]\Im\zeta^{(1)}_{\alpha_0 \alpha_0}\biggr) \\
& = - 2[\zeta^{(1)}_{t_0}, \nht_0](\zeta^{(1)}_{t_0 \alpha_0}\overline{\zeta}^{(1)}_{\alpha_0}) + [\zeta^{(1)}_{t_0}, [\zeta^{(1)}_{t_0}, \nht_0]]\overline{\zeta}^{(1)}_{\alpha_0 \alpha_0} \\
& = 2k^3Be^{i\phi}(I + \nht_0)|B|^2 - 2k^3Be^{i\phi}\nht_0|B|^2 \\
& = 2k^3B|B|^2e^{i\phi}
\end{align*}  In summary,
\begin{equation}\label{G3Formula}
G_3 = 2k^3B|B|^2e^{i\phi}
\end{equation}

We can now arrange the $O(\epsilon^3)$ terms of \eqref{NewEuler}, and using \eqref{Zeta1Formula} and \eqref{Zeta2Formula} along with \eqref{b2Formula}, \eqref{G3Formula} and \eqref{CommutatorPhaseIdentity} arrive at the equation
\begin{align}\label{Zeta3NewEuler}
(\partial_{t_0}^2 - i\partial_{\alpha_0})(I - \nht_0)\zeta^{(3)} & = -(\partial_{t_0}^2 - i\partial_{\alpha_0})(-\nht^{(1)})\zeta^{(2)} - (\partial_{t_0}^2 - i\partial_{\alpha_0})(-\nht^{(2)})\zeta^{(1)} \cr
& \quad -\;(2\partial_{t_0}\partial_{t_1} - i\partial_{\alpha_1})(I - \nht_0)\zeta^{(2)} - (2\partial_{t_0}\partial_{t_1} - i\partial_{\alpha_1})(-\nht^{(1)})\zeta^{(1)} \cr
& \quad -\; (2\partial_{t_0}\partial_{t_2} + \partial_{t_1}^2 + 2b_2\partial_{t_0}\partial_{\alpha_0})(I - \nht_0)\zeta^{(1)} + G_3 \cr
& = - (2\partial_{t_0}\partial_{t_2} + \partial_{t_1}^2 + 2b_2\partial_{t_0}\partial_{\alpha_0})(I - \nht_0)\zeta^{(1)} \cr
& \quad +\; 2k^3B|B|^2e^{i\phi} \cr
& = - 2\omega(2iB_T - \omega^{\prime \prime}B_{XX} + k^2\omega B|B|^2)e^{i\phi}
\end{align} where $\omega^{\prime\prime} = d^2\omega/dk^2$.  To supress secular growth we now insist that the amplitude $B$ satisfy the focusing cubic nonlinear Schr\"odinger equation\footnote{Observe that this equation agrees with the equation derived in \cite{CraigSulemSulemNLSFiniteDepth} when one formally lets the depth of the fluid tend to infinity.} \begin{equation}\label{NLS}2iB_T - \omega^{\prime \prime}B_{XX} + k^2\omega B|B|^2 = 0,\end{equation}  With this choice made we solve \eqref{Zeta3NewEuler} by taking $(I - \nht_0)\zeta^{(3)} = 0$.

Finally, the $O(\epsilon^3)$ terms from \eqref{XiIsAntihol} yields the equation
\begin{align}\label{Zeta3XiIsAntihol}
(I - \nhtb_0)\zeta^{(3)} & = \nhtb^{(1)} \zeta^{(2)} + \nhtb^{(2)} \zeta^{(1)} \notag\\
& = [\overline{\zeta}^{(1)}, \nhtb_0]\zeta^{(2)}_{\alpha_0} + [\overline{\zeta}^{(2)}, \nhtb_0]\zeta^{(1)}_{\alpha_0} + [\overline{\zeta}^{(1)}, \nhtb_0]\zeta^{(1)}_{\alpha_1} \\
& \quad -\; [\overline{\zeta}^{(1)}, \nhtb_0](\overline{\zeta}_{\alpha_0}^{(1)}\zeta^{(1)}_{\alpha_0}) + \frac{1}{2}[\overline{\zeta}^{(1)}, [\overline{\zeta}^{(1)}, \nhtb_0]]\zeta^{(1)}_{\alpha_0 \alpha_0} \notag\\
& = (I - \nhtb_0)(\overline{B}B_X) - k^2\overline{B}e^{-i\phi}(I + \nhtb_0)|B|^2 + k^2\overline{B}e^{-i\phi}\nhtb_0|B|^2 \notag\\
& = -k^2\overline{B}|B|^2e^{-i\phi} + (I - \nhtb_0)\left( \overline{B}B_X\right) \notag
\end{align}  Hence we choose \begin{equation}\label{Zeta3Formula} \zeta^{(3)} = -\frac{1}{2}k^2\overline{B}|B|^2e^{-i\phi} + \frac{1}{2}(I - \nhtb_0)\left(\overline{B}B_X\right)\end{equation}

Now that we have constructed an approximate solution $\tilde{\zeta}$ to the equations \eqref{NewEuler}-\eqref{XiIsAntihol},  we claim that $\tilde{\zeta}$ also solves the system \eqref{DtNewEuler}-\eqref{DtXiIsAntihol} up to an $O(\epsilon^4)$ residual.  First we notice that \eqref{DtNewEuler} is obtained by applying a derivative $D_t$ to  \eqref{NewEuler}, therefore it is clear that $\tilde{\zeta}$ solves 
 \eqref{DtNewEuler} up to an $O(\epsilon^4)$ residual.  Now we consider \eqref{DtXiIsAntihol}.  By \eqref{XiIsAntihol} we have that
\begin{align*}
(I - \nhtb)D_t\zeta & = (I - \nhtb)D_t(\zeta - \alpha) + (I - \nhtb)D_t\alpha \\
& = [D_t, \nhtb](\zeta - \alpha) + (I - \nhtb)D_t\alpha \\
& = [D_t\zeta, \nhtb]\frac{\zeta_\alpha - 1}{\overline{\zeta}_\alpha} + (I - \nhtb)b
\end{align*}  Hence to show that $\tilde{\zeta}$ satisfies \eqref{DtXiIsAntihol} up to an $O(\epsilon^4)$ residual 
 it suffices to show that our approximation of $b$ satisfies \eqref{bFormula} up to a residual of size $O(\epsilon^4)$.  Hence we need only choose $b_3$ so that
\begin{align}\label{HolPartOfb3}
(I - \nhtb_0)b_3 & = \nhtb^{(1)} b_2 \notag\\
& \quad -\; [\partial_{t_0}\overline{\zeta}^{(2)}, \nhtb_0]\zeta^{(1)}_{\alpha_0} - [\partial_{t_1}\overline{\zeta}^{(1)}, \nhtb_0]\zeta^{(1)}_{\alpha_0} \notag\\
& \quad -\; [\partial_{t_0}\overline{\zeta}^{(1)}, \nhtb^{(1)}]\zeta^{(1)}_{\alpha_0} - [\partial_{t_0}\overline{\zeta}^{(1)}, \nhtb_0]\zeta^{(2)}_{\alpha_0} \notag\\
& \quad -\; [\partial_{t_0}\overline{\zeta}^{(1)}, \nhtb_0]\zeta^{(1)}_{\alpha_1} + [\partial_{t_0}\overline{\zeta}^{(1)}, \nhtb_0]|\zeta^{(1)}_{\alpha_0}|^2 \notag\\
& \notag \\
& = -\frac{1}{2}i\omega(I - \nhtb_0)(B\overline{B}_X) \notag\\
& \quad -\; i\omega k^2 \overline{B}e^{-i\phi}(I - \nhtb_0)|B|^2 \notag\\
& \quad +\; i\omega(I - \nhtb_0)(\overline{B}B_X) \notag\\
& \quad -\; i\omega k^2 \overline{B}e^{-i\phi}(I + \nhtb_0)|B|^2 \notag \\
& \notag\\
& = i\omega(I - \nhtb_0)\biggl(B\overline{B}_X - \frac{1}{2}\overline{B}B_X \biggr) -2i\omega k^2 \overline{B}|B|^2e^{-i\phi}
\end{align} 
In summary, we have shown that the equations \eqref{NewEuler}-\eqref{XiIsAntihol}-\eqref{DtNewEuler}-\eqref{DtXiIsAntihol} are satisfied up to a residual of size $O(\epsilon^4)$ by the approximation
\begin{align}\label{TildeZetaFormula}
\tilde{\zeta} & := \alpha + \epsilon \zeta^{(1)} + \epsilon^2 \zeta^{(2)} + \epsilon^3 \zeta^{(3)} \notag\\
& = \alpha + \epsilon Be^{i\phi} + \epsilon^2\frac{1}{2}ik(I - \nhtb_0)|B|^2 \notag \\
& \qquad + \epsilon^3 \left(-\frac{1}{2}k^2\overline{B}|B|^2e^{-i\phi} + \frac{1}{2}(I - \nhtb_0)\left(\overline{B}B_X\right)\right)
\end{align} where $B = B(\epsilon(\alpha + \omega^\prime t), \epsilon^2 t) = B(X, T)$ satisfies the NLS equation $$2iB_T - \omega^{\prime\prime}B_{XX} + k^2\omega B|B|^2 = 0$$  From \eqref{HolPartOfb3}, enforcing the reality condition on $b_3$ yields 
\begin{align}\label{TildeBFormula} 
\tilde{b} & := b_0 + \epsilon b_1 + \epsilon^2 b_2 + \epsilon^3 b_3 \notag\\
& = \epsilon^2(-k\omega|B|^2) \notag \\
& + \epsilon^3\biggl(\Re\left(2i\omega k^2 B|B|^2e^{i\phi}\right) + \frac{3}{4}i\omega(B\overline{B}_X - \overline{B}B_X) - \frac{1}{4}i\omega\nhtb_0(B\overline{B}_X + \overline{B}B_X)\biggr)\end{align}
We also define
\begin{equation}\label{TildeAFormula}\tilde{\mathcal{A}} := \mathcal{A}_0 + \epsilon \mathcal{A}_1 + \epsilon^2 \mathcal{A}_2 = 1\end{equation}
and
\begin{equation}
\tilde G := G_0 + \epsilon G_1 + \epsilon^2 G_2 + \epsilon^3 G_3 = \epsilon^3 G_3
\end{equation}
Corresponding to this approximate solution \eqref{TildeZetaFormula} we introduce $$\tilde{\xi} := \tilde{\zeta} - \alpha$$ as well as
\begin{equation}\label{TildeDtFormulas}
\tilde{D}_t := \partial_t + \tilde{b}\partial_\alpha \qquad \tilde{\mathcal{P}} := \tilde{D}_t^2 - i\tilde{\mathcal{A}}\partial_\alpha
\end{equation}
We then have the formulas for the difference
\begin{equation}\label{DiffDtFormula}
D_t - \tilde{D}_t = (b - \tilde{b})\partial_\alpha
\end{equation}
as well as for
\begin{equation}\label{DiffDtSquaredFormula}
D_t^2 - \tilde{D}_t^2 = \left(D_t(b - \tilde{b})\right)\partial_\alpha + (b - \tilde{b})\left(D_t \partial_\alpha + \partial_\alpha \tilde{D}_t\right)
\end{equation}
and so
\begin{equation}\label{DiffPFormula}
\mathcal{P} - \tilde{\mathcal{P}} = \left(D_t(b - \tilde{b}) - i(\mathcal{A} - \tilde{\mathcal{A}})\right)\partial_\alpha + (b - \tilde{b})\left(D_t \partial_\alpha + \partial_\alpha \tilde{D}_t\right)
\end{equation}

For future reference we also include the following calculation.

\begin{proposition}\label{VariousMultiscaleIdentities}
Let $\tilde{\zeta}$, $\tilde{b}$, $\tilde{\mathcal{A}}$ be as above.  Then 
\begin{enumerate}
\item{$\tilde{\mathcal{P}}\tilde{\xi} = O(\epsilon^3)$.}
\item{$[\tilde{\mathcal{P}}, \tilde{\nht}]\tilde{\xi} = O(\epsilon^4)$.}
\end{enumerate}
\end{proposition}

\begin{proof}
The first statement is straightforward.  For the second, observe that by \eqref{CommutatorPhaseIdentity} we have that $\nht^{(1)}\zeta^{(1)} = \nht^{(2)} \zeta^{(1)} = O(\epsilon^4)$.  The $O(\epsilon)$ term is $[\partial_{t_0}^2 - i\partial_{\alpha_0}, \nht_0]\zeta^{(1)} = 0$.  The $O(\epsilon^2)$ terms are $$[\partial_{t_0}^2 - i\partial_{\alpha_0}, \nht_0]\zeta^{(2)} + [\partial_{t_0}^2 - i\partial_{\alpha_0}, \nht^{(1)}]\zeta^{(1)} + [2\partial_{t_0}\partial_{t_1} - i\partial_{\alpha_1}, \nht_0]\zeta^{(1)}$$ which vanishes by virtue of the above observation, \eqref{Zeta1Formula}, \eqref{BTravelsAtGroupVelocity}, and \eqref{Zeta2Formula}.  Finally, the $O(\epsilon^3)$ terms are given by
\begin{align*}
& \quad\; [\partial_{t_0}^2 - i\partial_{\alpha_0}, \nht_0]\zeta^{(3)} \\
& + [\partial_{t_0}^2 - i\partial_{\alpha_0}, \nht^{(1)}]\zeta^{(2)} \\
& + [\partial_{t_0}^2 - i\partial_{\alpha_0}, \nht^{(2)}]\zeta^{(1)} \\
& + [2\partial_{t_0}\partial_{t_1} - i\partial_{\alpha_1}, \nht_0]\zeta^{(2)} \\
& + [2\partial_{t_0}\partial_{t_1} - i\partial_{\alpha_1}, \nht^{(1)}]\zeta^{(1)} \\
& + [2\partial_{t_0}\partial_{t_2} + \partial_{t_1}^2 + 2b_2\partial_{\alpha_0}\partial_{t_0}, \nht_0]\zeta^{(1)}
\end{align*}  For the same reasons as for the $O(\epsilon^2)$ terms all of the above are immediately seen to vanish except for the last, which by \eqref{b2Formula} is given by $$2[b_2, \nht_0]\zeta^{(1)}_{\alpha_0 t_0} = 2k^3[|B|^2, \nht_0]Be^{i\phi} = 0,$$ by \eqref{CommutatorPhaseIdentity}.
\end{proof}

Now we have shown that the approximation $\tilde{\zeta}$ depends on $B$ and $B_X$, where $B$ satisfies the NLS equation \eqref{NLS}.  To be certain that the forthcoming objects are well-defined, we appeal to the following global well-posedness result for NLS:

\begin{theorem}\label{NLSWellPosedness}
(c.f. \cite{CazenaveSemilinearSchrodingerEquations}, \cite{SulemSulemNLSBook})  Let $m \geq 1$ be given, and suppose that $B_0 \in H^m$ is given.  Then there exists a unique solution $B \in C([0, \infty); H^m)$ to \eqref{NLS} with initial condition $B(0) = B_0$.
\end{theorem}

Fix  $s \geq 6$, $\mathscr T > 0$. For the rest of the paper we assume that $B_0 \in H^{s + 7}$, and hence by the above theorem that $B \in C([0, \infty), H^{s + 7})$ with $\|B\|_{C([0, \mathscr T); H^{s + 7})} \leq C(\|B_0\|_{H^{s + 7}}, \mathscr T)$.  If we calculate $\tilde{\zeta}$ and $\tilde{D}_t\tilde{\zeta}$ from $B \in H^{s + 7}$ through \eqref{TildeZetaFormula}, we see by counting the maximum number of derivatives that fall on $B$ that we have the bound 
\begin{equation}\label{NLSGlobalBound}
\left\|\left(\tilde{\xi}, \tilde{D}_t\tilde{\zeta}, \tilde{D}_t^2 \tilde{\zeta}\right)\right\|_{C([0, \mathscr T); H^{s + 6} \times H^{s + 4} \times H^{s + 2})} \leq C(\|B_0\|_{H^{s + 7}}, \mathscr T)\epsilon^{1/2}\end{equation}  For the rest of the paper, we  choose $\epsilon < \epsilon_0$ for $\epsilon_0\le 1$ sufficiently small depending on $B_0$ so that $\tilde{\zeta}$ satisfies the chord-arc condition \eqref{ChordArcCondition}.  Along with the a priori assumption \eqref{ZetaLocalAPrioriBound} using an appropriately small choice of $\delta > 0$ , this implies that the singular integrals in the next section are well-defined.

\section{Estimates of the Remainder}

Now that we have derived a formal approximation of the solution $\zeta$ to the system \eqref{NewEuler}-\eqref{XiIsAntihol}-\eqref{DtNewEuler}-\eqref{DtXiIsAntihol}, we can consider the size of the remainder $r = \zeta - \tilde{\zeta}$.  Our basic approach is to expand the known equations for $\zeta$ and formulas for quantities defined in terms of $\zeta$ given in \S 2 by writing $\zeta = r + \tilde{\zeta}$ and thereby find the appropriate governing equations from which we will derive energy estimates for $r$.

In \S 4.1 we derive from \eqref{NewEuler}-\eqref{XiIsAntihol}-\eqref{DtNewEuler}-\eqref{DtXiIsAntihol} new equations in terms of quantities related to $r$.  Many functions and operators will arise in these equations that we need to study before we can estimate them appropriately.  In particular we devote \S 4.2 to studying the remainder between the true and approximate Hilbert transforms introduced in \S 3.1.  

To clearly describe the respects in which we consider quantities to be small, we adopt the following terminology: we say a term is of $n$th order (with linear, quadratic, cubic having the typical meaning) if the term consists of $n$ small factors.

Alternately, given a Banach space $X$ with norm $\|\cdot\|_X$, we say that a term $f \in X$ as being $O(\epsilon^n)$ in $X$ when there exists a constant $C$ so that $\|f\|_X \leq C\epsilon^n$.  If we use the notation $O(\epsilon^n)$ without mentioning a norm explicitly, we mean size in the physical sense $O(\epsilon^n)$ as we have used in \S 3.  Since we ultimately seek bounds in Sobolev spaces $H^s$, we introduce the special notation that $f \in H^s$ is $\mathcal{O}(\epsilon^n)$, which means that $f$ is $O(\epsilon^n)$ in $H^s$ where the index $s$ will be clear from context.

We ultimately plan to control all of our quantities in terms of $r_\alpha$ and $D_t r$ in Sobolev space, and so we need some idea of how large we expect $r_\alpha$ and $D_tr$ to be in terms of $\epsilon$.  Since we are only interested in the leading term of the approximation, it is a suitable goal to seek a remainder which is of physical size $O(\epsilon^2)$, and in the $L^2$ sense to be $\mathcal{O}(\epsilon^{3/2})$.  Therefore, we expect here that $r_\alpha$ and $D_t r$ should be $\mathcal{O}(\epsilon^{3/2})$.

In \S 4.3 we bound in $H^s$ the remaining quantities appearing in the cubic nonlinearities of the equations of \S 4.1 by terms involving the quantity $$E_s^{1/2} := \|r_\alpha\|_{H^s} + \|D_t r\|_{H^s},$$ which we expect to be $O(\epsilon^3)$.  We will then show that for $\epsilon < \epsilon_0$ with $\epsilon_0$ chosen sufficiently small, the quantity $E_s$  is bounded above by the quantity
$$\sum_{n = 0}^s \|D_t\partial_\alpha^n \rho\|_{L^2}^2 + \|D_t \partial_\alpha^n \sigma\|^2_{L^2}$$ where $$
\rho := \frac{1}{2}(I - \nht)r \qquad \text{and} \qquad \sigma := \frac{1}{4}(I - \nht)\left(D_t(I - \nht)\xi - \tilde{D}_t(I - \tilde{\nht})\tilde{\xi}\right)$$
which in turn is bounded above by the energy $\mathcal {E}$ for the remainder.

We then use these estimates to show that the cubic nonlinearities of the remainder equations of \S 4.1 are $\mathcal{O}(\epsilon^{7/2})$.
Having done so, we derive in \S 4.5 an energy inequality which roughly reads $d\mathcal{E}/dt \leq O(\epsilon^5)$.  Heuristically, an inequality of this type is suitable since on time scales on the order $O(\epsilon^{-2})$ this implies $E_s$ is of size $O(\epsilon^3)$, as we would like.  We then go on to rigorously derive a priori bounds of $E_s$ on $O(\epsilon^{-2})$ time scales.

\subsection{The Derivation of the Equations for the Remainder}

Here we derive the equations governing the evolution of the quantities \begin{equation}\label{RhoSigmaDefn}\rho := \frac{1}{2}(I - \nht)r \qquad \text{and} \qquad \sigma := \frac{1}{4}(I - \nht)\left(D_t(I - \nht)\xi - \tilde{D}_t(I - \tilde{H})\tilde{\xi}\right)\end{equation}  Our goal in this section is to manipulate the nonlinearities of these equations so that they will be in a suitable form for showing they are of size $\mathcal{O}(\epsilon^{7/2})$.  For example, from \eqref{DtXiIsAntihol} we have \begin{equation}\label{DtXiIsAntiholRemainder}
(I - \nht)D_t\overline{r} = -(I - \tilde{\nht})\tilde{D}_t\overline{\tilde{\zeta}} - (I - \tilde{\nht})(D_t - \tilde{D}_t)\overline{\tilde{\zeta}} + (\nht - \tilde{\nht}){D}_t\overline{\tilde{\zeta}}\end{equation}  We will show in \S 4.2 that the operator norm of $\nht - \tilde{\nht}$ on $H^s$ is of size $O(\epsilon^{3/2})$, and in \S 4.3 that the function $b - \tilde{b}$ is of size $\mathcal{O}(\epsilon^{5/2})$.  Hence the right hand side of \eqref{DtXiIsAntiholRemainder} is of size $\mathcal{O}(\epsilon^{5/2})$.  

We now give the equation for the remainder corresponding to \eqref{NewEuler}.  In decomposing the right hand side of this equation, we keep two goals in mind.  First, we must split the terms in such a way as to arrive at $\tilde{G}$ so as to cancel the $O(\epsilon^3)$ contribution from $G$.   Next, we must whenever possible avoid estimating terms formed by $\mathcal{P}$ acting on complicated terms, so as to reduce all estimates whenever possible to those already derived.  Specifically we expand using Proposition \ref{HilbertCommutatorIdentities} as follows:
\begin{align*}
\mathcal{P}(I - \nht)r & = G - \mathcal{P}(I - \nht)\tilde{\xi} \\
& = G + [\mathcal{P}, \nht]\tilde{\xi} - (I - \nht)\mathcal{P}\tilde{\xi} \\
& = G + 2[D_t\zeta, \nht]\frac{\partial_\alpha}{\zeta_\alpha}D_t\tilde{\xi} - \frac{1}{\pi i}\int \left(\frac{D_t\zeta(\alpha) - D_t\zeta(\beta)}{\zeta(\alpha) - \zeta(\beta)}\right)^2 \tilde{\xi}_\beta(\beta) \, d\beta \\
& \quad - (I - \nht)(\mathcal{P} - \tilde{\mathcal{P}})\tilde{\xi} - (I - \nht)\tilde{\mathcal{P}}\tilde{\xi} \\
& = G + 2[D_t\zeta, \nht]\frac{\partial_\alpha}{\zeta_\alpha}D_t\tilde{\xi} - \frac{1}{\pi i}\int \left(\frac{D_t\zeta(\alpha) - D_t\zeta(\beta)}{\zeta(\alpha) - \zeta(\beta)}\right)^2 \tilde{\xi}_\beta(\beta) \, d\beta \\
& \quad - (I - \nht)(\mathcal{P} - \tilde{\mathcal{P}})\tilde{\xi} + (\nht - \tilde{\nht})\tilde{\mathcal{P}}\tilde{\xi} - (I - \tilde{\nht})\tilde{\mathcal{P}}\tilde{\xi} \\
& = (G - \tilde{G}) + 2[D_t\zeta, \nht]\frac{\partial_\alpha}{\zeta_\alpha}D_t\tilde{\xi} - \frac{1}{\pi i}\int \left(\frac{D_t\zeta(\alpha) - D_t\zeta(\beta)}{\zeta(\alpha) - \zeta(\beta)}\right)^2 \tilde{\xi}_\beta(\beta) \, d\beta \\
& \quad - (I - \nht)(\mathcal{P} - \tilde{\mathcal{P}})\tilde{\xi} + (\nht - \tilde{\nht})\tilde{\mathcal{P}}\tilde{\xi} - [\tilde{\mathcal{P}}, \tilde{\nht}]\tilde{\xi} + \epsilon^4 R,
\end{align*} where $\epsilon^4 R := \tilde{G} - \tilde{\mathcal{P}}(I - \tilde{\nht})\tilde{\xi}$ is the residual arising from the approximate equation corresponding to \eqref{NewEuler}.  

Note that at most five\footnote{Observe that, despite the appearance of formulas \eqref{MultiscaleHilbertFormulas}, since the operators $\nht_1$ and $\nht_2$ can be written as singular integrals as in \eqref{HilbertFormulas}, they do not lose derivatives due to Proposition \ref{SingIntSobolevEstimates}.} derivatives of $B$ are taken in $R$ through the term $\partial_t^2 \nht_2 \zeta^{(3)}$, and so $R \in H^s$ provided $B \in H^{s + 5}$.  Similarly, at most seven derivatives of $B$ are taken in $\tilde{D}_t R$ through the term $\partial_t^3 \nht^{(2)} \zeta^{(3)}$, and so $\tilde{D}_t R \in H^s$ provided $B \in H^{s + 7}$.

The only term that is not immediately of size $\mathcal{O}(\epsilon^{7/2})$ is $2[D_t\zeta, \nht]\frac{\partial_\alpha}{\zeta_\alpha}D_t\tilde{\xi}$.  As in the calculation (2.13) et. seq. of \cite{WuAlmostGlobal2D}, we exploit the fact that $D_t\tilde{\xi}$ is almost holomorphic.  Using \eqref{DtXiIsAntihol} and Proposition \ref{HoloProperties} allows us to rewrite this term as \begin{align*}
2[D_t\zeta, \nht]\frac{\partial_\alpha}{\zeta_\alpha}D_t\tilde{\xi} & = 2\left[D_t\zeta, \nht\frac{1}{\zeta_\alpha} + \nhtb\frac{1}{\overline{\zeta}_\alpha}\right]\partial_\alpha D_t \tilde{\xi} - 2[D_t\zeta, \nhtb]\frac{\partial_\alpha}{\overline{\zeta}_\alpha}(D_t\zeta - D_t\alpha - D_t r) \\
& = 2\left[D_t\zeta, \nht\frac{1}{\zeta_\alpha} + \nhtb\frac{1}{\overline{\zeta}_\alpha}\right]\partial_\alpha D_t \tilde{\xi} + 2[D_t\zeta, \nhtb]\frac{b_\alpha}{\overline{\zeta}_\alpha} + 2[D_t\zeta, \nhtb]\frac{\partial_\alpha}{\overline{\zeta}_\alpha}D_t r
\end{align*}  To see that the last of these terms is acceptably small, we again apply Proposition \ref{HoloProperties} to see that
 \begin{equation}\label{4.2}
 2[D_t\zeta, \nhtb]\frac{\partial_\alpha}{\overline{\zeta}_\alpha}D_t r = [(I + \nhtb)D_t\zeta, \nhtb]\frac{\partial_\alpha}{\overline{\zeta}_\alpha}D_t r = [D_t\zeta, \nhtb]\frac{\partial_\alpha}{\overline{\zeta}_\alpha}(I - \nhtb)D_t r,
 \end{equation}
  which is now easily seen to be $\mathcal{O}(\epsilon^{7/2})$ by \eqref{DtXiIsAntiholRemainder}.  Thus our equation for $\rho$ is now
\begin{align}\label{NewEulerRemainder}
2\mathcal{P}\rho & = (G - \tilde{G}) - (I - \nht)(\mathcal{P} - \tilde{\mathcal{P}})\tilde{\xi} + (\nht - \tilde{\nht})\tilde{\mathcal{P}}\tilde{\xi} - [\tilde{\mathcal{P}}, \tilde{\nht}]\tilde{\xi} \notag\\
& + 2\left[D_t\zeta, \nht\frac{1}{\zeta_\alpha} + \nhtb\frac{1}{\overline{\zeta}_\alpha}\right]\partial_\alpha D_t \tilde{\xi} + 2[D_t\zeta, \nhtb]\frac{b_\alpha}{\overline{\zeta}_\alpha} + 2[D_t\zeta, \nhtb]\frac{\partial_\alpha D_t r}{\overline{\zeta}_\alpha} \notag\\
& - \frac{1}{\pi i}\int \left(\frac{D_t\zeta(\alpha) - D_t\zeta(\beta)}{\zeta(\alpha) - \zeta(\beta)}\right)^2 \tilde{\xi}_\beta(\beta) \, d\beta+\epsilon^4 R
\end{align}
Note that the terms on the right hand side of \eqref{NewEulerRemainder} are cubic, and so a priori there may be contributions of size $\mathcal{O}(\epsilon^{5/2})$.  However, we will show later that all such contributions arise as terms depending only on $\tilde{\xi}$ and $\epsilon$ of physical size $O(\epsilon^3)$; moreover, these putative terms will be shown to vanish by multiscale calculations.

Next we derive the evolution equation for $\sigma$.  First we calculate that
\begin{align*}
\mathcal{P}(I - \nht)D_t(I - \nht)\xi & = -[\mathcal{P}, \nht]D_t(I - \nht)\xi + (I - \nht)\mathcal{P}D_t(I - \nht)\xi \\
& = -2[D_t\zeta, \nht]\frac{\partial_\alpha D_t^2(I - \nht)\xi}{\zeta_\alpha} \\
& \quad + \frac{1}{\pi i} \int \left(\frac{D_t\zeta(\alpha) - D_t\zeta(\beta)}{\zeta(\alpha) - \zeta(\beta)}\right)^2 \partial_\beta D_t(I - \nht)\xi(\beta) d\beta \\
& \quad + (I - \nht)[\mathcal{P}, D_t](I - \nht)\xi + (I - \nht)(D_t G) \\
& = -2[D_t\zeta, \nht]\frac{\partial_\alpha D_t^2(I - \nht)\xi}{\zeta_\alpha} \\
& \quad + \frac{1}{\pi i} \int \left(\frac{D_t\zeta(\alpha) - D_t\zeta(\beta)}{\zeta(\alpha) - \zeta(\beta)}\right) \partial_\beta D_t(I - \nht)\xi(\beta) d\beta \\
& \quad + (I - \nht)iU_{\kappa^{-1}}\left(\frac{\mathfrak{a}_t}{\mathfrak{a}}\right) \partial_\alpha(I - \nht)\xi \\
& \quad + (I - \nht)(D_t G)
\end{align*}
Similarly we have
\begin{align*}
\mathcal{P}(I - \nht)\tilde{D}_t(I - \tilde{\nht})\tilde{\xi} & = -[\mathcal{P}, \nht]\tilde{D}_t(I - \tilde{\nht})\tilde{\xi} + (I - \nht)\mathcal{P}\tilde{D}_t(I - \tilde{\nht})\tilde{\xi} \\
& = -2[D_t\zeta, \nht]\frac{\partial_\alpha D_t \tilde{D}_t(I - \tilde{\nht})\tilde{\xi}}{\zeta_\alpha} \\
& \quad + \frac{1}{\pi i} \int \left(\frac{D_t\zeta(\alpha) - D_t\zeta(\beta)}{\zeta(\alpha) - \zeta(\beta)}\right)^2 \partial_\beta \tilde{D_t}(I - \tilde{\nht})\tilde{\xi}(\beta) d\beta \\
& \quad + (I - \nht)\mathcal{P}\tilde{D}_t(I - \tilde{\nht})\tilde{\xi} \\
& = -2[D_t\zeta, \nht]\frac{\partial_\alpha D_t \tilde{D}_t(I - \tilde{\nht})\tilde{\xi}}{\zeta_\alpha} \\
& \quad +\frac{1}{\pi i} \int \left(\frac{D_t\zeta(\alpha) - D_t\zeta(\beta)}{\zeta(\alpha) - \zeta(\beta)}\right)^2 \partial_\beta \tilde{D_t}(I - \tilde{\nht})\tilde{\xi}(\beta) d\beta \\
& \quad + (I - \nht)(\mathcal{P} - \tilde{\mathcal{P}})\tilde{D}_t(I - \tilde{\nht})\tilde{\xi} -i (I - \nht)\tilde{b}_\alpha\partial_\alpha(I - \tilde{\nht})\tilde{\xi} \\
& \quad + (I - \nht)(\tilde{D}_t \tilde{G}) + (I - \nht)\epsilon^4 (\tilde{D}_t R)
\end{align*}
Subtracting these two equations then gives the desired evolution equation for $\sigma$:
\begin{align}\label{DtNewEulerRemainder}
4\mathcal{P}\sigma & = -8[D_t\zeta, \nht]\frac{\partial_\alpha D_t \sigma}{\zeta_\alpha} \notag\\
& \quad + \frac{4}{\pi i} \int \left(\frac{D_t\zeta(\alpha) - D_t\zeta(\beta)}{\zeta(\alpha) - \zeta(\beta)}\right)^2 \sigma_\beta(\beta) d\beta \notag\\
& \quad + (I - \nht)iU_{\kappa^{-1}}\left(\frac{\mathfrak{a}_t}{\mathfrak{a}}\right) \partial_\alpha(I - \nht)\xi \notag\\
& \quad - (I - \nht)(\mathcal{P} - \tilde{\mathcal{P}})\tilde{D}_t(I - \tilde{\nht})\tilde{\xi} \notag\\
& \quad +i (I - \nht)\tilde{b}_\alpha\partial_\alpha(I - \tilde{\nht})\tilde{\xi} \notag\\
& \quad + (I - \nht)(D_t G - \tilde{D}_t\tilde{G}) - (I - \nht)\epsilon^4 (\tilde{D}_t R)
\end{align}

The right hand side of \eqref{DtNewEulerRemainder} is $\mathcal{O}(\epsilon^{7/2})$ provided we can show that the right hand side of \eqref{DtNewEulerRemainder} is $\mathcal{O}(\epsilon^{7/2})$.  The formula \eqref{atOveraFormula} implies that the third term on the right hand side of \eqref{DtNewEulerRemainder} is of size $\mathcal{O}(\epsilon^{7/2})$.  Before we can show that the rest of the terms are appropriately small, we must study the quantities appearing on the right hand side of these equations further.  We will see that estimates for these quantities presuppose a satisfactory bound for the difference $\nht - \tilde{\nht}$, and so estimating this operator in Sobolev space is our first task.

\subsection{Estimates for the Difference Operator $\nht - \tilde{\nht}$}

While the operator $\tilde{\nht}$ is well suited for multiscale calculation, it remains to be seen how $\tilde{\nht}$ compares to our original Hilbert Transform $\nht$ corresponding to the true solution $\zeta$ of the water wave system.  To do so, we will bound the operator $\nht - \tilde{\nht}$ in $H^s$.  This entails decomposing it as $$\nht - \tilde{\nht} = (\nht - \nht_{\tilde{\zeta}}) + (\nht_{\tilde{\zeta}} - \tilde{\nht}),$$ where $\nht_{\tilde{\zeta}}$ is the Hilbert transform corresponding to the approximate interface $\tilde{\zeta}$.  If we apply Proposition \ref{SingIntSobolevEstimates} to the formula of Lemma \ref{DiffHilbertPart1Formula} we arrive at

\begin{lemma}\label{DiffHilbertBoundPart1}
Let $s \geq 4$ be given.  Then we have the bounds $$\|(\nht_{\tilde{\zeta}} - \tilde{\nht})f\|_{H^s} \leq C \epsilon^3 \|f\|_{H^s} \qquad \text{and} \qquad \|(\nht_{\tilde{\zeta}} - \tilde{\nht})f\|_{H^s} \leq C \epsilon^{5/2} \|f\|_{W^{s, \infty}}$$
where the constant $C = C
\left(\|B\|_{H^{s + 2}}\right)$.
\end{lemma}  

The analogous result for the first sum in the decomposition is
\begin{lemma}\label{DiffHilbertBoundPart2}
Let $s \geq 4$ be given, and suppose \eqref{ZetaLocalAPrioriBound} holds.  Then for all $t \leq T_0$, $$\|(\nht - \nht_{\tilde{\zeta}})f\|_{H^s} \leq C\|r_\alpha\|_{H^{s - 1}}\|f\|_{H^s} \qquad \text{and} \qquad \|(\nht - \nht_{\tilde{\zeta}})f\|_{H^s} \leq C\|r_\alpha\|_{H^{s - 1}}\|f\|_{W^{s, \infty}}$$ where the constant $C = C(\mathfrak{S}(T_0), \|B\|_{H^{s + 2}})$.
\end{lemma}

\begin{proof}
We use the fact that this operator can be written in two different ways using integration by parts:
\begin{align*}(\nht - \nht_{\tilde{\zeta}})f & = \frac{1}{\pi i} \int \log\left(1 + \frac{r(\alpha) - r(\beta)}{\tilde{\zeta}(\alpha) - \tilde{\zeta}(\beta)}\right) f_\beta(\beta) d\beta \\
& = \frac{1}{\pi i} \int \left(\frac{r_\beta(\beta)}{\zeta(\alpha) - \zeta(\beta)} - \frac{\tilde{\zeta}_\beta(r(\alpha) - r(\beta))}{(\zeta(\alpha) - \zeta(\beta))(\tilde{\zeta}(\alpha) - \tilde{\zeta}(\beta))}\right) f(\beta)d\beta
\end{align*}  Now consider the $n$th derivative of the first formula.  If all $n$ derivatives fall on $f$, then we can pass to an integral of the second form above via integration by parts.  Such an integral can then be bounded in $L^2$ by either $$C\left(\mathfrak{S}(T_0), \|B\|_{H^{n + 2}}\right)\|r_\alpha\|_{H^2}\|f\|_{H^n} \qquad \text{or} \qquad C\left(\mathfrak{S}(T_0), \|B\|_{H^{n + 2}}\right)\|r_\alpha\|_{H^2}\|f\|_{W^{n, \infty}}$$  If at least one derivative falls on the logarithm, then we have a kernel of the form $$(\partial_\alpha + \partial_\beta)\log\left(1 + \frac{r(\alpha) - r(\beta)}{\tilde{\zeta}(\alpha) - \tilde{\zeta}(\beta)}\right) = \frac{r_\alpha(\alpha) - r_\beta(\beta)}{\zeta(\alpha) - \zeta(\beta)} - \frac{(r(\alpha) - r(\beta))(\tilde{\zeta}_\alpha(\alpha) - \tilde{\zeta}_\beta(\beta))}{(\zeta(\alpha) - \zeta(\beta))(\tilde{\zeta}(\alpha) - \tilde{\zeta}(\beta))}$$  This yields a singular integral which can be bounded in $H^n$ by either $$C\left(\mathfrak{S}(T_0), \|\tilde{\zeta}_\alpha - 1\|_{H^{n + 1}}\right)\|r_\alpha\|_{H^{n - 1}}\|f\|_{H^{n - 1}} \; \; \text{or} \; \; C\left(\mathfrak{S}(T_0), \|\tilde{\zeta}_\alpha - 1\|_{H^{n + 1}}\right)\|r_\alpha\|_{H^{n - 1}}\|f\|_{W^{n - 1, \infty}}$$  The proposition follows by summing these bounds $n = 0, 1, \ldots, s$.
\end{proof}

Combining these lemmas yields the

\begin{corollary}\label{DiffHilbertBound}
Let $s \geq 4$ be given, and suppose that \eqref{ZetaLocalAPrioriBound} holds.  Then for all $t \leq T_0$, $$\|(\nht - \tilde{\nht})f\|_{H^s} \leq C(\epsilon^3 + \|r_\alpha\|_{H^{s - 1}})\|f\|_{H^s}$$
$$\|(\nht - \tilde{\nht})f\|_{H^s} \leq C(\epsilon^{5/2} + \|r_\alpha\|_{H^{s - 1}})\|f\|_{W^{s, \infty}}$$ where $C = C\left(\mathfrak{S}(T_0), \|B\|_{H^{s + 2}}\right)$.
\end{corollary}

We will also need to estimate the operator $D_t(\nht - \tilde{\nht})$.  To do so, it will suffice to consider the commutator $[D_t, \nht - \tilde{\nht}]$.

\begin{proposition}\label{DiffHilbertCommuteDtBound}
Let $s \geq 4$, and suppose that \eqref{ZetaLocalAPrioriBound} holds.  Then $\|[D_t, \nht - \tilde{\nht}]f\|_{H^s} \leq C(\epsilon^3 + \|r_\alpha\|_{H^{s - 1}} + \|D_t r\|_{H^s})\|f\|_{H^s}$, where the constant $C = C\left(\mathfrak{S}(T_0), \|B\|_{H^{s + 4}}\right)$.
\end{proposition}

\begin{proof}
We decompose $\nht - \tilde{\nht} = (\nht - \nht_{\tilde{\zeta}}) + (\nht_{\tilde{\zeta}} - \tilde{\nht})$ and estimate each term separately.  We begin with the latter operator and apply Lemma \ref{SingIntCommuteWithDt} to \eqref{HTildeZetaMinusTildeHFormula}.  Using the product rule, this results in a sum of singular integrals whose numerators are products of differences involving the functions $\tilde{\xi}$, $\zeta^{(n)}$, $D_t \tilde{\xi}$, $D_t \zeta^{(n)}$, $n = 1, 2, 3$.  Then using the identity $D_t g = (b - \tilde{b})g_\alpha + \tilde{D}_t g$, we can further split these terms until we arrive at a sum of kernels whose numerators are products of differences involving the functions $$\tilde{\xi}, \; \zeta^{(n)}, \; \tilde{D}_t \tilde{\xi}, \; \tilde{D}_t \zeta^{(n)}, \; (b - \tilde{b})\tilde{\xi}, \; (b - \tilde{b})\zeta^{(n)}_\alpha, \qquad n = 1, 2, 3$$

In order to estimate the terms $(b - \tilde{b})g$ that arise here for $g = \tilde{\xi}, \zeta^{(n)}$, notice that \eqref{ZetaLocalAPrioriBound}, along with \eqref{bFormula} and Lemma \ref{DoubleLayerPotentialArgument}, shows that $$\|(b - \tilde{b})g_\alpha\|_{H^s} \leq C\|b - \tilde{b}\|_{H^s}\|g_\alpha\|_{W^{s, \infty}} \leq C\left(\mathfrak{S}(T_0)\right)\|g_\alpha\|_{W^{s, \infty}}$$

The resulting kernels have the properties that (1) each has at least three factors in its numerator of size at most $O(\epsilon)$ in the sense of $L^\infty$, (2) each has the same number of factors in the numerator as in the denominator.  In estimating this sum of singular integrals we always estimate $f$ in $L^2$ so as not to lose any half-powers of $\epsilon$.  In doing so, the largest number of derivatives of $B$ that appears is in $\tilde{D}_t \tilde{\zeta}$; a time derivative will fall on $B_X$ in the formula for $\zeta^{(3)}$ which by \eqref{NLS} is equivalent to a term with three derivatives on $B$.  The result is the bound $C(\mathfrak{S}(T_0), \|B\|_{H^{s + 3}})\epsilon^3\|f\|_{H^s}$.

Next, using Lemma \ref{SingIntCommuteWithDt}, we explicitly write the kernel $$[D_t, \nht - \nht_{\tilde{\zeta}}]f = \frac{1}{\pi i} \int f_\beta(\beta) \left(\frac{D_t r(\alpha) - D_t r(\beta)}{\zeta(\alpha) - \zeta(\beta)} - \frac{(r(\alpha) - r(\beta))(D_t \tilde{\zeta}(\alpha) - D_t \tilde{\zeta}(\beta))}{(\zeta(\alpha) - \zeta(\beta))(\tilde{\zeta}(\alpha) - \tilde{\zeta}(\beta))}\right) d\beta$$ and appealing to the crude bound of $b - \tilde{b}$ above now implies the proposition.
\end{proof}

\begin{corollary}\label{DtHilbertDiffBound}
Let $s \geq 4$ be given, and suppose that \eqref{ZetaLocalAPrioriBound} holds.  Then $$\|D_t(\nht - \tilde{\nht})f\|_{H^s} \leq C(\epsilon^3 + \|r_\alpha\|_{H^{s - 1}} + \|D_t r\|_{H^s})(\|f\|_{H^s} + \|D_t f\|_{H^s}),$$ where the constant $C = C\left(\mathfrak{S}(T_0), \|B\|_{H^{s + 3}}\right)$.
\end{corollary}

\subsection{Formulas for Remainders of $b$ and $\mathcal{A}$}\label{formulaforremainders}

Applying the energy method to the remainder equations \eqref{NewEulerRemainder}-\eqref{DtNewEulerRemainder},
we expect to obtain bounds on the quantity:
\begin{equation}\label{ProxyRemainderEnergy}
E_s^{1/2} := \|r_\alpha\|_{H^s} + \|D_t r\|_{H^s}.
\end{equation}
However in \eqref{NewEulerRemainder}-\eqref{DtNewEulerRemainder}, the quantities $b - \tilde{b}$, $\mathcal{A} - \tilde{\mathcal{A}}$, etc., arise as coefficients of the operators $\mathcal{P} - \tilde{\mathcal{P}}$ and $D_t(\mathcal{P} - \tilde{\mathcal{P}})$.  Moreover, such energy estimates would give bounds on the quantities $D_t\partial_\alpha^n \rho$ and $D_t\partial_\alpha^n \sigma$, not directly on the quantities $r_\alpha$ and $D_t r$. 
So in the following subsections we must perform the following tasks:
\begin{enumerate}
\item{Bound $b - \tilde{b}$ in terms of $E_s$ and $\epsilon$.}
\item{Bound $D_t(b - \tilde{b})$ in terms of $E_s$, $\epsilon$, and a small multiple of $D_t^2 r$.}
\item{Bound $\mathcal{A} - \tilde{\mathcal{A}}$ in terms of $E_s$, $\epsilon$, and a small multiple of $D_t^2 r$.}
\item{Bound $D_t^2 r$ in terms of $E_s$, $\epsilon$ and a small multiple of $\mathcal{A} - \tilde{\mathcal{A}}$, and thus bound $D_t^2 r$, $\mathcal{A} - \tilde{\mathcal{A}}$ and $D_t(b - \tilde{b})$ appropriately by $E_s$ and $\epsilon$ alone.}
\item{Show that $D_t\rho$ and $D_t\sigma$ are equivalent to $D_t r$ and $D_t^2 r$, respectively.}
\end{enumerate}
Since $\tilde{b}$ and $\tilde{\mathcal{A}}$ are intended to be power expansions in $\epsilon$ of $b$ and $\mathcal{A}$ up to at least quadratic terms, we expect that the differences $b - \tilde{b}$, $D_t(b - \tilde{b})$ and $\mathcal{A} - \tilde{\mathcal{A}}$ will be of size $\mathcal{O}(\epsilon^{5/2})$.

\subsubsection*{Step 1.  Controlling $b - \tilde{b}$ by $E_s$ and $\epsilon$.}

In order to use \eqref{bFormula}, we write
\begin{align*}
(I - \nht)(b - \tilde{b}) & = (I - \nht)b + (\nht - \tilde{\nht})\tilde{b} - (I - \tilde{\nht})\tilde{b}
\end{align*}
By the multiscale calculation, the residual quantity $$(I - \tilde{\nht})\tilde{b} + [\tilde{D}_t\tilde{\zeta}, \tilde{\nht}]\frac{\overline{\tilde{\zeta}}_\alpha - 1}{\tilde{\zeta}_\alpha}$$ consists only of terms $O(\epsilon^4)$.  The largest number of derivatives of $B$ appearing in this residual is through the term $\nht_2 \tilde{D}_t \zeta^{(3)}$, where three derivatives fall on $B$.  Hence this residual is bounded in $H^s$ by $C(\|B\|_{H^{s + 3}})\epsilon^{7/2}$.  By Corollary  \ref{DiffHilbertBound}, we have
\begin{align*}
\|(\nht - \tilde{\nht})\tilde{b}\|_{H^s} & \leq C(\epsilon^3 + E_s^{1/2})\|\tilde{b}\|_{H^s} \\
& \leq C(\epsilon^3 + E_s^{1/2})\epsilon^{3/2} \\
& \leq C(\epsilon E_s^{1/2} + \epsilon^{5/2})
\end{align*} where $C = C(\mathfrak{S}(T_0), \|B\|_{H^{s + 3}})$.  Observe that in the last step we have relaxed the estimate so that every term is of the optimal size $\mathcal{O}(\epsilon^{5/2})$.  

It now suffices to consider the difference
\begin{align*}
-[D_t \zeta, \nht]\frac{\overline{\xi}_\alpha}{\zeta_\alpha} + [\tilde{D}_t\tilde{\zeta}, \tilde{\nht}]\frac{\overline{\tilde{\xi}}_\alpha}{\tilde{\zeta}_\alpha} & = -[D_t r, \nht]\frac{\overline{\xi}_\alpha}{\zeta_\alpha} - [(D_t - \tilde{D}_t)\tilde{\zeta}, \nht]\frac{\overline{\xi}_\alpha}{\zeta_\alpha} \\
& \quad - [\tilde{D}_t\tilde{\zeta}, \nht]\frac{r_\alpha}{\zeta_\alpha} - [\tilde{D}_t\tilde{\zeta}, \nht]\overline{\tilde{\xi}}_\alpha\left(\frac{1}{\zeta_\alpha} - \frac{1}{\tilde{\zeta}_\alpha}\right) \\
& \quad - [\tilde{D}_t\tilde{\zeta}, \nht - \tilde{\nht}] \frac{\overline{\tilde{\xi}}_\alpha}{\tilde{\zeta}_\alpha}
\end{align*}  Estimating each of these terms in $H^s$ using Proposition \ref{SingIntSobolevEstimates}, we sum the bounds under the assumption of \eqref{ZetaLocalAPrioriBound} to find by Corollary \ref{DiffHilbertBound} that for $s \geq 4$:
\begin{align*}
\|b - \tilde{b}\|_{H^s} & \leq C\epsilon^{7/2} \; + \; C(\epsilon E_s^{1/2} + \epsilon^{5/2}) \\
& \quad + \; C E_s^{1/2}(E_s^{1/2} + \epsilon) \; + \; C\|b - \tilde{b}\|_{H^s}(\delta + \epsilon) \\
& \quad + \; C\epsilon E_s^{1/2} \; + \; C\epsilon(\epsilon^3 + E_s^{1/2})\epsilon^{1/2} \\
& \leq C\left(E_s + \epsilon E_s^{1/2} + \epsilon^{5/2}\right) + C\|b - \tilde{b}\|_{H^s}(\epsilon + \delta)
\end{align*} 
and so choosing $\epsilon_0$ and $\delta$ so that the coefficient $C(\epsilon + \delta)$ of $\|b - \tilde{b}\|_{H^s}$ on the right hand side is less than $\frac{1}{2}$ for all $\epsilon < \epsilon_0$ yields the bound \begin{equation}\label{bMinusTildebBound}\|b - \tilde{b}\|_{H^s} \leq C\left(E_s + \epsilon E_s^{1/2} + \epsilon^{5/2}\right)\end{equation} where the constant $C = C(\mathfrak{S}(T_0), \|B\|_{H^{s + 4}})$.  From this bound and \eqref{TildeBFormula} we also have
\begin{equation}\label{bBound}\|b\|_{H^s} \leq C\left(E_s^{1/2} + \epsilon^{3/2}\right)\end{equation}

\subsubsection*{Step 2.  Controlling $D_t(b - \tilde{b})$ by $E_s$, $\epsilon$, and a small multiple of $D_t^2 r$.}
To control $D_t(b - \tilde{b})$, we write
\begin{equation*}
(I - \nht)D_t(b - \tilde{b}) 
= \left((I - \nht)D_tb - (I - \tilde{\nht})\tilde{D}_t\tilde{b}\right) + (\nht - \tilde{\nht})\tilde{D}_t\tilde{b} - (I - {\nht})(b - \tilde{b})\partial_\alpha\tilde{b}
\end{equation*}

By Step 1 and Corollary \ref{DiffHilbertBound} we have that
$$\| (I - {\nht})(b - \tilde{b})\partial_\alpha\tilde{b}   \|_{H^s} \leq C(E_s + \epsilon E_s^{1/2} + \epsilon^{5/2})$$ and 
$$\| (\nht - \tilde{\nht})\tilde{D}_t\tilde{b}     \|_{H^s} \leq C(\epsilon^3 + E_s^{1/2})(\epsilon^{5/2}) \leq C(\epsilon E_s^{1/2} + \epsilon^{5/2})$$ where the constant $C$ depends only on $\mathfrak{S}(T_0)$ and $\|B\|_{H^{s + 4}}$.  To estimate the remaining terms we appeal to the formula \eqref{DtbFormula}:  
\begin{align*}
(I - \nht)D_t b & = [D_t\zeta, \nht]\frac{\partial_\alpha(2b - D_t\overline{\zeta})}{\zeta_\alpha} - [D_t^2\zeta, \nht]\frac{\overline{\zeta}_\alpha - 1}{\zeta_\alpha} \notag \\ 
& \qquad + \frac{1}{\pi i} \int \left(\frac{D_t\zeta(\alpha) - D_t\zeta(\beta)}{\zeta(\alpha) - \zeta(\beta)}\right)^2 (\overline{\zeta}_\beta(\beta) - 1) d\beta
\end{align*}
By a multiscale calculation, the term $(I - \tilde{\nht})\tilde{D}_t\tilde{b}$ has the property that the residual quantity
\begin{align*}
& (I - \tilde{\nht})\tilde{D}_t\tilde{b} - [\tilde{D}_t\tilde{\zeta}, \tilde{\nht}]\frac{\partial_\alpha(2\tilde{b} - \tilde{D}_t\overline{\tilde{\zeta}})}{\tilde{\zeta}_\alpha} + [\tilde{D}_t^2\tilde{\zeta}, \tilde{\nht}]\frac{\overline{\tilde{\zeta}}_\alpha - 1}{\tilde{\zeta}_\alpha} \notag \\ 
& \qquad - \frac{1}{\pi i} \int \left(\frac{\tilde{D}_t\tilde{\zeta}(\alpha) - \tilde{D}_t\tilde{\zeta}(\beta)}{\tilde{\zeta}(\alpha) - \tilde{\zeta}(\beta)}\right)^2 (\overline{\tilde{\zeta}}_\beta(\beta) - 1) d\beta
\end{align*} is of size $O(\epsilon^4)$.  Therefore it suffices to estimate the difference between each term in \eqref{DtbFormula} with its approximate analogue.  We may estimate the first such difference crudely, since by Step 1 we have that
\begin{align*}
\left\|[D_t \zeta, \nht]\frac{b_\alpha}{\zeta_\alpha} - [\tilde{D}_t \tilde{\zeta}, \tilde{\nht}]\frac{\tilde{b}_\alpha}{\tilde{\zeta}_\alpha}\right\|_{H^s} & \leq \left\|[D_t \zeta, \nht]\frac{(b - \tilde{b})_\alpha}{\zeta_\alpha}\right\|_{H^s} + \left\|[D_t \zeta, \nht]\frac{\tilde{b}_\alpha}{\zeta_\alpha}\right\|_{H^s} + \left\|[\tilde{D}_t \tilde{\zeta}, \tilde{\nht}]\frac{\tilde{b}_\alpha}{\tilde{\zeta}_\alpha}\right\|_{H^s} \\
& \leq C(E_s + \epsilon E_s^{1/2} + \epsilon^{5/2}) + C\delta\epsilon^{5/2} + C\epsilon^{1/2}\epsilon^{5/2} \\
& \leq C(E_s + \epsilon E_s^{1/2} + \epsilon^{5/2}),
\end{align*} where the constant $C$ depends only on $\mathfrak{S}(T_0)$ and $\|B\|_{H^{s + 4}}$, and where we estimated the commutator $[D_t \zeta, \nht]\frac{\tilde{b}_\alpha}{\zeta_\alpha}$ term-by-term.  The estimate of the difference $$[D_t\zeta, \nht]\frac{\partial_\alpha D_t\overline{\zeta}}{\zeta_\alpha} - [\tilde{D}_t\tilde{\zeta}, \tilde{\nht}]\frac{\partial_\alpha \tilde{D}_t\overline{\tilde{\zeta}}}{\tilde{\zeta}_\alpha}$$ proceeds by decomposing in the same manner as in Step 1, and yields the bound $C(E_s + \epsilon E_s^{1/2} + \epsilon^{5/2})$.

Next, by writing $D_t \zeta = D_t r + (b - \tilde{b})\tilde{\zeta}_\alpha + \tilde{D}_t \tilde{\zeta}$, the remaining singular integrals
$$\frac{1}{\pi i} \int \left(\frac{D_t\zeta(\alpha) - D_t\zeta(\beta)}{\zeta(\alpha) - \zeta(\beta)}\right)^2 (\overline{\zeta}_\beta(\beta) - 1) d\beta - \frac{1}{\pi i} \int \left(\frac{\tilde{D}_t\tilde{\zeta}(\alpha) - \tilde{D}_t\tilde{\zeta}(\beta)}{\tilde{\zeta}(\alpha) - \tilde{\zeta}(\beta)}\right)^2 (\overline{\tilde{\zeta}}_\beta(\beta) - 1) d\beta$$ are controlled in $H^s$ with Proposition \ref{SingIntSobolevEstimates} by $C(E_s + \epsilon E_s^{1/2} + \epsilon^{5/2})$.  Finally we address the difference $$[D_t^2\zeta, \nht]\frac{\overline{\zeta}_\alpha - 1}{\zeta_\alpha} - [\tilde{D}_t^2\tilde{\zeta}, \tilde{\nht}]\frac{\overline{\tilde{\zeta}}_\alpha - 1}{\tilde{\zeta}_\alpha}.$$  Again decomposing in the fashion of Step 1, we arrive at a sum of commutators all controlled in $H^s$ by $C(E_s + \epsilon E_s^{1/2} + \epsilon^{5/2})$ except for two commutators.  The first is
$$[D_t^2 r, \nht]\frac{\overline{\zeta}_\alpha - 1}{\zeta_\alpha}$$ which is controlled in $H^s$ by $(E_s^{1/2} + \epsilon)\|D_t^2 r\|_{H^s}$.  The second is
\begin{align*}
[(D_t^2 - \tilde{D}_t^2)\tilde{\zeta}, \nht]\frac{\overline{\zeta}_\alpha - 1}{\zeta_\alpha} & = \left[\left(D_t(b - \tilde{b}) \right)\tilde{\zeta}_\alpha + (b - \tilde{b})\left(D_t\tilde{\zeta}_\alpha + \partial_\alpha \tilde{D}_t\tilde{\zeta}\right), \nht\right]\frac{\overline{\zeta}_\alpha - 1}{\zeta_\alpha}
\end{align*} which has been expanded using \eqref{DiffDtSquaredFormula}, and is controlled in $H^s$ by
$$C\delta\|D_t(b - \tilde{b})\|_{H^s} + C(E_s + \epsilon E_s^{1/2} + \epsilon^{5/2})$$  Summing all of these estimates, we therefore have for $\delta$ chosen sufficiently small that
\begin{align}\label{PrelimDtbBound}
\|D_t(b - \tilde{b})\|_{H^s} \leq C(E_s + \epsilon E_s^{1/2} + \epsilon^{5/2}) + C(E_s^{1/2} + \epsilon)\|D_t^2 r\|_{H^s} \end{align}

\subsubsection*{Step 3.  Controlling $\mathcal{A} - \tilde{\mathcal{A}}$ in terms of $E_s$, $\epsilon$, and a small multiple of $D_t^2 r$.}

Since $\tilde{\mathcal{A}} = 1$ by \eqref{TildeAFormula}, it suffices to control $\mathcal{A} - 1$ in $H^s$.  The right hand side of the formula \eqref{TildeAFormula} consists of terms that are almost the same as those in the formula \eqref{DtbFormula} for $D_t b$, and so the same methods of estimation will apply.  However, from \S 3.3 we know that $\mathcal{A}_2 = 0$, and so we will want to decompose the right hand side of the formula \eqref{AFormula} so that it is easily seen that the pure $O(\epsilon^2)$ contribution vanishes.  From \eqref{AFormula} we have $$(I - \nht)(\mathcal{A} - 1) = i[D_t^2 \zeta, \nht]\frac{\overline{\zeta}_\alpha - 1}{\zeta_\alpha} + i[D_t \zeta, \nht]\frac{\partial_\alpha D_t \overline{\zeta}}{\zeta_\alpha} := I_1 + I_2$$  Decomposing the difference corresponding to $I_2$ as in Step 2, we have $$\left\|[D_t \zeta, \nht]\frac{\partial_\alpha D_t \overline{\zeta}}{\zeta_\alpha} - [\tilde{D}_t \tilde{\zeta}, \tilde{\nht}]\frac{\partial_\alpha \tilde{D}_t \overline{\tilde{\zeta}}}{\tilde{\zeta}_\alpha}\right\|_{H^s} \leq C\left(E_s + \epsilon E_s^{1/2} + \epsilon^{5/2}\right),$$ where $C = C(\mathfrak{S}(T_0), \|B\|_{H^{s + 4}})$.  The difference corresponding to $I_1$ is decomposed as follows:
\begin{align*}
[D_t^2 \zeta, \nht]\frac{\overline{\xi}_\alpha}{\zeta_\alpha} - [\tilde{D}_t^2 \tilde{\zeta}, \tilde{\nht}] \frac{\overline{\tilde{\xi}}_\alpha}{\tilde{\zeta}_\alpha} & = [D_t^2 r, \nht]\frac{\overline{\xi}_\alpha}{\zeta_\alpha} + [(D_t^2 - \tilde{D}_t^2)\tilde{\zeta}, \nht]\frac{\overline{\xi}_\alpha}{\zeta_\alpha} \\
& \quad + [\tilde{D}_t^2\tilde{\zeta}, \nht]\left(\frac{\overline{\xi}_\alpha}{\zeta_\alpha} - \frac{\overline{\tilde{\xi}}_\alpha}{\tilde{\zeta}_\alpha}\right) + [\tilde{D}_t^2 \tilde{\zeta}, \nht - \tilde{\nht}] \frac{\overline{\tilde{\xi}}_\alpha}{\tilde{\zeta}_\alpha}
\end{align*}  Note that in the expression $\tilde{D}_t^2 \tilde{\zeta}$, five derivatives fall on $B$ through $\zeta^{(3)}$, and so we need five extra derivatives on $B$ to bound $\tilde{D}_t^2 \tilde{\zeta}$ in $H^s$.  Using Step 1 and Corollary \ref{DiffHilbertBound} then gives
\begin{align*}
\left\|[D_t^2 \zeta, \nht]\frac{\overline{\xi}_\alpha}{\zeta_\alpha} - [\tilde{D}_t^2 \tilde{\zeta}, \tilde{\nht}] \frac{\overline{\tilde{\xi}}_\alpha}{\tilde{\zeta}_\alpha}\right\|_{H^s} & \leq C(E_s + \epsilon E_s^{1/2} + \epsilon^{5/2}) \\
& \quad + C\left(\epsilon + E_s^{1/2}\right)\left(\|D_t^2 r\|_{H^s} + \|D_t(b - \tilde{b})\|_{H^s}\right)
\end{align*} where $C = C(\mathfrak{S}(T_0), \|B\|_{H^{s + 5}})$.  Now since a multiscale calculation shows that the function $$[\tilde{D}_t \tilde{\zeta}, \tilde{\nht}]\frac{\partial_\alpha \tilde{D}_t \overline{\tilde{\zeta}}}{\tilde{\zeta}_\alpha} + [\tilde{D}_t^2 \tilde{\zeta}, \tilde{\nht}] \frac{\overline{\tilde{\xi}}_\alpha}{\tilde{\zeta}_\alpha}$$ consists only of terms of order $O(\epsilon^3)$, the highest number of derivatives appearing is through the term $\nht_2 \partial_t^2 \zeta^{(3)}$ which contains five derivatives of $B$.  This residual is thus controlled in $H^s$ by $C(\|B\|_{H^{s + 5}})\epsilon^{5/2}$.  Combining these estimates, we can choose $\epsilon_0$ and $\delta$ sufficiently small so as to arrive at the following estimate for $\mathcal{A} - \tilde{\mathcal{A}}$:
\begin{align*}
\|\mathcal{A} - \tilde{\mathcal{A}}\|_{H^s} \leq C\left(E_s + \epsilon E_s^{1/2} + \epsilon^{5/2}\right) + C\left(\epsilon + E_s^{1/2}\right)(\|D_t^2 r\|_{H^s} + \|D_t(b - \tilde{b})\|_{H^s})
\end{align*}
Now using Step 2 and possibly choosing $\delta$ and $\epsilon_0$ smaller still allows us to give the following preliminary bound for $\mathcal{A} - \tilde{\mathcal{A}}$:
\begin{equation}\label{PrelimAMinusTildeABound}
\|\mathcal{A} - \tilde{\mathcal{A}}\|_{H^s} \leq C\left(E_s + \epsilon E_s^{1/2} + \epsilon^{5/2}\right) + C\left(\epsilon + E_s^{1/2}\right)\|D_t^2 r\|_{H^s}
\end{equation}

\subsubsection*{Step 4.  Bounding $D_t^2 r$ in terms of $E_s$, $\epsilon$, and a small multiple of $\mathcal{A} - \tilde{\mathcal{A}}$.}

We start by deriving a formula for $D_t^2 r$.  Changing variables via $U_{\kappa^{-1}}$ in \eqref{OldEuler} yields the equation
$\mathcal{P}\zeta = -i$ and so decomposing as $\xi = \tilde{\xi} + r$ yields
\begin{align*}
\mathcal{P}r & = -i - \mathcal{P}\alpha - \mathcal{P}\tilde{\xi} \\
& = -i - (D_t b - i\mathcal{A}) - (\mathcal{P} - \tilde{\mathcal{P}})\tilde{\xi} - \tilde{\mathcal{P}}\tilde{\xi} \\
& = -D_t b + i(\mathcal{A} - 1) - (\mathcal{P} - \tilde{\mathcal{P}})\tilde{\xi} - \tilde{\mathcal{P}}\tilde{\xi}
\end{align*} and so
\begin{equation}\label{DtSquaredrFormula}
D_t^2 r - ir_\alpha = i(\mathcal{A} - 1)(1 + \xi_\alpha) - (D_t^2 - \tilde{D}_t^2)\tilde{\xi} - \tilde{\mathcal{P}}\tilde{\xi} - D_t b
\end{equation}  By Proposition \ref{VariousMultiscaleIdentities} we have $\|\tilde{\mathcal{P}}\tilde{\xi}\|_{H^s} \leq C \epsilon^{5/2}$ with the constant depending on $\mathfrak{S}(T_0)$ and 
$\|B\|_{H^{s + 5}}$.  Next, using Step 1, \eqref{DiffDtSquaredFormula} and \eqref{PrelimAMinusTildeABound} gives
\begin{align*}
\|(D_t^2 - \tilde{D}_t^2)\tilde{\xi}\|_{H^s} & \leq C\epsilon\|D_t(b - \tilde{b})\|_{H^s} + C\epsilon\left(E_s + \epsilon E_s^{1/2} + \epsilon^{5/2}\right) \\
& \leq C(\epsilon^{1/2} + \delta)\|D_t^2 r\|_{H^s} + C\left(E_s + \epsilon E_s^{1/2} + \epsilon^{5/2}\right)
\end{align*}  We also have
\begin{align*}
\|D_t b\|_{H^s} & \leq \|D_t(b - \tilde{b})\|_{H^s} + \|(D_t - \tilde{D}_t)\tilde{b}\|_{H^s} + C\epsilon^{5/2} \\
& \leq C(\epsilon^{1/2} + \delta)\|D_t^2 r\|_{H^s} + C(E_s + \epsilon E_s^{1/2} + \epsilon^{5/2})
\end{align*}
Finally we have from \eqref{PrelimAMinusTildeABound} that $$\|(\mathcal{A} - 1)\zeta_\alpha\|_{H^s} \leq C\left(E_s + \epsilon E_s^{1/2} + \epsilon^{5/2}\right) + C(\epsilon^{1/2} + \delta)\|D_t^2 r\|_{H^s}$$  Combining these estimates through \eqref{DtSquaredrFormula} gives $$\|D_t^2 r - ir_\alpha\|_{H^s} \leq C\left(E_s + \epsilon E_s^{1/2} + \epsilon^{5/2}\right) + C(\epsilon^{1/2} + \delta)\|D_t^2 r\|_{H^s}$$  Hence we can choose $\epsilon_0$ and $\delta$ sufficiently small so that
\begin{align}\label{DtSquaredrBound}
\|D_t^2 r\|_{H^s} & \leq \|r_\alpha\|_{H^s} + C\left(E_s + \epsilon E_s^{1/2} + \epsilon^{5/2}\right) \notag\\
& \leq C(E_s^{1/2} + \epsilon^{5/2})
\end{align} where the constant $C$ depends only on $\mathfrak{S}(T_0)$ and $\|B\|_{H^{s + 5}}$.  Then we immediately have
\begin{equation}\label{AMinusTildeABound}
\|\mathcal{A} - \tilde{\mathcal{A}}\|_{H^s} \leq C(E_s + \epsilon E_s^{1/2} + \epsilon^{5/2})
\end{equation} by virtue of Step 3, as well as
\begin{equation}\label{dtbminusb}
\|D_t(b - \tilde{b})\|_{H^s} \leq C(E_s + \epsilon E_s^{1/2} + \epsilon^{5/2})
\end{equation} from Step 2.  From this last inequality we have
\begin{equation}\label{DtbBound}
\|D_t b\|_{H^s} \leq C(E_s + \epsilon E_s^{1/2} + \epsilon^{5/2})
\end{equation}

Note that from \eqref{DtSquaredrFormula}, applying \eqref{AMinusTildeABound}, \eqref{DtbBound},
we also have the estimate
\begin{equation*}
\|r_\alpha\|_{H^s} - C(E_s + \epsilon E_s^{1/2}) \leq C\|D_t^2 r\|_{H^s} + C\epsilon^{5/2}
\end{equation*} and hence if we choose $\delta$ and $\epsilon_0$ sufficiently small, we conclude that
\begin{equation}\label{DtSquaredrEnergyBound}
E_s^{1/2} \leq C(\|D_t r\|_{H^s} + \|D_t^2 r\|_{H^s} + \epsilon^{5/2})
\end{equation}

\subsubsection*{Step 5.  Showing that $D_t\rho$, $\sigma$ and $D_t\sigma$ are equivalent to $r_\alpha$ and $D_t r$.}

In the sequel we will show that the energy constructed from the equations of \S 4.1 is bounded below by the sum
$$\sum_{n = 0}^s \|D_t \partial_\alpha^n \rho\|_{L^2} + \|D_t\partial_\alpha^n\sigma\|_{L^2}$$  Therefore, this energy will control $E_s$ provided we can show that $E_s$ is bounded above by this sum.  We will show that this is the case with the following three claims.

\vspace{0.25cm}
\noindent \textbf{Claim 1.}  For $s \geq 4$ we have, for $\delta$ and $\epsilon < \epsilon_0$ chosen sufficiently small, that $$\|D_t r\|_{H^s} \leq C\|\sigma\|_{H^s} + C(\delta + \epsilon)E_s^{1/2} + C\epsilon^{5/2}$$ and $$\|\sigma\|_{H^s} \leq CE_s^{1/2} + C\epsilon^{5/2}$$
\vspace{0.25cm}

\textit{Proof of Claim 1.}  Denote $\mathscr{I} := \frac{1}{2}D_t(I - \nht)\xi - \frac{1}{2}\tilde{D}_t(I - \tilde{\nht})\tilde{\xi}$.  First consider the difference
\begin{equation}\label{dti}
\begin{aligned}
D_t r - \mathscr{I} & = D_t\xi - \frac{1}{2}D_t(I - \nht)\xi \\
& \quad - \tilde{D}_t\tilde{\xi} +\frac{1}{2}\tilde{D}_t(I - \tilde{\nht})\tilde{\xi} \\
& \quad - (D_t - \tilde{D}_t)\tilde{\xi} \\
& = \frac{1}{2}D_t(\nht + \nhtb)\xi \\
& \quad - \tilde{D}_t\tilde{\xi} + \frac{1}{2}\tilde{D}_t(I - \tilde{\nht})\tilde{\xi} \\
&\quad  - (b - \tilde{b})\tilde{\xi}_\alpha
\end{aligned}
\end{equation}
  By Step 1 we have that $\|(b - \tilde{b})\tilde{\xi}_\alpha\|_{H^s} \leq C(E_s + \epsilon E_s^{1/2} + \epsilon^{5/2})$, and by a multiscale calculation we have that $\|\tilde{D}_t\tilde{\xi} - \frac{1}{2}\tilde{D}_t(I - \tilde{\nht})\tilde{\xi}\|_{H^s} \leq C\epsilon^{5/2}$.  The final term can be expanded as
\begin{align*}
\frac{1}{2}D_t(\nht + \nhtb)\xi & = \frac{1}{2}[D_t \zeta, \nht]\frac{\xi_\alpha}{\zeta_\alpha} + \frac{1}{2}[D_t \overline{\zeta}, \nhtb]\frac{\xi_\alpha}{\overline{\zeta}_\alpha} + \frac{1}{2}(\nht + \nhtb)D_t \xi
\end{align*}  Decomposing these terms as in Step 1 yields a sum of terms all bounded in $H^s$ by $C(E_s + \epsilon E_s^{1/2})$.  The only terms which are not immediately $\mathcal{O}(\epsilon^{5/2})$ after this decomposition are $$\frac{1}{2}[\tilde{D}_t \overline{\tilde{\zeta}}, \overline{\tilde{\nht}}] \frac{\tilde{\xi}_\alpha}{\overline{\tilde{\zeta}}_\alpha} + \frac{1}{2}(\tilde{\nht} + \overline{\tilde{\nht}})\tilde{D}_t \tilde{\xi}$$ whose leading  $O(\epsilon^2)$ term is $$\frac{1}{2}[\overline{\zeta}^{(1)}_{t_0}, \nhtb_0]\zeta^{(1)}_{\alpha_0} + \frac{1}{2}\nhtb^{(1)}\zeta^{(1)}_{t_0} = 0.$$  Hence, we have $$\left\|D_t r - \mathscr I\right\|_{H^s} \leq C(E_s + \epsilon E_s^{1/2} + \epsilon^{5/2})$$
We can further write
\begin{align*}
D_t r - \sigma & = D_t r - \frac{1}{2}(I - \nht)\mathscr{I} \\
& = \frac{1}{2}(I - \nhtb)D_tr + \frac{1}{2}(\nht + \nhtb)D_tr \\
& \quad + \frac{1}{2}(I - \nht)(D_t r - \mathscr{I})
\end{align*} which by virtue of \eqref{DtXiIsAntiholRemainder} and the above bound on $D_t r - \mathscr{I}$ yields $$\|D_t r - \sigma\|_{H^s} \leq C(E_s + \epsilon E_s^{1/2} + \epsilon^{5/2})$$  Hence for sufficiently small $\delta$ and $\epsilon_0$ the claim follows.$\Box$

\vspace{0.25cm}
\noindent \textbf{Claim 2.}  Given $s \geq 4$, then for $\delta$ and $\epsilon < \epsilon_0$ chosen sufficiently small we have for all $n = 0, 1, \ldots, s$ that $$\|D_t^2 r\|_{H^s} \leq C\sum_{n = 0}^s \|D_t\partial_\alpha^n\sigma\|_{L^2} + C(\epsilon + \delta)E_s^{1/2} + C\epsilon^{5/2}$$ and $$\sum_{n = 0}^s \|D_t\partial_\alpha^n \sigma\|_{L^2} \leq CE_s^{1/2} + C\epsilon^{5/2}$$
\vspace{0.25cm}

\textit{Proof of Claim 2.}  First note that for every $n = 0, 1, \ldots, s$ we have $$\partial_\alpha^n D_t^2 r - D_t \partial_\alpha^n\sigma = \partial_\alpha^n(D_t^2 r - D_t \sigma) - [b, \partial_\alpha^n] \sigma_\alpha$$  The latter term can be easily estimated by $C(E_s^{1/2} + \epsilon^{3/2})^2$ using the product rule, Claim 1, and Step 1.  Therefore it suffices to bound $D_t^2 r - D_t \sigma$ in $H^s$.  Again denote $\mathscr{I} := \frac{1}{2}D_t(I - \nht)\xi - \tilde{D}_t(I - \tilde{\nht})\tilde{\xi}$, so that $\sigma = \frac{1}{2}(I - \nht)\mathscr{I}$.  We first write
\begin{align*}
D_t^2 r - D_t \sigma & = D_t^2 r - \frac{1}{2}D_t(I - \nht)\mathscr{I} \\
& = \frac{1}{2}D_t(I - \nhtb)D_tr + \frac{1}{2}D_t(\nht + \nhtb)D_t r + \frac{1}{2}D_t(I - \nht)\left(D_t r - \mathscr{I}\right) \\
& = \frac{1}{2}D_t(I - \nhtb)D_tr + \frac{1}{2}D_t(\nht + \nhtb)D_t r - \frac{1}{2}[D_t\zeta, \nht]\frac{\partial_\alpha}{\zeta_\alpha}(D_t r - \mathscr{I}) \\ 
& \quad + \frac{1}{2}(I - \nht)\left(D_t^2 r - D_t \mathscr{I}\right)
\end{align*}  All of the terms except the last are appropriately bounded in $H^s$, by \eqref{DtXiIsAntiholRemainder}, Lemma \ref{SingIntCommuteWithDt}, Claim 1,  and Proposition \ref{SingIntSobolevEstimates}.  Hence it suffices to estimate $D_t^2 r - D_t \mathscr{I}$ in $H^s$.  We have by \eqref{dti} that 
\begin{equation*}
\begin{aligned}
D_t^2 r - D_t\mathscr{I} & = \frac{1}{2}D_t^2(\nht + \nhtb)\xi \\
& \quad - D_t(\tilde{D}_t\tilde{\xi} - \frac{1}{2}\tilde{D}_t(I - \tilde{\nht})\tilde{\xi} )\\
&\quad  - D_t((b - \tilde{b})\tilde{\xi}_\alpha)
\end{aligned}
\end{equation*}
The last two terms are controlled by $C(E_s + \epsilon E_s^{1/2} + \epsilon^{5/2})$ by Step 1, \eqref{dtbminusb} and by a multiscale calculation. 
Using Proposition \ref{HilbertCommutatorIdentities} we can write
\begin{align*}
D_t^2(\nht + \nhtb)\xi & = [D_t^2\zeta, \nht]\frac{\xi_\alpha}{\zeta_\alpha} + 2[D_t\zeta, \nht]\frac{\partial_\alpha D_t \xi}{\zeta_\alpha} - \frac{1}{\pi i}\int\left(\frac{D_t\zeta(\alpha) - D_t\zeta(\beta)}{\zeta(\alpha) - \zeta(\beta)}\right)^2 \xi_\beta(\beta) d\beta \\
& \quad + [D_t^2\overline{\zeta}, \nhtb]\frac{\xi_\alpha}{\overline{\zeta}_\alpha} + 2[D_t\overline{\zeta}, \nhtb]\frac{\partial_\alpha D_t \xi}{\overline{\zeta}_\alpha} - \frac{1}{\pi i}\int\left(\frac{D_t\overline{\zeta}(\alpha) - D_t\overline{\zeta}(\beta)}{\overline{\zeta}(\alpha) - \overline{\zeta}(\beta)}\right)^2 \xi_\beta(\beta) d\beta \\
& \quad + (\nht + \nhtb)D_t^2\xi
\end{align*}  Now we effect the usual decomposition of all of these terms.  The terms which are not of size $\mathcal{O}(\epsilon^{5/2})$ are 
\begin{align*}
& \quad \; \epsilon^2[\overline{\zeta}^{(1)}_{t_0 t_0}, \nhtb_0]\zeta^{(1)}_{\alpha_0} + 2\epsilon^2[\overline{\zeta}^{(1)}_{t_0}, \nhtb_0]\zeta^{(1)}_{\alpha_0 t_0} + \epsilon^2\overline{\nht}^{(1)}\zeta^{(1)}_{t_0 t_0} \\
& = \epsilon^2[\overline{\zeta}^{(1)}_{t_0 t_0}, \nhtb_0]\zeta^{(1)}_{\alpha_0} + 2\epsilon^2[\overline{\zeta}^{(1)}_{t_0}, \nhtb_0]\zeta^{(1)}_{\alpha_0 t_0} + \epsilon^2[\overline{\zeta}^{(1)}, \nhtb_0]\zeta^{(1)}_{\alpha_0 t_0 t_0} \\
& = 0,
\end{align*} This completes the estimate of the term $D_t^2 r - D_t \mathscr{I}$, and hence the claim.$\Box$

\vspace{0.25cm}
\noindent \textbf{Claim 3.}  Given $s \geq 4$, for $\delta$ and $\epsilon < \epsilon_0$ chosen sufficiently small, we have for all $n = 0, 1, \ldots, s$ that $$\|D_t r\|_{H^s} \leq C\sum_{n = 0}^s \|D_t\partial_\alpha^n\rho\|_{L^2} +C\left(E_s + \epsilon E_s^{1/2} + \epsilon^{5/2}\right)$$
\vspace{0.25cm}

\textit{Proof of Claim 3.}  First observe that we can write
\begin{align*}
D_t r - \frac{1}{2}D_t(I - \nht)r & = D_t r - \frac{1}{2}(I - \nht)D_t r + \frac{1}{2}[D_t\zeta, \nht]\frac{r_\alpha}{\zeta_\alpha} \\
& = \frac{1}{2}(I - \nhtb)D_t r + \frac{1}{2}(\nht + \nhtb)D_tr +\frac{1}{2}[D_t\zeta, \nht]\frac{r_\alpha}{\zeta_\alpha},
\end{align*} and thus
\begin{align*}
\partial_\alpha^n D_t r - D_t\partial_\alpha^n\rho & = \partial_\alpha^n D_t r - \frac12\partial_\alpha^n D_t(I - \nht)r -\frac12 [b, \partial_\alpha^n]\partial_\alpha(I - \nht)r\\
& = \partial_\alpha^n\left(\frac{1}{2}(I - \nhtb)D_t r + \frac{1}{2}(\nht + \nhtb)D_tr +\frac{1}{2}[D_t\zeta, \nht]\frac{r_\alpha}{\zeta_\alpha}\right) \\
& \qquad + \frac12\sum_{j = 1}^n \binom{n}{j}\left(\partial_\alpha^{j - 1}b_\alpha\right)\left(\partial_\alpha^{n - j + 1}(I - \nht)r\right)
\end{align*}  Taking the $L^2$ norm of this equation, using \eqref{DtXiIsAntiholRemainder} and summing over $n = 0, 1, \ldots, s$ yields
$$\|D_t r\|_{H^s} \leq C\sum_{n = 0}^s \|D_t\partial_\alpha^n\rho\|_{L^2} + C\left(E_s + \epsilon E_s^{1/2} + \epsilon^{5/2}\right)$$ and so the claim follows.$\Box$

\subsubsection*{Summary of Estimates}

Hence we have shown that for $s \geq 4$, there exists an $\epsilon_0 > 0$ and a $\delta > 0$ so that if \eqref{ZetaLocalAPrioriBound} holds, then for all $0 < \epsilon < \epsilon_0$, the quantity $b$ is bounded in $H^s$ by $C(E_s^{1/2} + \epsilon^{3/2})$, and the quantities $$b - \tilde{b}, \qquad \mathcal{A} - \tilde{\mathcal{A}}, \qquad D_t(b - \tilde{b}), \qquad D_t b$$ are bounded in $H^s$ by $C\left(E_s + \epsilon E_s^{1/2} + \epsilon^{5/2}\right)$, where the constant $C$ depends only on $\mathfrak{S}(T_0)$ and $\|B\|_{H^{s + 7}}$.  It is also useful to note that under the same conditions,
\begin{equation}\label{balphaBound}
\|b_\alpha\|_{H^{s - 1}} \leq \|b - \tilde{b}\|_{H^s} + \|\tilde{b}_\alpha\|_{H^{s - 1}} \leq C(E_s + \epsilon E_s^{1/2} + \epsilon^{5/2})
\end{equation}  Finally, from step 4 we have that for $\delta$ and $\epsilon_0$ sufficiently small,
\begin{equation}\label{equivalentRbound}
C_1(\|D_tr\|_{H^s}+\|D_t^2 r\|_{H^s}-\epsilon^{5/2})\le E_s^{1/2} \leq C_2(\|D_tr\|_{H^s}+\|D_t^2 r\|_{H^s}+\epsilon^{5/2});
\end{equation}
from Step 5 and \eqref{DtSquaredrEnergyBound} we have that for $\delta$ and $\epsilon_0$ sufficiently small,
\begin{equation}\label{EnergyEsEquivalence}
\begin{aligned}
E_s^{1/2} \leq C\sum_{n = 0}^s (\|D_t \partial_\alpha^n \rho\|_{L^2} + \|D_t\partial_\alpha^n \sigma\|_{L^2}) + C\epsilon^{5/2}  \\
\|\sigma\|_{H^s} + \|D_t\sigma\|_{H^s} + \sum_{n = 0}^s \|D_t\partial_\alpha^n \sigma\|_{L^2} \leq CE_s^{1/2} + C\epsilon^{5/2}
\end{aligned}
\end{equation}

\subsection{The Estimates of the Cubic Nonlinearities in the Equations for the Remainder}

Now that we have satisfactory estimates of the remainders of the auxiliary quantities, we can show that the right hand sides of \eqref{NewEulerRemainder} and \eqref{DtNewEulerRemainder} are sufficiently small to provide suitable energy estimates.  We begin by controlling the quantities appearing in the right hand side of \eqref{NewEulerRemainder}.

\begin{proposition}\label{NewEulerRemainderIsCubic}
Let $s \geq 4$ be given.  Then there exist $\epsilon_0, \delta$ so that if \eqref{ZetaLocalAPrioriBound} holds, then for all $\epsilon < \epsilon_0$, $$\|\mathcal{P}\rho\|_{H^s} \leq C\left(E_s^{3/2} + \epsilon E_s + \epsilon^2 E_s^{1/2} + \epsilon^{7/2}\right)$$ where the constant $C = C(\mathfrak{S}(T_0), \|B\|_{H^{s + 7}})$.
\end{proposition}

\begin{proof}
By \eqref{NewEulerRemainder} we must estimate the terms
\begin{align*}
& (G - \tilde{G}) - (I - \nht)(\mathcal{P} - \tilde{\mathcal{P}})\tilde{\xi} + (\nht - \tilde{\nht})\tilde{\mathcal{P}}\tilde{\xi} - [\tilde{\mathcal{P}}, \tilde{\nht}]\tilde{\xi} \notag\\
& + 2\left[D_t\zeta, \nht\frac{1}{\zeta_\alpha} + \nhtb\frac{1}{\overline{\zeta}_\alpha}\right]\partial_\alpha D_t \tilde{\xi} + 2[D_t\zeta, \nhtb]\frac{b_\alpha}{\overline{\zeta}_\alpha} + 2[D_t\zeta, \nhtb]\frac{\partial_\alpha D_t r}{\overline{\zeta}_\alpha} \notag\\
& - \frac{1}{\pi i}\int \left(\frac{D_t\zeta(\alpha) - D_t\zeta(\beta)}{\zeta(\alpha) - \zeta(\beta)}\right)^2 \tilde{\xi}_\beta(\beta) \, d\beta
\end{align*}
We estimate these terms in steps.  We make the blanket assumption that all constants $C$ may depend on $\mathfrak{S}(T_0)$ and $\|B\|_{H^{s + 7}}$.

\textbf{Step 1.}  We collect in this step terms with immediate bounds.  We have already seen though Proposition \ref{VariousMultiscaleIdentities} that $\|[\tilde{\mathcal{P}}, \tilde{\nht}]\tilde{\xi}\|_{H^s} \leq C\epsilon^{7/2}$.  We also have by Corollary \ref{DiffHilbertBound} that $$\|(\nht - \tilde{\nht})\tilde{\mathcal{P}}\tilde{\xi}\|_{H^s} \leq C(\epsilon^3 + E_s^{1/2})\|\tilde{\mathcal{P}}\tilde{\xi}\|_{H^s} \leq C\left(\epsilon^2 E_s^{1/2} + \epsilon^{7/2}\right)$$  By \eqref{DiffPFormula} and the estimates we obtained in Section~\ref{formulaforremainders}, we have
\begin{align*}
\|(I - \nht)(\mathcal{P} - \tilde{\mathcal{P}})\tilde{\xi}\|_{H^s} \leq C\left(\epsilon E_s + \epsilon^2 E_s^{1/2} + \epsilon^{7/2}\right)
\end{align*}  Next,
\begin{align*}
\left\|[D_t\zeta, \nhtb]\frac{b_\alpha}{\overline{\zeta}_\alpha}\right\|_{H^s} & = \left\|[D_t\zeta, \nhtb]\frac{\partial_\alpha(b - \tilde{b})}{\overline{\zeta}_\alpha}\right\|_{H^s} + \left\|[D_t\zeta, \nhtb]\frac{\tilde{b}_\alpha}{\overline{\zeta}_\alpha}\right\|_{H^s} \\
& \leq C\left(E_s^{1/2} + \epsilon\right)\left(E_s + \epsilon E_s^{1/2} + \epsilon^{5/2}\right) \\
& \leq C\left(E_s^{3/2} + \epsilon E_s + \epsilon^2 E_s^{1/2} + \epsilon^{7/2}\right) 
\end{align*} where as usual we estimated the former term with Proposition \ref{SingIntSobolevEstimates} and the latter term crudely in $H^s$.  By \eqref{DtXiIsAntiholRemainder}, \eqref{4.2} and Corollary \ref{DiffHilbertBound} we have
\begin{align*}
\left\|2[D_t\zeta, \nhtb]\frac{\partial_\alpha}{\overline{\zeta}_\alpha}D_t r\right\|_{H^s} & = \left\|[D_t\zeta, \nhtb]\frac{\partial_\alpha}{\overline{\zeta}_\alpha}(I - \nhtb)D_t r\right\|_{H^s} \\
& \leq C\left(E_s^{1/2} + \epsilon\right)\left(E_s + \epsilon E_s^{1/2} + \epsilon^{5/2}\right)\\
& \leq C\left(E_s^{3/2} + \epsilon E_s + \epsilon^2 E_s^{1/2} + \epsilon^{7/2}\right)
\end{align*}

\textbf{Step 2.}  Next we consider the integral $$\frac{1}{\pi i}\int \left(\frac{D_t\zeta(\alpha) - D_t\zeta(\beta)}{\zeta(\alpha) - \zeta(\beta)}\right)^2 \tilde{\xi}_\beta(\beta) \, d\beta$$  Since this integral is cubic, the only way it will contribute a term larger than $\mathcal{O}(\epsilon^{7/2})$ is if it contributes a term independent of $r$ of order $O(\epsilon^3)$.  To see that this does not occur, we decompose the integral in the same way as in Step 2 of \S 4.3.

Decomposing the differences in the numerator of the integrand by writing $$D_t\zeta = D_t r + (b - \tilde{b})\tilde\zeta_\alpha + \tilde{D}_t \tilde{\zeta}$$ yields a sum of integrals depending on $r$ or $b - \tilde{b}$ which are controlled in $H^s$ by $$C\left(\epsilon E_s + \epsilon^2 E_s^{1/2}\right),$$ as well as the following integral: $$\frac{1}{\pi i}\int \left(\frac{\tilde{D}_t\tilde{\zeta}(\alpha) - \tilde{D}_t\tilde{\zeta}(\beta)}{\zeta(\alpha) - \zeta(\beta)}\right)^2 \tilde{\xi}_\beta(\beta) \, d\beta$$  Next, decomposing the differences in the denominator of this integral via the identity $$\frac{1}{\zeta(\alpha) - \zeta(\beta)} = \frac{1}{\alpha - \beta} - \frac{\xi(\alpha) - \xi(\beta)}{\left(\zeta(\alpha) - \zeta(\beta)\right)\left(\alpha - \beta\right)}$$ yields a sum of integrals controlled in $H^s$ by $$C\epsilon^3\left(E_s^{1/2} + \epsilon^{1/2}\right)$$ along with the integral $$\frac{1}{\pi i}\int \left(\frac{\tilde{D}_t\tilde{\zeta}(\alpha) - \tilde{D}_t\tilde{\zeta}(\beta)}{\alpha - \beta}\right)^2 \tilde{\xi}_\beta(\beta) \, d\beta$$  Expanding $\tilde{D}_t \tilde{\zeta}$ and $\tilde{\xi}_\alpha$ in powers of $\epsilon$ and collecting like powers yields a sum of integrals controlled by $C\epsilon^{7/2}$ except for the leading term of size $O(\epsilon^3)$ given by the integral 
$$\frac{\epsilon^3}{\pi i}\int \left(\frac{\zeta^{(1)}_{t_0}(\alpha) - \zeta^{(1)}_{t_0}(\beta)}{\alpha - \beta}\right)^2 \zeta^{(1)}_{\beta_0}(\beta) \, d\beta = 2\epsilon^3[\zeta^{(1)}_{t_0}, \nht_0](\zeta^{(1)}_{t_0 \alpha_0} \zeta^{(1)}_{\alpha_0}) - \epsilon^3[\zeta^{(1)}_{t_0}, [\zeta^{(1)}_{t_0}, \nht_0]]\zeta^{(1)}_{\alpha_0 \alpha_0},$$ which is also controlled by $C\epsilon^{7/2}$ by Corollary \ref{CommutatorPhase}.

\textbf{Step 3.}  We turn to the term $$\left[D_t\zeta, \nht\frac{1}{\zeta_\alpha} + \nhtb\frac{1}{\overline{\zeta}_\alpha}\right]\partial_\alpha D_t \tilde{\xi} = -\frac{2}{\pi}\int \frac{\left(D_t\zeta(\alpha) - D_t\zeta(\beta)\right)\left(\Im \zeta(\alpha) - \Im \zeta(\beta)\right)}{|\zeta(\alpha) - \zeta(\beta)|^2} \partial_\beta D_t \tilde{\xi}(\beta) \, d\beta$$  Decomposing the differences in the numerator of the integral as in Step 2 yields a sum of singular integrals.  All but one of these singular integrals depends on $r$ and are controlled in $H^s$ by $$C\left(\epsilon E_s + \epsilon^2 E_s^{1/2} + \epsilon^{7/2}\right)$$  The remaining singular integral is given by $$\frac{2}{\pi} \int \frac{(\tilde{D}_t \tilde{\zeta}(\alpha) - \tilde{D}_t \tilde{\zeta}(\beta))(\Im \tilde{\xi}(\alpha) - \Im \tilde{\xi}(\beta))}{(\alpha - \beta)^2} \partial_\beta \tilde{D}_t \tilde{\xi} \, d\beta,$$ of which the leading term is isolated by expanding $\tilde{\zeta} = \alpha + \epsilon \zeta^{(1)} + \epsilon^2 \zeta^{(2)} + \epsilon^2 \zeta^{(3)}$, yielding $$\frac{2}{\pi}\epsilon^3 \int \frac{(\zeta^{(1)}_{t_0}(\alpha) -  \zeta^{(1)}_{t_0}(\beta))(\Im \zeta^{(1)}(\alpha) - \Im \zeta^{(1)}(\beta))}{(\alpha - \beta)^2} \zeta^{(1)}_{t_0 \beta_0} (\beta) \, d\beta$$  By the same calculation in \S 3.3 showing that the $I_1$ term of $G_3$ vanished, we see that this leading term  is actually $O(\epsilon^4)$ by Corollary \ref{CommutatorPhase}.  Therefore only terms of size $O(\epsilon^4)$ appear, and so we have that $$\left\|\left[D_t\zeta, \nht\frac{1}{\zeta_\alpha} + \nhtb\frac{1}{\overline{\zeta}_\alpha}\right]\partial_\alpha D_t \tilde{\xi}\,\right\|_{H^s} \leq C(\epsilon E_s + \epsilon^2 E_s^{1/2} + \epsilon^{7/2})$$

Similarly, the same method of decomposition allows us to expand $G$ in the same way, until the leading term of the part of the decomposition that is independent of $r$ is apparent.  However, this leading term of size $O(\epsilon^3)$ is by construction equal to $\tilde{G}$, with which it cancels.  Therefore $G - \tilde{G}$ and hence the whole right hand side of \eqref{NewEulerRemainder} is bounded in $H^s$ by $C(E_s^{3/2} + \epsilon E_s + \epsilon^2 E_s^{1/2} + \epsilon^{7/2})$.
\end{proof}

Next we consider the right hand side of \eqref{DtNewEulerRemainder}.

\begin{proposition}\label{DtNewEulerRemainderIsCubic}
Let $s \geq 4$ be given.  Then there exist $\epsilon_0 > 0$ and $\delta > 0$ so that if \eqref{ZetaLocalAPrioriBound} holds, then for all $\epsilon < \epsilon_0$, $$\|\mathcal{P}\sigma\|_{H^s} \leq C\left(E_s^{3/2} + \epsilon E_s + \epsilon^2 E_s^{1/2} + \epsilon^{7/2}\right),$$ where the constant $C = C(\mathfrak{S}(T_0), \|B\|_{H^{s + 7}})$.
\end{proposition}

\begin{proof}
It suffices to show that the following terms are $\mathcal{O}(\epsilon^{7/2})$:
\begin{align*}
& -8[D_t\zeta, \nht]\frac{\partial_\alpha D_t \sigma}{\zeta_\alpha} + \frac{4}{\pi i} \int \left(\frac{D_t\zeta(\alpha) - D_t\zeta(\beta)}{\zeta(\alpha) - \zeta(\beta)}\right)^2 \sigma_\beta(\beta) d\beta \notag\\
& \quad + (I - \nht)iU_{\kappa^{-1}}\left(\frac{\mathfrak{a}_t}{\mathfrak{a}}\right) \partial_\alpha(I - \nht)\xi \notag\\
& \quad - (I - \nht)(\mathcal{P} - \tilde{\mathcal{P}})\tilde{D}_t(I - \tilde{\nht})\tilde{\xi} \notag\\
& \quad +i (I - \nht)\tilde{b}_\alpha\partial_\alpha(I - \tilde{\nht})\tilde{\xi} \notag\\
& \quad + (I - \nht)(D_t G - \tilde{D}_t\tilde{G}) - (I - \nht)\epsilon^4 (\tilde{D}_t R) \notag\\
& := I_1 + I_2 + I_3 +I_4 + I_5 + I_6 +I_7.
\end{align*}
Clearly $\|I_7\|_{H^s} \leq C\epsilon^{7/2}$ and $\|I_5\|_{H^s} \leq C\epsilon^{7/2}$.  By \eqref{DiffPFormula}, and the estimates of \S 4.3 we have that 
\begin{align*}
\|I_4\|_{H^s} & \leq C(E_s + \epsilon E_s^{1/2} + \epsilon^{5/2}) \epsilon \\
& \leq C(E_s^{3/2} + \epsilon E_s + \epsilon^2 E_s^{1/2} + \epsilon^{7/2})
\end{align*}
Using Lemma \ref{SingIntCommuteWithDt} along with Proposition \ref{SingIntSobolevEstimates}, we can decompose $D_t G$ into a sum of singular integrals as in Step 2 of \S 4.3.  Each of these integrals can be bounded by $C(E_s^{3/2} + \epsilon E_s + \epsilon^2 E_s^{1/2})$ except for $\tilde{D}_t \tilde{G}$, which has leading term of size $O(\epsilon^3)$; but then $I_6$ is $\mathcal{O}(\epsilon^{7/2})$.  Similarly, if we effect the usual decomposition on the right hand side of the formula \eqref{atOveraFormula}, we see that the only term not of size $\mathcal{O}(\epsilon^{5/2})$ is the term
$$\epsilon^2 2i\left([\zeta^{(1)}_{t_0 t_0}, \nht_0]\overline{\zeta}^{(1)}_{\alpha_0 t_0} + [\zeta^{(1)}_{t_0}, \nht_0]\overline{\zeta}^{(1)}_{\alpha_0 t_0 t_0}\right) = 0$$ and hence that $\|I_3\|_{H^s} \leq C(E_s^{3/2} + \epsilon E_s + \epsilon^2 E_s^{1/2} + \epsilon^{7/2})$.  By Step 5 of \S 4.3 and Proposition \ref{SingIntSobolevEstimates} we estimate $I_2$ as
\begin{align*}
\|I_2\|_{H^s} & \leq C(E_s^{1/2} + \epsilon)^2\|\sigma\|_{H^s} \\
& \leq C(E_s^{1/2} + \epsilon)^2(E_s^{1/2}+ \epsilon^{5/2}) \\
& \leq C(E_s^{3/2} + \epsilon E_s + \epsilon^2 E_s^{1/2} + \epsilon^{7/2})
\end{align*}  The only term left to estimate is $I_1$.  We first write
$$2[D_t\zeta, \nht]\frac{\partial_\alpha D_t \sigma}{\zeta_\alpha} = 2[D_t\zeta, \nht]\frac{\partial_\alpha D_t^2 r}{\zeta_\alpha} + 2[D_t\zeta, \nht]\frac{\partial_\alpha (D_t \sigma - D_t^2 r)}{\zeta_\alpha},$$ and by Step 5 of \S 4.3 we have that the latter term is bounded by $C(E_s^{3/2} + \epsilon E_s + \epsilon^2 E_s^{1/2} + \epsilon^{7/2})$ in $H^s$.  Next we have
$$2[D_t\zeta, \nht]\frac{\partial_\alpha D_t^2 r}{\zeta_\alpha} = 2\left[D_t\zeta, \nht\frac{1}{\zeta_\alpha} + \nhtb\frac{1}{\overline{\zeta}_\alpha}\right]\partial_\alpha D_t^2 r - 2[D_t\zeta, \nhtb]\frac{\partial_\alpha D_t^2 r}{\overline{\zeta}_\alpha},$$ and the former term is bounded in $H^s$ by $C(E_s^{3/2} + \epsilon E_s + \epsilon^2 E_s^{1/2})$.  Of the latter term we write using Proposition \ref{HilbertCommutatorIdentities} that
$$2[D_t\zeta, \nhtb]\frac{\partial_\alpha D_t^2 r}{\overline{\zeta}_\alpha} = [(I + \nhtb)D_t\zeta, \nhtb]\frac{\partial_\alpha D_t^2 r}{\overline{\zeta}_\alpha} = [D_t\zeta, \nhtb]\frac{\partial_\alpha }{\overline{\zeta}_\alpha}(I - \nhtb)D_t^2 r$$  Finally, we have by \eqref{DtXiIsAntiholRemainder} that
\begin{align*}
(I - \nhtb)D_t^2 r & = [D_t \zeta, \nhtb]\frac{\partial_\alpha D_t r}{\bar \zeta_\alpha} + D_t\left(-(I - \tilde{\nht})\tilde{D}_t\overline{\tilde{\zeta}} - (I - \tilde{\nht})(D_t - \tilde{D}_t)\overline{\tilde{\zeta}} + (\nht - \tilde{\nht}){D}_t\overline{\tilde{\zeta}}\right)
\end{align*}
Therefore $\|(I - \nhtb)D_t^2 r\|_{H^s} \leq C(E_s + \epsilon E_s^{1/2} + \epsilon^{5/2})$, from which the Proposition follows.
\end{proof}

\subsection{Construction of the Energy for the Remainder}

In this section we construct the energy corresponding to the equations \eqref{NewEulerRemainder} and \eqref{DtNewEulerRemainder}.  We then show that this energy obeys a differential inequality which yields a priori bounds on a $O(\epsilon^{-2})$ time scale.  The energy so constructed will control the quantity $\|D_t r\|_{H^s}^2 + \|D_t^2 r\|_{H^s}^2$, and hence by \eqref{DtSquaredrEnergyBound} it follows that for sufficiently small energies also yields suitable bounds on $E_s$.

\subsubsection*{Bounds on the Equations for the Derivatives}

We must first show that the nonlinearities in the corresponding equations for the derivatives are appropriately bounded in $L^2$.
\begin{proposition}\label{NewEulerRemainderDerivBound}
Let $s \geq 4$ and $1\le n\le s$ be given.  Then there exist $\epsilon_0 > 0$ and $\delta > 0$ so that if \eqref{ZetaLocalAPrioriBound} holds, then for all $\epsilon < \epsilon_0$, if $\Theta = \rho, \sigma$, then
$$\|\mathcal{P}\partial_\alpha^n\Theta\|_{L^2} \leq C(E_s^{3/2} + \epsilon E_s + \epsilon^2 E_s^{1/2} + \epsilon^{7/2})$$ where $C$ depends only on $\mathfrak{S}(T_0)$ and $\|B\|_{H^{s + 7}}$.
\end{proposition}

\begin{proof}
Let $\Theta = \rho, \sigma$ as above.  Observe that for any $n \geq 1$ we can write $$\mathcal{P}\partial_\alpha^n\Theta = \partial_\alpha^n\mathcal{P}\Theta - \sum_{j = 1}^n \partial_\alpha^{n - j}[\partial_\alpha, \mathcal{P}]\partial_\alpha^{j - 1}\Theta$$  Using the identity 
\begin{equation}\label{PartialAlphaCommuteP}
[\partial_\alpha, \mathcal{P}] = \left\{\partial_\alpha\bigl(D_t b - i(\mathcal{A} - 1)\bigr)\right\}\partial_\alpha + 2b_\alpha D_t \partial_\alpha
\end{equation} we rewrite as
\begin{align*}
\mathcal{P}\partial_\alpha^n\Theta = \partial_\alpha^n\mathcal{P}\Theta & - \sum_{j = 1}^n \partial_\alpha^{n - j}\Bigl(\partial_\alpha\bigl(D_t b - i(\mathcal{A} - 1)\bigr)\partial_\alpha^j\Theta\Bigr) \\
& - 2\sum_{j = 1}^n \partial_\alpha^{n - j}\bigl(b_\alpha D_t \partial_\alpha^j\Theta\bigr)
\end{align*}  Now using the identity
\begin{align}\label{CommuteDtPastPartialAlpha}
D_t \partial_\alpha^j & = \partial_\alpha^j D_t - \sum_{l = 1}^j \partial_\alpha^{j - l}[\partial_\alpha, D_t]\partial_\alpha^{l - 1} \notag\\
& = \partial_\alpha^j D_t - \sum_{l = 1}^j \partial_\alpha^{j - l}( b_\alpha\partial_\alpha^l) 
\end{align}
we have by the product rule, Steps 2 and 3 of \S 4.3, \eqref{balphaBound} and Proposition \ref{NewEulerRemainderIsCubic} that for all $1 \leq n \leq s$,
\begin{align*}
\|\mathcal{P}\partial_\alpha^n\Theta\|_{L^2} & \leq C\|\mathcal{P}\Theta\|_{H^s} \\
& \quad + C\|D_t b - i(\mathcal{A} - 1)\|_{H^s}\|\partial_\alpha\Theta\|_{H^{s - 1}} \\
& \quad + C\|b_\alpha\|_{H^{s - 1}}(\|D_t\Theta\|_{H^s}+\|\partial_\alpha\Theta\|_{H^{s - 1}} ) \\
& \leq C\left(E_s^{3/2} + \epsilon E_s + \epsilon^2 E_s^{1/2} + \epsilon^{7/2}\right)
\end{align*} where the last inequality follows from Step 5 of \S 4.3.
\end{proof}

\subsubsection*{Construction of the Energy and the Energy Inequality}

Now that we have shown that the equations for the derivatives of the quantities in \eqref{NewEulerRemainder} and \eqref{DtNewEulerRemainder} also have $\mathcal{O}(\epsilon^{7/2})$ nonlinearities, we can construct the energies corresponding to these equations.  Since the principal operator of \eqref{NewEulerRemainder} and \eqref{DtNewEulerRemainder} is $\mathcal{P}$, we can use the same construction given by Lemma 4.1 of \cite{WuAlmostGlobal2D} to construct our energy; we record this lemma here for convenience.

\begin{proposition}[c.f. Lemma 4.1 of \cite{WuAlmostGlobal2D}]\label{BasicEnergyInequality}
Suppose that a function $\Theta \in C^0([0, T]; \dot{H}^{1/2}) \cap C^1([0, T]; L^2)$ is given satisfying $\mathcal{P}\Theta = \mathscr{G}$.  Define $$\mathfrak{E}(t) := \int \frac{1}{\mathcal{A}}|D_t\Theta(\alpha, t)|^2 + i\Theta(\alpha, t)\overline{\Theta}_\alpha(\alpha, t) d\alpha $$  Then
$$\frac{d\mathfrak{E}}{dt} = \int \frac{2}{\mathcal{A}}\Re\left(\mathscr{G}D_t\overline{\Theta}\right) - \frac{1}{\mathcal{A}}U_\kappa^{-1} \left(\frac{\mathfrak{a}_t}{\mathfrak{a}}\right)|D_t\Theta|^2 d\alpha$$  Moreover if $\Theta$ is the trace of a holomorphic function on $\Omega(t)^c$, i.e., if $\Theta = \frac{1}{2}(I - \nht)\Theta$, then $$\int i\Theta\overline{\Theta}_\alpha d\alpha = -\int i\overline{\Theta}\Theta_\alpha d\alpha \geq 0$$
\end{proposition}

For brevity, we introduce the quantities
\begin{equation}\label{DefinitionRhoNSigmaN}
\rho^{(n)} := \partial_\alpha^n\rho \qquad \text{and} \qquad \sigma^{(n)} := \partial_\alpha^n\sigma
\end{equation}
We cannot use the second part of Proposition \ref{BasicEnergyInequality} directly for $n > 0$ since $\rho^{(n)}$ and $\sigma^{(n)}$ need not be the trace of a holomorphic function on $\Omega(t)^c$.  Hence we further introduce the notation
\begin{align}\label{RhoSigmaHoloPart}
\rho^{(n)} = \frac{1}{2}(I - \nht)\rho^{(n)} + \frac{1}{2}(I + \nht)\rho^{(n)} := \phi^{(n)} + \mathcal{R}^{(n)} \notag\\
\sigma^{(n)} = \frac{1}{2}(I - \nht)\sigma^{(n)} + \frac{1}{2}(I + \nht)\sigma^{(n)} := \psi^{(n)} + \mathcal{S}^{(n)}
\end{align}
Consider now the case $0 \leq n \leq s$.  Define
\begin{equation}\label{RemainderEnergyFormulaE}
\mathcal{E}_n(t) = \int \frac{1}{\mathcal{A}}|D_t\rho^{(n)}|^2 + i \phi^{(n)} \overline{\phi}^{(n)}_\alpha \,d\alpha
\end{equation} and
\begin{equation}\label{RemainderEnergyFormulaF}
\mathcal{F}_n(t) = \int \frac{1}{\mathcal{A}}|D_t\sigma^{(n)}|^2 + i \sigma^{(n)} \overline{\sigma}^{(n)}_\alpha \,d\alpha
\end{equation}
We must show that the parts contributed by $\frac{d\mathcal{E}_n}{dt}$ to the energy inequality by the parts of these terms that are antiholomorphic in $\Omega(t)^c$ are at most of size $\mathcal{O}(\epsilon^5)$.

Observe first that we can write
\begin{align}\label{SmallAntiholPart}
\mathcal{R}^{(n)} & = \frac{1}{2}(I + \nht)\partial_\alpha^n\rho \notag\\
& = \frac{1}{4}\partial_\alpha^n(I + \nht)(I - \nht)r - \frac{1}{2}\sum_{j = 1}^n \partial_\alpha^{n - j}[\partial_\alpha, \nht]\partial_\alpha^{j - 1}\rho \notag\\
& = - \frac{1}{2}\sum_{j = 1}^n \partial_\alpha^{n - j}[\zeta_\alpha - 1, \nht]\frac{\partial_\alpha^j\rho}{\zeta_\alpha}
\end{align} and so $\mathcal{R}^{(n)}$ is bounded by $C(E_s + \epsilon E_s^{1/2} + \epsilon^{5/2})$ in $L^2$.  Writing $\phi^{(n)} = \rho^{(n)} - \mathcal{R}^{(n)}$ in  $\mathcal{E}_n$ yields
\begin{align*}
\mathcal{E}_n & = \int \frac{1}{\mathcal{A}}|D_t \rho^{(n)}|^2 + i \rho^{(n)} \overline{\rho}_\alpha^{(n)} \, d\alpha - i \int \phi^{(n)}\overline{\mathcal{R}}^{(n)}_\alpha + \mathcal{R}^{(n)} \overline{\phi}^{(n)}_\alpha + \mathcal{R}^{(n)}\overline{\mathcal{R}}^{(n)}_\alpha \, d\alpha
\end{align*}
Differentiating this with respect to $t$ and integrating by parts yields
\begin{align}\label{EnergyInequalityV1}
\frac{d\mathcal{E}_n}{dt} & = \int \frac{2}{\mathcal{A}}\Re\left(D_t \overline{\rho}^{(n)}\mathcal{P}\rho^{(n)}\right) - \frac{1}{\mathcal{A}} U_\kappa^{-1} \left(\frac{\mathfrak{a}_t}{\mathfrak{a}}\right)|D_t \rho^{(n)}|^2 \, d\alpha \notag\\
& \qquad + 2\Im \int \mathcal{R}^{(n)}_t\overline{\phi}^{(n)}_\alpha + \phi^{(n)}_t \overline{\mathcal{R}}^{(n)}_\alpha + \mathcal{R}^{(n)}_t\overline{\mathcal{R}}^{(n)}_\alpha \, d\alpha
\end{align}  We want to show that the right hand side of this inequality is $\mathcal{O}(\epsilon^5)$.  By Proposition~\ref{NewEulerRemainderDerivBound} and \eqref{atOveraFormula} it is clear that the first integral is $\mathcal{O}(\epsilon^5)$, and so it suffices to show that the second integral is of size $\mathcal{O}(\epsilon^{5})$.

The arguments for handling the first two terms rely on the fact that $\phi^{(n)}$ and $\mathcal{R}^{(n)}$ are almost orthogonal in $L^2$, and so the arguments showing these terms are small are similar to each other, so we will only consider the term $\mathcal{R}^{(n)}_t\overline{\phi}^{(n)}_\alpha$.  We have
\begin{equation*}
\mathcal{R}^{(n)}_t = \frac{1}{2}\partial_t(I + \nht) \mathcal{R}^{(n)} = \frac{1}{2}(I + \nht)\mathcal{R}^{(n)}_t + [\zeta_t, \nht]\frac{\partial_\alpha \mathcal{R}^{(n)}}{\zeta_\alpha}
\end{equation*} and since the latter term is $\mathcal{O}(\epsilon^{5/2})$, it suffices to consider only the former term.  Likewise, recalling that the adjoint\footnote{The adjoint $T^*$ of a linear operator $T : L^2 \to L^2$ is defined by $\int f \, T^*(g) \, d\alpha = \int g \, T(f) \, d\alpha$ for all $f, g \in L^2$.} $\nht^*$ of the Hilbert transform satisfies the identity $\nht^*f = -\zeta_\alpha\nht(f/\zeta_\alpha)$, the identity $[\nht, \partial_\alpha/\zeta_\alpha] = 0$ of Proposition \ref{HilbertCommutatorIdentities} implies that $\partial_\alpha \nht = -\nht^* \partial_\alpha$, and so we can write $\overline{\phi}_\alpha^{(n)}$ as
\begin{equation*}
\overline{\phi}_\alpha^{(n)} = \frac{1}{2}\partial_\alpha(I - \nhtb)\partial_\alpha^n\overline{\rho} = \frac{1}{2}(I + \nhtb^*)\partial_\alpha^{n + 1}\overline{\rho}
\end{equation*}
But now, using the usual $L^2$ pairing\footnote{Here we use the real inner product $\langle f, g\rangle = \int f\,g\, d\alpha$ for $f, g \in L^2$.} $\langle, \rangle$, we have
\begin{align*}
\frac{1}{4}\left\langle (I + \nht)\mathcal{R}^{(n)}_t, (I + \nhtb^*)\partial_\alpha^{n + 1}\overline{\rho} \right\rangle & = \frac{1}{4}\left\langle \nht\mathcal{R}^{(n)}_t, (I + \nhtb^*)\partial_\alpha^{n + 1}\overline{\rho} \right\rangle \\
& \; + \frac{1}{4}\left\langle \mathcal{R}^{(n)}_t, \nhtb^*(I + \nhtb^*)\partial_\alpha^{n + 1}\overline{\rho} \right\rangle \\
& = \frac{1}{2}\left\langle (\nht + \nhtb)\mathcal R^{(n)}_t, \overline{\phi}_\alpha^{(n)} \right\rangle
\end{align*} Therefore
\begin{align*}
\int \mathcal{R}^{(n)}_t\overline{\phi}^{(n)}_\alpha d\alpha = \int \frac{1}{2}\bar \phi_\alpha^{(n)}(\nht + \nhtb)\mathcal{R}^{(n)}_t d\alpha + \int \overline{\phi}^{(n)}_\alpha[\zeta_t, \nht]\frac{\partial_\alpha \mathcal{R}^{(n)}}{\zeta_\alpha} d\alpha,
\end{align*} and so these integrals are bounded by $$CE_s^{1/2}(E_s^{1/2} + \epsilon)(E_s + \epsilon E_s^{1/2} + \epsilon^{5/2}) \leq C(E_s^2 + \epsilon E_s^{3/2} + \epsilon^2 E_s + \epsilon^{7/2})$$  From \eqref{SmallAntiholPart}, estimating as usual gives bound of $\mathcal{R}^{(n)}_\alpha$ and $\mathcal{R}_t^{(n)}$ in $L^2$ of $C(E_s + \epsilon E_s^{1/2})$ and $C(E_s + \epsilon E_s^{1/2} + \epsilon^{5/2})$, respectively.  Summing these bounds, we have that \eqref{EnergyInequalityV1} reads
\begin{equation*}
\frac{d\mathcal{E}_n}{dt} \leq C(E_s^2 + \epsilon E_s^{3/2} + \epsilon^2 E_s + \epsilon^{7/2}E_s^{1/2})
\end{equation*}  If we try to apply the same argument to $\mathcal{F}_n$ as we just did to $\mathcal{E}_n$, we find that $\frac{d\mathcal{F}_n}{dt}$ has an extra half-derivative than can be controlled by the energy, since $\mathcal{F}_n$ consists of quantities with one time derivative more than the quantities comprising $\mathcal{E}_n$.  Now since
$$\frac{d\mathcal{F}_n}{dt} = \int \frac{2}{\mathcal{A}}\Re\left(D_t \overline{\sigma}^{(n)}\mathcal{P}\sigma^{(n)}\right) - \frac{1}{\mathcal{A}} U_\kappa^{-1} \left(\frac{\mathfrak{a}_t}{\mathfrak{a}}\right)|D_t \sigma^{(n)}|^2 \, d\alpha,$$ which by Step 5 of \S 4.3 and Proposition \ref{NewEulerRemainderDerivBound} implies
$$\frac{d\mathcal{F}_n}{dt} \leq C(E_s^2 + \epsilon E_s^{3/2} + \epsilon^2 E_s + \epsilon^{7/2} E_s^{1/2} + \epsilon^6).$$ Hence we need only show that the quantity $\mathcal{F}_n$ itself is bounded below by $\|D_t\sigma^{(n)}\|_{L^2}^2$ up to a term of size $\mathcal{O}(\epsilon^5)$, for $n = 0, \ldots, s$.  By writing $\sigma^{(n)} = \psi^{(n)} + \mathcal{S}^{(n)}$ we can use Proposition \ref{BasicEnergyInequality} to estimate
\begin{align*}
\mathcal{F}_n & = \int \frac{1}{\mathcal{A}}|D_t\sigma^{(n)}|^2 + i\sigma^{(n)}\overline{\sigma}^{(n)}_\alpha d\alpha \\
& \geq \int \frac{1}{\mathcal{A}}|D_t\sigma^{(n)}|^2 d\alpha - \left|\int \psi^{(n)}\overline{\mathcal{S}}^{(n)}_\alpha + \mathcal{S}^{(n)} \overline{\psi}^{(n)}_\alpha + \mathcal{S}^{(n)}\overline{\mathcal{S}}^{(n)}_\alpha d\alpha \right|
\end{align*}  Now, as with $\mathcal{R}^{(n)}$, we can rewrite
$$\mathcal{S}^{(n)} = -\frac{1}{2}(I + \nht)\sum_{j = 1}^n \partial_\alpha^{n - j}[\zeta_\alpha - 1, \nht]\frac{\partial_\alpha}{\zeta_\alpha}\partial_\alpha^{j - 1}\sigma$$
$$\mathcal{S}^{(n)}_\alpha = -\frac{1}{2}(I - \nht^*)\sum_{j = 1}^n \partial_\alpha^{n - j + 1}[\zeta_\alpha - 1, \nht]\frac{\partial_\alpha}{\zeta_\alpha}\partial_\alpha^{j - 1}\sigma$$  From the above formula for $\psi^{(n)}$ we see that $\psi^{(n)}_\alpha$ has one more spatial derivative than the energy provides.  However, if we integrate by parts and use Step 5 of \S 4.3, we can estimate
\begin{align*}
\mathcal{F}_n & \geq \int \frac{1}{\mathcal{A}}|D_t\sigma^{(n)}|^2 d\alpha - \left|\int \psi^{(n)}\overline{\mathcal{S}}^{(n)}_\alpha - \mathcal{S}^{(n)}_\alpha \overline{\psi}^{(n)} + \mathcal{S}^{(n)}\overline{\mathcal{S}}^{(n)}_\alpha d\alpha \right| \\
& \geq \int \frac{1}{\mathcal{A}}|D_t\sigma^{(n)}|^2 d\alpha - C\delta( E_s^{1/2} + \epsilon^{5/2})^2 \\
& \geq \int \frac{1}{\mathcal{A}}|D_t\sigma^{(n)}|^2 d\alpha - C\delta(E_s + \epsilon^5)
\end{align*}
If we set
\begin{equation}\label{FullEnergyDefinition}
\mathcal{E} = \sum_{n = 0}^s (\mathcal{E}_n + \mathcal{F}_n)
\end{equation} 
and if we choose  $\delta$ sufficiently small, then we have by \eqref{EnergyEsEquivalence} that
\begin{align*}
E_s^{1/2} 
& \leq C\mathcal{E}^{1/2} + C\epsilon^{5/2}.
\end{align*}  Thus if we choose $\epsilon_0$ and $\delta$ still smaller, we have from the inequality $$\sum_{n = 0}^s \left(\frac{d\mathcal{E}_n}{dt} + \frac{d\mathcal{F}_n}{dt}\right) \leq C(E_s^2 + \epsilon E_s^{3/2} + \epsilon^2 E_s + \epsilon^{7/2} E_s^{1/2}+\epsilon^6)$$ that the following lemma is demonstrated:
\begin{lemma}\label{EnergyInequality}
Let $\mathcal{E}$ be defined as in \eqref{FullEnergyDefinition}.  Then there exists an $\epsilon_0 > 0$ and a $\delta > 0$ so that if \eqref{ZetaLocalAPrioriBound} holds, then there is a constant $C = C(\epsilon_0, \delta)$ so that for all $\epsilon < \epsilon_0$, 
\begin{enumerate}
\item[(1)]{$E_s^{1/2} \leq C(\mathcal{E}^{1/2} + \epsilon^{5/2})$}
\item[(2)]{$\frac{d\mathcal{E}}{dt} \leq C(\mathcal{E}^2 + \epsilon \mathcal{E}^{3/2} + \epsilon^2 \mathcal{E} + \epsilon^{7/2}\mathcal{E}^{1/2} + \epsilon^6)$}
\end{enumerate} where the constants $C$ depend only on $\epsilon_0$ and $\delta$.
\end{lemma}

\subsubsection*{A Priori Bounds on the Remainder Energy}

Now we can derive a priori bounds from the energy inequality derived in the last section.

\begin{proposition}\label{RemainderEnergyAPrioriBound}
Let $s \geq 4$, $\mathscr{T}$, $B_0\in H^{s+7}$ be given, let $\epsilon_0$, $\delta$  be given.  Let $T_0$ be a time so that \eqref{ZetaLocalAPrioriBound} hold.  Suppose further that $\mathcal{E}(0) = M_0^2\epsilon^3$.  Then there is a possibly smaller $\epsilon_0 = \epsilon_0(\mathscr{T}, M_0, \delta, \|B_0\|_{H^{s + 7}})$ so that for all $0 < \epsilon < \epsilon_0$ and $0 \leq t \leq \min(T_0, \epsilon^{-2}\mathscr{T})$ we have $\mathcal{E}(t) \leq C\epsilon^3$, where the constant $C = C(\mathscr{T}, M_0, \delta, \|B_0\|_{H^{s + 7}})$.
\end{proposition}

\begin{proof}
Let $C_0$ be the constant appearing in Lemma~\ref{EnergyInequality}.  Define $\mathcal{S}(T) = \sup_{0 \leq t \leq T} \mathcal{E}(t)$.  Then for any $T \in [0, \min(T_0, \epsilon^{-2}\mathscr{T})]$ we have for all $t \in [0, T]$ that
\begin{align*}
\frac{d\mathcal{E}}{dt}(t) & \leq C_0\left(\mathcal{E}^2(t) + \epsilon \mathcal{E}^{3/2}(t) + \epsilon^2 \mathcal{E}(t) + \epsilon^{7/2} \mathcal{E}^{1/2}(t) + \epsilon^6\right) \\
& \leq C_0\left(\mathcal{S}(T) + \epsilon \mathcal{S}(T)^{1/2} + \epsilon^2\right)\mathcal{E}(t) + C_0(\epsilon^{7/2} \mathcal{S}(T)^{1/2} + \epsilon^6)
\end{align*}  Solving this differential inequality for $0 \leq t \leq T$ gives 
$$\mathcal{E}(t) \leq \left(\mathcal{E}(0) + \frac{\epsilon^{7/2}\mathcal{S}(T)^{1/2} + \epsilon^6}{\mathcal{S}(T) + \epsilon \mathcal{S}(T)^{1/2} + \epsilon^2}\right)e^{C_0(\mathcal{S}(T) + \epsilon \mathcal{S}(T)^{1/2} + \epsilon^2)t}$$ and so taking the supremum over $[0, T]$ gives for all $T \leq \min(T_0, \epsilon^{-2}\mathscr{T})$ that \begin{equation}\label{EnergySupBound}
\mathcal{S}(T) \leq \left(\mathcal{E}(0) + \frac{\epsilon^{7/2}\mathcal{S}(T)^{1/2} + \epsilon^6}{\mathcal{S}(T) + \epsilon \mathcal{S}(T)^{1/2} + \epsilon^2}\right)e^{C_0(\mathcal{S}(T) + \epsilon \mathcal{S}(T)^{1/2} + \epsilon^2)T}
\end{equation}

We now begin a continuity argument.  Let $M_1$ be the positive root of the equation $\frac{1}{2}e^{-3C_0\mathscr{T}}M_1 = M_0 + \sqrt{M_1} + 1$.  If $\mathcal{S}(\min(T_0, \epsilon^{-2}\mathscr{T})) \leq M_1 \epsilon^3$ then we are done.  If not, let $T^* < \min(T_0, \epsilon^{-2}\mathscr{T})$ be the first time at which $S(T^*) = M_1\epsilon^3$.  Choose $\epsilon_0$ so that $\epsilon_0 M_1 \leq 1$.  Then we have from \eqref{EnergySupBound} that 
\begin{align*}
S(T^*) & \leq \left(\mathcal E(0) + \frac{\epsilon^5 \sqrt{M_1} + \epsilon^6}{\epsilon^2}\right)e^{C_0(M_1\epsilon^3 + \sqrt{M_1}\epsilon^{5/2} + \epsilon^2)\epsilon^{-2}\mathscr{T}} \\
& \leq (M_0 + \sqrt{M_1} + 1)e^{3C_0\mathscr{T}}\epsilon^3 \\
& \leq \frac{1}{2}M_1\epsilon^3,
\end{align*} which contradicts the definition of $T^*$.
\end{proof}

\section{Long time existence of wave packet-like solutions}

We would like to show that for wave packet-like data, the solution of the water wave system \eqref{OldEuler}-\eqref{ztIsAntihol} exists on the $O(\epsilon^{-2})$ time scale, and is well approximated by the wave packet whose modulation evolves according to NLS.  Thus far we have found a globally existing approximation $\tilde{\zeta}$, as well as a suitable a priori bound on the energy of the remainder $r$ for $O(\epsilon^{-2})$ time scales.  Since $\tilde \zeta$ does not in general satisfy the water wave system, the wave packet data $(\tilde\zeta(0), \tilde D_t\tilde \zeta(0), \tilde D_t^2\tilde \zeta(0))$ cannot be taken as the initial data for the water wave system \eqref{OldEuler}-\eqref{ztIsAntihol}. 

In what follows, we will show that there is data for the water wave system that is within $\mathcal O(\epsilon^{3/2})$ to the wave packet $(\tilde\zeta(0), \tilde D_t\tilde \zeta(0), \tilde D_t^2\tilde \zeta(0))$. Moreover for all such data, the solution of the system \eqref{OldEuler}-\eqref{ztIsAntihol} exists on the $O(\epsilon^{-2})$ time scale. The a priori bound on $r$ gives the estimate of the error between $\zeta$ and the wave packet $\tilde\zeta$ on the order $\mathcal O(\epsilon^{3/2})$ for time on the $O(\epsilon^{-2})$ scale.  The appropriate wave packet approximation to $z$ is then  obtained upon changing coordinates back to the Lagrangian variable.

\subsection{Construction of Appropriate Initial Data}

Notice that we can parametrize the initial interface $z = z(\cdot,0)$ arbitrarily, and that we are only concerned with such data that $z_\alpha(\cdot,0) - 1$ is $\mathcal{O}(\epsilon^{1/2})$.  For any initial interface that is a small perturbation of the $x$-axis in this sense, $\kappa(\cdot,0): \mathbb R\to \mathbb R$ is a diffeomorphism (c.f. Lemma~\ref{KappaDerivativeBound}).  Hence we may without loss of generality assume that $z = z(\cdot, 0)$ is initially parametrized so that $\kappa(\alpha,0) = \alpha$, 
and hence that $z(\cdot,0) = \zeta(\cdot,0)$.  

In order for $(\zeta_0, v_0, w_0) = (\zeta(0), D_t\zeta(0), D^2_{t}\zeta(0)) = (z(0), z_t(0), z_{tt}(0))$ to be data for a solution $z$ of the water wave system \eqref{OldEuler}-\eqref{ztIsAntihol}, we must enforce the compatibility conditions $(I - \nht_{\zeta_0})\overline{v}_0 = 0$ and $w_0 := i \mathcal{A}_0 \partial_\alpha \zeta_0 - i$, with the formula for $\mathcal{A}_0$ given through \eqref{AFormula} by
\begin{equation}\label{a0}
(I - \nht_{\zeta_0})(\mathcal{A}_0 - 1) = i[w_0, \nht_{\zeta_0}]\frac{\partial_\alpha\overline{\xi}_0}{\partial_\alpha \zeta_0} + i[v_0, \nht_{\zeta_0}]\frac{\partial_\alpha \overline{v}_0}{\partial_\alpha \zeta_0},\end{equation} 
where $\zeta_0 := \xi_0+\alpha$.
We therefore define the manifold of initial data for \eqref{OldEuler}-\eqref{ztIsAntihol} or for \eqref{NewEuler}-\eqref{DtXiIsAntihol} by
\begin{equation*}
\begin{aligned}
\mathscr{A}^s = \{(\xi_0, v_0, w_0): (|D_\alpha|^{1/2}\xi_0,v_0, w_0)\in H^{s+1/2} \times H^{s + 1}\times H^{s+1/2},\; \\\xi_0= \nhtb_{\xi_0+ \alpha} \xi_0, \;(I - \nhtb_{\xi_0 + \alpha})v_0 = 0,\;
 w_0=i\mathcal A_0(\partial_\alpha \xi_0+1) - i\}
 \end{aligned}
 \end{equation*}
 with $\mathcal A_0$  defined by \eqref{a0}.

In the remainder of this section let $s \geq 6$ and $k > 0$ be fixed, and let an arbitrary initial envelope $B_0 \in H^{s + 7}$ be given.  By Theorem \ref{NLSWellPosedness}, for any $\mathscr{T} > 0$ there is a $B \in C([0, \mathscr{T}]; H^{s + 7})$ which solves \eqref{NLS} with initial data $B(0) = B_0$.  Using \eqref{TildeZetaFormula} we can construct, using this $B$, an approximate profile $\tilde{\zeta} \in C([0, \mathscr{T}\epsilon^{-2}]; H^{s + 6})$ satisfying \eqref{NLSGlobalBound} which solves the equations \eqref{NewEuler}-\eqref{XiIsAntihol}-\eqref{DtNewEuler}-\eqref{DtXiIsAntihol} up to a residual of size $O(\epsilon^4)$, provided the initial profile $\tilde{\zeta}(0)$ is calculated through $B_0$.

As we observed above, we cannot simply take $(\tilde{\xi}(0), \tilde{D}_t \tilde{\zeta}(0), \tilde{D}_t^2 \tilde{\zeta}(0))$ as our initial data for \eqref{NewEuler}-\eqref{DtXiIsAntihol}, as these may not be in the manifold $\mathscr A^{s}$.
 Since we found in Proposition \ref{RemainderEnergyAPrioriBound} that a $\mathcal{O}(\epsilon^{3/2})$ error is acceptable, we construct data for $(\zeta - \alpha, D_t\zeta, D_t^2 \zeta)$ which lie in the manifold $\mathscr A^s$, and which are also $\mathcal{O}(\epsilon^{3/2})$ away from $(\tilde{\xi}(0), \tilde{D}_t\tilde{\zeta}(0), \tilde{D}_t^2 \tilde{\zeta}(0))$.

\begin{lemma}\label{InitialDataConstruction}
For sufficiently small $\epsilon_0(\|B_0\|_{H^{s + 7}}) > 0$, there exist functions $\xi_0 \in H^{s + 6}$ and $v_0 \in H^{s + 4}$ with $\zeta_0 := \alpha + \xi_0$ such that for all $\epsilon < \epsilon_0$ the following properties hold:
\begin{enumerate}
\item{$\xi_0 = \frac{1}{2}(1 + \nhtb_{\zeta_0})\tilde{\xi}(0)$,}
\item{$\|\xi_0 - \tilde{\xi}(0)\|_{H^{s + 6}} \leq C(\|B_0\|_{H^{s + 7}})\epsilon^{3/2}$.}
\item{$v_0 := \frac12(I + \nhtb_{\zeta_0})\tilde{D}_t \tilde{\zeta}(0)$ satisfies $\|v_0 - \tilde{D}_t\tilde{\zeta}(0)\|_{H^{s + 4}} \leq C(\|B_0\|_{H^{s + 7}})\epsilon^{3/2}$.}
\item{for $(\xi_0, v_0)$ as constructed  in parts (1) - (3),  $w_0 := i\mathcal{A}_0\partial_\alpha \zeta_0 - i$, with $\mathcal A_0$ calculated by \eqref{a0}
 satisfies $\|w_0 - \epsilon(i\omega)^2\zeta^{(1)}(0)\|_{H^{s+4}} \leq C\epsilon^{3/2}$.}
\end{enumerate}
\end{lemma}

\begin{proof}
We prove Part 1 by an iteration argument.  Define a sequence of functions $g_n(\alpha, t), n = -1, 0, 1, \ldots$ along with $\gamma_n(\alpha, t) := \alpha + g_n(\alpha, t)$ by setting $g_{-1} = 0$ and for $n \geq -1$, 
\begin{equation}\label{RecursiveDefnInitialData}g_{n + 1} = \frac{1}{2}(1 + \nhtb_{\gamma_n})\tilde{\xi}(0)
\end{equation}  Observe first that $g_0 = \frac{1}{2}(I + \nhtb_0)\tilde{\xi}(0)$ and so $\|g_0\|_{H^{s + 6}} \leq C(\|B_0\|_{H^{s + 7}})\epsilon^{1/2}$.  Next, as in the proof of Lemma \ref{DiffHilbertBoundPart2}, we can write
\begin{equation*}
(\nht_{\gamma_n} - \nht_{\gamma_{n - 1}})f = \frac{1}{\pi i}\int \log\left(1 + \frac{(g_n - g_{n - 1})(\alpha) - (g_n - g_{n - 1})(\beta))}{\gamma_{n - 1}(\alpha) - \gamma_{n - 1}(\beta)}\right) f_\beta(\beta) \, d\beta
\end{equation*}
\begin{equation*}
= \frac{1}{\pi i} \int \biggl(\frac{(g_n^\prime(\beta) - g_{n - 1}^\prime(\beta))}{\gamma_n(\alpha) - \gamma_n(\beta)} - \frac{\gamma_{n - 1}^\prime(\beta)\left((g_n - g_{n - 1})(\alpha) - (g_n - g_{n - 1})(\beta)\right)}{(\gamma_n(\alpha) - \gamma_n(\beta))(\gamma_{n - 1}(\alpha) - \gamma_{n - 1}(\beta))}\biggr) f(\beta) \, d\beta \\
\end{equation*} From this formula and Proposition \ref{SingIntSobolevEstimates} we have the estimate $$\|(\nht_{\gamma_n} - \nht_{\gamma_{n - 1}})f\|_{H^{s + 6}} \leq C\left(\|g_n\|_{H^{s + 6}}, \|g_{n - 1}\|_{H^{s + 6}}\right)\|g_n - g_{n - 1}\|_{H^{s + 6}} \|f\|_{H^{s + 6}}$$ if we can show that $\gamma_n$ and $\gamma_{n - 1}$ obey the chord-arc condition.  However, there indeed exists some $\delta \in (0, \frac{1}{2}]$ so that if $\|g_n\|_{H^{s + 6}}, \|g_{n - 1}\|_{H^{s + 6}} \leq \delta$, then $\gamma_n$ and $\gamma_{n - 1}$ satisfy the chord-arc condition and the operator norm $\|\nht_{\gamma_n} - \nht_{\gamma_{n - 1}}\|_{H^{s + 6} \to H^{s + 6}} \leq C_1\|g_n - g_{n - 1}\|_{H^{s + 6}}$, where $C_1$ is a universal constant.  Choose $\epsilon_0$ so small so that $C_1\|\tilde{\xi}(0)\|_{H^{s + 6}} \leq \delta$ and $\|g_0\|_{H^{s + 6}} \leq \frac{1}{2}\delta$.

It now suffices to prove the following statement by induction:  For every $n \geq 0$, $$\|g_{n + 1} - g_n\|_{H^{s + 6}} \leq \frac{1}{2}\delta\|g_n - g_{n - 1}\|_{H^{s + 6}} \qquad \text{and} \qquad \|g_n\|_{H^{s + 6}} \leq \delta$$  By our choice of $\epsilon_0$ we have already shown the case $n = 0$.  If we assume the above statement is true for all integers $k = 0, 1, \ldots, n$, note that 
\begin{align*}
\|g_{n + 1} - g_n\|_{H^{s + 6}} & = \frac{1}{2}\|(\nhtb_{\gamma_n} - \nhtb_{\gamma_{n - 1}})\tilde{\xi}(0)\|_{H^{s + 6}} \\
& \leq \frac{1}{2}\|\nhtb_{\gamma_n} - \nhtb_{\gamma_{n - 1}}\|_{H^{s + 6} \to H^{s + 6}} \cdot \|\tilde{\xi}(0)\|_{H^{s + 6}} \\
& \leq \frac{1}{2}\delta\|g_n - g_{n - 1}\|_{H^{s + 6}},
\end{align*} from which the induction statement follows immediately.

To prove Part 2, we note that since $\|\xi_0\|_{H^{s + 6}}$ and $\|\tilde{\xi}(0)\|_{H^{s + 6}}$ do not exceed $\delta$, we can estimate that
\begin{align*}
\|\xi_0 - \tilde{\xi}(0)\|_{H^{s + 6}} & = \frac{1}{2}\|(1 - \nhtb_{\zeta_0})\tilde{\xi}(0)\|_{H^{s + 6}} \\
& \leq \frac{1}{2}\|(\nhtb_{\zeta_0} - \nhtb_{\tilde{\zeta}(0)})\tilde{\xi}(0)\|_{H^{s + 6}} + \|(I- \nhtb_{\tilde{\zeta}(0)})\tilde{\xi}(0)\|_{H^{s + 6}} \\
& \leq \delta\|\xi_0 - \tilde{\xi}(0)\|_{H^{s + 6}} + C\epsilon^{3/2}
\end{align*} from which Part 2 follows.  Since the construction of $v_0$ is determined by $\zeta_0$, Part 3 is now shown in the same way as was Part 2 once we observe that $\tilde{D}_t\tilde{\zeta}(0) \in H^{s + 4}$ and $\nhtb_{\zeta_0}$ is bounded from $H^{s + 4}$ to $H^{s + 4}$.

We now prove Part 4.
By the definition of $w_0$ we have
\begin{align*}
w_0 - \epsilon(i\omega)^2\zeta^{(1)}(0) & = i\mathcal{A}_0\partial_\alpha \zeta_0 - i - \epsilon(i\omega)^2\zeta^{(1)}(0) \\
& = i(\mathcal A_0 - 1)\partial_\alpha \zeta_0 + i\left(\partial_\alpha \xi_0 - \epsilon (i k)\zeta^{(1)}(0)\right)
\end{align*}  Since we are assuming $\xi_0$ and $v_0$ are constructed as above, we can write $v_0 = (v_0 - \epsilon(i\omega)\zeta^{(1)}(0)) + \epsilon(i\omega)\zeta^{(1)}(0) \in H^{s + 4}$ and $\partial_\alpha \xi_0 = (\partial_\alpha \xi_0 - \epsilon(ik)\zeta^{(1)}(0)) + \epsilon(ik)\zeta^{(1)} (0)\in H^{s + 5}$ in the above formula for $\mathcal{A}_0 - 1$.  As usual, we can isolate the $O(\epsilon^2)$ leading term and see that it vanishes by a multiscale calculation, and what remains gives us the estimate $$\|\mathcal{A}_0 - 1\|_{H^{s + 4}} \leq C\epsilon^{3/2} + C\epsilon\|w_0 -\epsilon (i\omega)^2\zeta^{(1)}(0)\|_{H^{s + 4}}$$  But then we have by the above that $\|w_0 - \epsilon(i\omega)^2\zeta^{(1)}(0)\|_{H^{s + 4}} \leq C\epsilon^{3/2}$ for a sufficiently small choice of $\epsilon_0$.
\end{proof}

\begin{definition}  We  call $(\xi_0, v_0, w_0)$ a  $B_0$-\textbf{admissible} initial data  
 if $(\xi_0, v_0, w_0) \in \mathscr{A}^s$ and there is a constant $C$ depending only on $\|B_0\|_{H^{s + 7}}$ so that $$\left\|(|D_\alpha|^{1/2}\xi_0, v_0, w_0) - (\epsilon |D_\alpha|^{1/2}\zeta^{(1)}(0), \epsilon \zeta^{(1)}_{t_0}(0), \epsilon \zeta^{(1)}_{t_0 t_0}(0))\right\|_{H^{s + 1/2} \times  H^{s + 1} \times H^{s + 1/2}} \leq C\epsilon^{3/2}$$  \end{definition}

Recall from \eqref{FullEnergyDefinition} that
\begin{align*}
\mathcal{E} & = \sum_{n = 0}^s (\mathcal{E}_n + \mathcal{F}_n) \\
& \leq C\sum_{n = 0}^s(\|D_t\partial_\alpha^n \rho\|_{L^2}^2 + \|D_t\partial_\alpha^n \sigma\|_{L^2}^2) + \||D|^{1/2}\rho\|_{H^{s + 1/2}}^2 + \|\sigma\|_{H^{s + 1}}^2.
\end{align*}
It is clear that for $B_0$-\textbf{admissible} initial data, we have
\begin{equation}\label{energyzero}
\mathcal E(0)\le C \epsilon^{3}.
\end{equation}

\subsection{Long-Time Existence of $\zeta$ and $z$}

In this section we will make rigorous the existence and uniqueness of the solutions $z$ and $\zeta$ on the appropriate $O(\epsilon^{-2})$ time scales.  We begin with the following local well-posedness (c.f.,  \cite{WuLocal2DWellPosed}, \cite{WuAlmostGlobal2D}):

\begin{theorem}\label{zLocalWellPosed}
Let $n \geq 5$ be given.  Suppose that initial data $\xi_0$, $v_0$, $w_0$ are given so that $\partial_\alpha z(0) - 1 = \partial_\alpha \xi_{0}$ is in $H^{n - 1/2}$, $z_t(0) = v_0$ is in $H^{n + 1/2}$, $ z_{tt}(0)=w_0 $ is in $H^n$; $\xi_0$, $v_0$, $w_0$ satisfy the  water wave system: i.e. $\bar v_0={\mathcal H}_{z(0)}\bar v_0$, and $w_0 = i\mathfrak{a}_0\partial_\alpha z(0) - i$ for some real valued function $\mathfrak{a}_0$.  Suppose further that $z(0) = \alpha + \xi_0(\alpha)$, $\alpha\in \mathbb{R}$ defines a chord-arc curve: i.e. there exists $\nu>0$, such that 
$$|\alpha + \xi_0(\alpha) - \beta - \xi_0(\beta)| \geq \nu |\alpha - \beta|,\qquad\text{for all } \alpha,\beta\in \mathbb R;$$

Then there exists a $T_0> 0$ so that the system \eqref{OldEuler}-\eqref{ztIsAntihol} with initial data $z(0) = \xi_0 + \alpha$, $z_t(0) = v_0$, $z_{tt}(0)=w_0 $ has a unique solution $z(\alpha, t)$ for $t\in [0, T_0]$ with the property that there exist constants $C = C(T, \|\partial_\alpha\xi_{0}\|_{H^{n - 1/2}}, \|v_0\|_{H^{n + 1/2}}, \|w_0\|_{H^{n }}, \nu)$ and $\mu>0$,  $$\|(z_\alpha - 1, z_t, z_{tt})\|_{C([0, T_0]; H^{n - 1/2} \times H^{n + 1/2} \times H^n)} \leq C\left(\|\partial_\alpha\xi_{0}\|_{H^{n - 1/2}} + \|v_0\|_{H^{n + 1/2}} + \|w_0\|_{H^{n }}\right),$$
and  $|z(\alpha,t)-z(\beta,t)|\ge \mu |\alpha-\beta|$ for all $ \alpha,\beta\in \mathbb R, \ t\in [0, T_0]$.

Moreover,  if $T^*$ is the supremum over all such $T_0$, then either $T^* = \infty$ or 
\begin{equation}\label{zBlowUpQuantity}
\lim_{t \nearrow T^{*}}\left( \|(z_t, z_{tt})\|_{C([0, t], H^{n} \times H^{n})} + \sup_{\alpha \neq \beta} \left|\frac{\alpha - \beta}{z(\alpha, t) - z(\beta, t)}\right| \right)=\infty
\end{equation}
\end{theorem}

Given this result, we take any $B_0$-admissible initial data $(\xi_0, v_0, w_0) \in \mathscr A^{s}$ and use Theorem \ref{zLocalWellPosed} to construct a solution $z = z(\alpha, t)$ on the time interval $[0, T_0]$ with $(z_\alpha(t) - 1, z_t(t), z_{tt}(t)) \in H^s \times H^{s + 1} \times H^{s + 1/2}$.  Using this solution we construct the change of variables $$\kappa = \overline{z} + \frac{1}{2}(I + \oht)(I + \mathfrak{K})^{-1}(z - \overline{z})$$ on $[0, T_0]$ as in \S 2.  
In order to use this change of variables to control $\zeta$ in terms of $z$, we need the following elementary calculus lemma.

\begin{lemma}\label{ChangeVarBounds}
Let $n \geq 3$, let $f \in H^n$, and let $\gamma \in H^n$ be given with $\gamma^\prime(\alpha) \geq c_0 > 0$ for all $\alpha\in \mathbb{R}$ and $\|\gamma^\prime - 1\|_{H^{n - 1}} \leq M$.  Then
\begin{enumerate}
\item{$\|f \circ \gamma\|_{L^2} \leq C(c_0)\|f\|_{L^2}$.}
\item{$\|f \circ \gamma\|_{H^n} \leq C(M, c_0)\|f\|_{H^n}$.}
\end{enumerate}
\end{lemma}

\begin{proof}
First we have $$\|f \circ \gamma\|_{L^2} = \left(\int |f \circ \gamma|^2 d\alpha\right)^{1/2} = \left(\int |f|^2 \frac{d\alpha}{\gamma^\prime} \right)^{1/2} \leq \frac{1}{\sqrt{c_0}}\|f\|_{L^2}$$ which proves (1).  To prove (2), first observe that 
\begin{align*}
\|\partial_\alpha(f \circ \gamma)\|_{L^2} & = \|(f^\prime \circ \gamma)\gamma^\prime\|_{L^2} \\
& \leq C(c_0)\|\gamma^\prime\|_{L^\infty}\|f'\|_{L^2} \\
& \leq C(c_0)(1 + \|\gamma^\prime - 1\|_{H^2})\|f'\|_{L^2}
\end{align*} Now let $n \geq 3$ and let $2 \leq j \leq n$ be an integer.  By the chain and product rules there exist polynomials $p_{l, j}( \gamma^\prime, \ldots, \gamma^{(j - 1)})$ of total degree\footnote{This is meant to include both algebraic multiplicity and the number of differentiations.  For instance, the term $f^{\prime \prime}(f^\prime)^2$ has total order $4$.} at most $j$ such that
$$\partial_\alpha^j(f \circ \gamma) = (f' \circ \gamma)\partial_\alpha^{j - 1}(\gamma^\prime - 1) + \sum_{l = 2}^j (f^{(l)} \circ \gamma) \, p_{l, j}( \gamma^\prime, \ldots, \gamma^{(j - 1)})$$  The lemma follows upon estimating the first term with $f' \circ \gamma$ in $L^\infty$ and $\partial_\alpha^{(j - 1)}(\gamma^\prime - 1)$ in $L^2$, and the remaining terms with $f^{(l)} \circ \gamma$ in $L^2$ by (1) and $p_{l, j}$ in $L^\infty$.
\end{proof}

To use this Lemma to change from the $\zeta$ quantites back to the $z$ quantities, we need control of $\kappa_\alpha - 1$ in terms of $z_\alpha - 1$ in $H^s$.

\begin{lemma}\label{KappaDerivativeBound}
For $n \geq 3$, if $\|z_\alpha - 1\|_{C([0, T_0]; H^n)}$ is sufficiently small, then $\|\kappa_\alpha - 1\|_{C([0, T_0]; H^n)} \leq C\|z_\alpha - 1\|_{C([0, T_0]; H^n)}$.
\end{lemma}

\begin{proof}
Differentiating \eqref{ChangeOfVariables} with respect to $\alpha$ we get
\begin{align*}
\kappa_\alpha - 1 & = (\overline{z}_\alpha - 1) + \frac{1}{2}(I + \oht)\partial_\alpha(I + \mathfrak{K})^{-1}(z - \overline{z}) + \frac{1}{2}[z_\alpha - 1, \oht]\frac{\partial_\alpha(I + \mathfrak{K})^{-1}(z - \overline{z})}{z_\alpha}
\end{align*}  
Now the Lemma follows from Lemma~\ref{DoubleLayerPotentialArgument} and the $H^n$ boundedness of the operator $\oht$. 
\end{proof}

Since we have chosen initial data  that is $\mathcal{O}(\epsilon^{1/2})$, by Theorem \ref{zLocalWellPosed} there is an interval $[0, T_0]$, such that for all $t\in [0, T_0]$, both  $\|z_\alpha(t) - 1\|_{H^{s}}$ and $\|z_t(t)\|_{H^{s}}$ are of order $\mathcal{O}(\epsilon^{1/2})$.  Also by Lemma \ref{KappaDerivativeBound} we have that $\|\kappa_\alpha - 1\|_{H^{s}}$ is of order $\mathcal{O}(\epsilon^{1/2})$. Then for $\epsilon<\epsilon_0$ we can choose $\epsilon_0>0$ so small that $\|\kappa_\alpha - 1\|_{L^\infty} \leq \frac{1}{2}$ and $\|\kappa_\alpha - 1\|_{H^s} \leq 1$.  Applying Lemma \ref{ChangeVarBounds}, we can choose $\epsilon_0$ sufficiently small so that  
\begin{align*}
\|\zeta_\alpha(t) - 1\|_{H^{s}} & \leq C\left\|\frac{z_\alpha}{\kappa_\alpha} - 1\right\|_{H^{s}} \\
& \leq C\|z_\alpha - 1\|_{H^{s }} + C\|\kappa_\alpha - 1\|_{H^{s}} \\
& \leq \frac{1}{2}\delta
\end{align*}
and 
$$\|D_t \zeta(t)\|_{H^{s + 1}} = \|z_t \circ \kappa^{-1}(t)\|_{H^{s + 1}} \leq \frac{1}{2}\delta$$
 for all times $t\in [0, T_0]$, where $\delta$ is the quantity required by \eqref{ZetaLocalAPrioriBound}.  
 This now justifies the a priori bound \eqref{ZetaLocalAPrioriBound} on $[0, T_0]$.  Since we now legitimately have such a bound, all of the work  through Proposition \ref{RemainderEnergyAPrioriBound} now holds on $[0, T_0]$ for $\delta$ and $\epsilon_0$ chosen sufficiently small.  We are now ready to prove the main

\begin{theorem}\label{MainResult}
Let $s \geq 6$ and $k > 0$ be given.  Let $B_0 \in H^{s + 7}$, and $\mathscr{T} > 0$ be given.  Denote by $B(X, T)$ the solution of \eqref{NLS} with initial data $B(0) = B_0$, and let $\zeta^{(1)}$ be defined as in \eqref{Zeta1Formula}.  Then there exists an $\epsilon_0 = \epsilon_0(\|B_0\|_{H^{s + 7}}) > 0$ so that for all $\epsilon < \epsilon_0$ the following holds: there exists initial data $(\xi_0, v_0, w_0) \in \mathscr{A}^s$ for the system  \eqref{OldEuler}-\eqref{ztIsAntihol} satisfying $$\|(|D_\alpha|^{1/2}\xi_0, v_0, w_0) - (\epsilon |D_\alpha|^{1/2} \zeta^{(1)}(0), \epsilon \zeta^{(1)}_{t}(0), \epsilon \zeta^{(1)}_{tt}(0))\|_{H^{s + 1/2} \times H^{s + 1} \times H^{s + 1/2}} \leq M_0\epsilon^{3/2},$$ and for all such initial data, there exists a possibly smaller $\epsilon_0 = \epsilon_0(\|B_0\|_{H^{s + 7}}, \mathscr{T}, M_0) > 0$ so that the system \eqref{OldEuler}-\eqref{ztIsAntihol} has a unique solution $z(\alpha, t)$ with $\left(|D_\alpha|^{1/2}(z - \alpha), z_t, z_{tt}\right)$ in the space $C([0, \mathscr{T}\epsilon^{-2}]; H^{s + 1/2} \times H^{s + 1} \times H^{s + 1/2})$ satisfying
\begin{equation}\label{errorestimate}
\begin{aligned}
\|(\zeta_\alpha(t) - 1, D_t\zeta(t),  D_t^2\zeta(t)) &- (\epsilon \zeta^{(1)}_{\alpha}(t), \epsilon \zeta^{(1)}_{t}(t), \epsilon \zeta^{(1)}_{tt}(t))\|_{H^{s} \times H^{s} \times H^s}
\\& \le C(\|B_0\|_{H^{s + 7}}, \mathscr{T}, M_0)\epsilon^{3/2}
\end{aligned}
\end{equation}
for all $0 \leq t \leq \epsilon^{-2}\mathscr{T}$.
\end{theorem}

\begin{proof}
Given our initial data, we have shown that there is some time interval $[0, T_0]$ on which a solution to \eqref{OldEuler}-\eqref{ztIsAntihol} exists with that initial data.  We have also shown that for sufficiently small $\epsilon_0$ the a priori bound \eqref{ZetaLocalAPrioriBound} on $\zeta$ holds and  $\kappa$ satisfies  $\|\kappa_\alpha - 1\|_{L^\infty} \leq \frac{1}{2}$ and $\|\kappa_\alpha - 1\|_{H^s} \leq 1$ on $[0, T_0]$. Now assume that $[0, T^*]$ is the maximum of such intervals contained in $ [0, \mathscr T\epsilon^{-2}]$. We will show in what follows that $T^* = \mathscr T\epsilon^{-2}$. 
 We assume now $T^* < \mathscr T\epsilon^{-2}$ for otherwise we are done. 

First we have by \eqref{NLSGlobalBound}, \eqref{energyzero}, the estimates in Section 4 and Proposition \ref{RemainderEnergyAPrioriBound} that for all $t \in [0, T^*]$,
\begin{equation}\label{zetaestimate}
\begin{aligned}
\|D_t \zeta(t)\|_{H^{s}} + \|\zeta_\alpha(t) - 1\|_{H^{s}} +& \|D_t^2 \zeta(t)\|_{H^{s}}  \leq \|\tilde{D_t}\tilde{\zeta}(t)\|_{H^{s}} + \|\tilde{\xi}_\alpha(t)\|_{H^{s}} + \|\tilde{D}_t^2 \tilde{\zeta}(t)\|_{H^{s}} \\
& \quad + \|(D_t - \tilde{D}_t)\tilde{\zeta}(t)\|_{H^s} + \|(D_t^2 - \tilde{D}_t^2)\tilde{\zeta}(t)\|_{H^s} \\
& \quad + C(\mathcal{E}^{1/2} + \epsilon^{5/2}) \\
& \leq C\epsilon^{1/2}.
\end{aligned}
\end{equation}  In particular,  this estimate holds with a constant $C$ independent of $T^*$. 

In order to use this bound on $\zeta$ to in turn control $z$, we would like to show that the change of variables $\kappa$ can be constructed in terms of $\zeta$ so that it is controlled independently of $T^*$.  This will imply that there are similar a priori estimates for $z$, and so the long-time existence with appropriate regularity will then follow from the blow-up criterion of Theorem \ref{zLocalWellPosed}.

We know $\kappa(\alpha,t)$ satisfies   
\begin{equation}\label{KappaIVP}
\begin{cases} \kappa_t(\alpha, t) = b(\kappa(\alpha,t), t) \\ \kappa(\alpha,0) = \alpha\end{cases}
\end{equation} with $b$  determined through \eqref{bFormula}.   
Writing \eqref{KappaIVP} in integral form, differentiating with respect to $\alpha$, and using Lemma \ref{ChangeVarBounds} then gives the bound
\begin{align*}
\|\kappa_\alpha(t) - 1\|_{H^{s - 1}} & \leq \int_0^t \|b_\alpha(\kappa(\tau), \tau)\|_{H^{s - 1}}\left(1 + \|\kappa_\alpha(\tau) - 1\|_{H^{s - 1}}\right)  d\tau \\
& \leq C\epsilon^{1/2}\left(1 + \|\kappa_\alpha(t) - 1\|_{C([0, T^*]; H^{s - 1})}\right)
\end{align*}  Taking the supremum over $0 \leq t \leq T^*$ and choosing $\epsilon_0$ sufficiently small then yield 
 \begin{equation}\label{kappaestimate}
 \|\kappa_\alpha - 1\|_{C([0, T^*]; H^{s - 1})}\leq C\epsilon^{1/2}
 \end{equation}
  where the constant $C$ depends on $\mathscr{T}$, and is independent of $T^*$.

Now on $[0, T^*]$, we have that 
$\zeta(\kappa(\alpha,t), t) = z(\alpha, t)$.  
  Hence if we apply Lemma \ref{ChangeVarBounds} we have for $t\in [0, T^*]$, 
\begin{align*}
\|z_\alpha(t)-1\|_{H^{s-1}}&+\|z_t(t)\|_{H^s} + \|z_{tt}(t)\|_{H^{s}} \\
& \leq C\left(\|\zeta_\alpha(t)-1\|_{H^{s-1}}+\|\kappa_\alpha(t)-1\|_{H^{s-1}}+\|D_t\zeta(t)\|_{H^s} + \|D_t^2\zeta(t)\|_{H^s}\right) \\
& \leq C\epsilon^{1/2}
\end{align*}
and that
\begin{equation*}
\sup_{\alpha \neq \beta} \left|\frac{\alpha - \beta}{z(\alpha) - z(\beta)}\right| \leq \frac{1}{(1 - \|\zeta_\alpha - 1\|_{L^\infty})(1 - \|\kappa_\alpha - 1\|_{L^\infty})},
\end{equation*} where the constants $C$ are independent of $T^*$.  Thus it follows by the blow-up criteria given in Theorem \ref{zLocalWellPosed} that we can continue the solution $z$ to  $t\in [0, T_1]$ for some $T_1 > T^*$.
On the other hand, we can choose $\epsilon_0$ so small that for $\epsilon < \epsilon_0$, the bounds $C\epsilon^{1/2}$ in \eqref{zetaestimate} and \eqref{kappaestimate} are  small enough that there exist $T^* < T_2 < T_1$, so that on $[0, T_2]$,  $\|\kappa_\alpha - 1\|_{L^\infty} \leq \frac{1}{2}$, $\|\kappa_\alpha - 1\|_{H^s} \leq 1$  and the a priori estimate \eqref{ZetaLocalAPrioriBound} holds.
This contradicts the maximality of $T^*$. Therefore we must have $T^*=\mathscr T\epsilon^{-2}$ and the long time existence of $z$ follows. The error estimate \eqref{errorestimate} then follows from \eqref{energyzero} and Proposition~\ref{RemainderEnergyAPrioriBound}.
\end{proof}

There is still the matter of interpreting this result in more familiar coordinates.  Changing variables by $\kappa$, we can convert the estimates of the above theorem into estimates in Lagrangian coordinates:  
\begin{equation}\label{lag}
\|(z_\alpha - \kappa_\alpha, z_t, z_{tt}) - (\epsilon \zeta^{(1)}_\alpha \circ \kappa, \epsilon \zeta^{(1)}_t \circ \kappa, \epsilon \zeta^{(1)}_{tt} \circ \kappa)\|_{H^s \times H^s \times H^s} \leq C\epsilon^{3/2}
\end{equation}  Calculating the asymptotic expansion of $z_\alpha - 1, z_t, z_{tt}$ now depends on understanding $\kappa - \alpha$.  From \eqref{KappaIVP} we have that 
\begin{align*}
\kappa(\alpha, t) - \alpha & = \int_0^t b(\kappa(\alpha, \tau), \tau) d\tau
\end{align*}  
Using our estimate of $\|\kappa_\alpha - 1\|_{H^s} \leq C\epsilon^{1/2}$ and 
writing the integrand as $b = (b - \tilde{b}) + \epsilon^2 b_2 + \epsilon^3 b_3$ yields the following leading order expression:
\begin{align}\label{kexp}
\kappa(\alpha, t) - \alpha & = -k\omega\epsilon^2 \int_0^t |B|^2(\epsilon\alpha + \epsilon \omega^\prime\tau, \epsilon^2\tau) d\tau + \mathcal{O}(\epsilon^{1/2})
\end{align} 
From \eqref{kexp}, we can obtain and justify asymptotics for $\partial_\alpha \Im z$, $z_t$ and $z_{tt}$ without any additional restriction on the initial data. 
However, justifying the asymptotics for $\Re z_\alpha - 1$ requires an understanding of the asymptotic for $\kappa_\alpha$ up to order $\mathcal O(\epsilon^{3/2})$, which is not available merely from the estimates given in Theorem~\ref{MainResult}.  We therefore leave open the justification of the modulation approximations for 
$\Re z_\alpha-1$. 
 Note that the leading term of the right hand side of \eqref{kexp} can be as large as $O(1)$ on times of order $O(\epsilon^{-2})$, and so would contribute corrections to the asymptotic formula for $\Re z_\alpha  -1$.

\section{Justification of an Eulerian Version}

By imposing some additional mild restrictions on the initial data, we are able to obtain justifications of the derivative in the space variable of the interface and the trace of the velocity field on the interface in Eulerian coordinates. With further restrictions on the initial data, we are able to justify the asymptotics for the profile itself. All these reduce to obtaining an appropriate bound and, in the latter case, asymptotics for $\Re \zeta(\alpha,t)-\alpha$ in $C([0, \mathscr T\epsilon^{-2}]; L^2)$, which can be achieved by introducing another quantity as follows.
 
Following the proof of Proposition 2.3 of \cite{WuAlmostGlobal2D}, we introduce the velocity potential $\Phi(x, t)$ of the fluid in the domain $\Omega(t)$ that satisfies $\nabla \Phi = {\mathbf{v}}$.  Let $\psi(\alpha, t) = \Phi(z(\alpha, t), t)$ be the trace of $\Phi$ on the interface $\Sigma(t)$.  If we write $\Psi = \psi \circ \kappa^{-1}$, then the time derivative of the quantity $\lambda := (I - \nht)\Psi$ is comparable to the imaginary part of $\zeta$ through the identity (c.f. (2.46) of \cite{WuAlmostGlobal2D}):
\begin{equation}\label{DtLambdaFormula}
D_t\lambda = -(I - \nht)\Im(\zeta) - \frac{1}{2}[D_t\zeta, \nht]\frac{\overline{\zeta}_\alpha D_t\zeta}{\zeta_\alpha}
\end{equation}  We also know by Proposition 2.3 of \cite{WuAlmostGlobal2D} that $\lambda$ satisfies an evolution equation of the form
\begin{align}\label{LambdaEvolutionEquation}
\mathcal{P}\lambda & = -\left[D_t\zeta,\nht\frac{1}{\zeta_\alpha} + \nhtb\frac{1}{\overline{\zeta}_\alpha}\right](\overline{\zeta}_\alpha D_t^2\zeta) + [D_t\zeta, \nhtb]\left(D_t\overline{\zeta}\frac{\partial_\alpha D_t\zeta}{\overline{\zeta}_\alpha}\right) + D_t\zeta[D_t\zeta, \nht]\frac{\partial_\alpha D_t\overline{\zeta}}{\zeta_\alpha} \notag \\
& \qquad - 2[D_t\zeta, \nht]\frac{D_t\zeta \cdot \partial_\alpha D_t \zeta}{\zeta_\alpha} + \frac{1}{\pi i} \int \left(\frac{D_t\zeta(\alpha) - D_t\zeta(\beta)}{\zeta(\alpha) - \zeta(\beta)}\right)^2(D_t\zeta(\beta) \cdot \zeta_\beta(\beta)) d\beta \notag \\
& := G_\lambda
\end{align}  Since $G_\lambda$ is of third order and depends only on $\zeta_\alpha - 1, D_t\zeta, D_t^2 \zeta$, we expect that we can construct an energy from this equation that allows us to bound $D_t \lambda$ by $C\epsilon^{1/2}$, provided the initial energy is bounded by $C\epsilon$.  This is enough to control $\|\Re \zeta(\cdot,t)- \alpha\|_{L^2}$ and justify an Eulerian version of Theorem \ref{MainResult}.  The details are given in Section 6.1 below.

However, with further restrictions on the initial data we can justify asymptotics for the profile itself, and we will devote the remainder of Section 6 to this task.  Specifically, we will develop an approximate solution $\tilde{\lambda}$ to \eqref{LambdaEvolutionEquation} to the desired order $O(\epsilon^4)$ and thereby construct an energy for the remainder $l = \lambda - \tilde{\lambda}$.  As was the case with the quantities $D_t \rho$ and $D_t \sigma$, such an energy will bound the $L^2$ norm of $D_t l$ for $O(\epsilon^{-2})$ times.  This will allow us to justify asymptotics for the profile under reasonable restrictions on the initial profile and the initial velocity potential restricted to the initial interface.

\subsection{Justifying Eulerian Asymptotics for Derivatives of the Profile}

Our first task is to prove the 

\begin{lemma}\label{RealPartProfileControl}
Suppose that the hypotheses of Theorem \ref{MainResult} hold.  Suppose further that $\|\xi_0\|_{L^2} \leq C\epsilon^{1/2}$ and $\|\mathbf{v}_0\|_{L^2(\Omega(0))} \leq C\epsilon^{1/2}$, where $\mathbf{v}_0$ is the initial velocity field.  Then $\|\Re \zeta(\cdot, t) - \alpha\|_{L^2} \leq C\epsilon^{1/2}$ for all $t \leq \mathscr{T}\epsilon^{-2}$.
\end{lemma}

\begin{proof}
We begin by deriving conditions under which $\Re(\zeta) - \alpha$ is controlled in $L^2$.  We can construct the energy corresponding to \eqref{LambdaEvolutionEquation}:  $$\mathcal{L}(t) = \int \frac{1}{\mathcal{A}}|D_t \lambda|^2 + i \lambda \overline{\lambda}_\alpha$$  Since $\lambda$ is the trace of a holomorphic function on $\Omega(t)^c$, we have by Proposition \ref{BasicEnergyInequality} that $\|D_t \lambda\|_{L^2}^2 \leq C\mathcal{L}(t)$.  Formula \eqref{DtLambdaFormula} provides the estimate
\begin{align*}
\Bigl| \|D_t\lambda\|_{L^2} - \|(I - \nht)\Im\zeta\|_{L^2} \Bigr|  & \leq \|D_t \lambda + (I - \nht)\Im(\zeta)\|_{L^2} \\
& \leq C\epsilon^{5/2}
\end{align*}  Clearly  we also have $\|(I - \nht)\Im \zeta\|_{L^2} \leq C\|\xi\|_{L^2}$.  However, by \eqref{XiIsAntihol} and Lemma \ref{DoubleLayerPotentialArgument}, we conversely have that
\begin{align*}
\|\xi\|_{L^2} & \leq \|(I - \nht)\Re\xi\|_{L^2} + \|\Im\xi\|_{L^2} \\
& = \|(I - \nht)\Im\xi\|_{L^2} + \|\Im\xi\|_{L^2} \\
& \leq C\|(I - \nht)\Im\xi\|_{L^2} \\
& \leq C\mathcal{L}^{1/2} + C\epsilon^{5/2}
\end{align*}  Hence it suffices to show that $\mathcal{L}(t)$ is $O(\epsilon)$ whenever $t \leq \mathscr{T}\epsilon^{-2}$.  Now by Proposition \ref{BasicEnergyInequality} and Theorem \ref{MainResult} the energy $\mathcal{L}$ satisfies
$$\frac{d\mathcal{L}}{dt} \leq C\epsilon^{5/2}\mathcal{L}^{1/2} + C\epsilon^{2}\mathcal{L}$$  
therefore
$$\frac{d\mathcal{L}^{1/2}}{dt} \leq C\epsilon^{5/2} + C\epsilon^{2}\mathcal{L}^{1/2}.$$ 
Solving this inequality  gives us that $$\sup_{0 \leq t \leq \mathscr{T}\epsilon^{-2}} \mathcal{L}(t)^{1/2} \leq C\mathcal{L}(0)^{1/2} + C\epsilon^{1/2}$$

Hence the question now reduces to asking which conditions on the initial data imply that $\mathcal{L}(0)$ is $O(\epsilon)$.  We first have that $$\int \frac{1}{\mathcal{A}_0} |D_t\lambda_0|^2 d\alpha \leq C\|D_t \lambda_0\|_{L^2}^2 \leq (\|\xi_0\|_{L^2} + C\epsilon^{5/2})^2,$$ and so to control this part of $\mathcal{L}(0)$ it suffices to take $\|\xi_0\|_{L^2} \leq C\epsilon^{1/2}$.

The second part of $\mathcal{L}(0)$ takes more work.  Recall that our parametrization for the initial data was chosen so that $\zeta(0) = z(0)$.  Let $\psi_0$, $\lambda_0$, etc., be the initial values of $\psi$, $\lambda$, etc., respectively.  To estimate the second part of $\mathcal{L}(0)$, we follow the discussion of initial data in section 5.1 of \cite{WuAlmostGlobal2D}.  Observe that we can choose a function $\Xi_0$ holomorphic in $\Omega(0)$ for which $\Re(\Xi_0) \circ \zeta_0 = \Psi_0$, specifically $\Xi_0 \circ \zeta_0 = (I + \nht_{\zeta_0})(I + \mathcal{K}_{\zeta_0})^{-1}\Psi_0$; such a function will satisfy $\partial_z \Xi_0 = \overline{\mathbf{v}}_0$.  Since we have the operator identity $$(I - \nht) - (I + \nhtb)(I + \mathcal{K})^{-1} = -(I + \nht)(I + \mathcal{K})^{-1}\mathcal{K}$$ it follows that
\begin{equation*}
\lambda_0 - \overline{\Xi}_0 \circ \zeta_0 = -(I + \nht_{\zeta_0})(I + \mathcal{K}_{\zeta_0})^{-1}\mathcal{K}_{\zeta_0}\Psi_0
\end{equation*}  Observe that we can control derivatives of $\Psi_0$ but not $\Psi_0$ itself; however, the expression for $\mathcal{K}_{\zeta_0}\Psi_0$ contains an extra derivative.  Write $z^\tau_0 = (1 - \tau)\zeta_0 + \tau \overline{\zeta}_0$, so $z^1_0 = \bar \zeta_0$ and $z^0_0 = \zeta_0$.  Then by the Fundamental Theorem of Calculus we have
\begin{align*}
\mathcal{K}_{\zeta_0} & =\frac12(\mathcal H_0+\overline{\mathcal H}_0)= -\frac{1}{2}(\nht_{z_0^1} - \nht_{z_0^0}) \\
& = -\frac{1}{2} \int_0^1 \partial_\tau \nht_{z_0^\tau} d\tau \\
& = -\frac{1}{2} \int_0^1 [\overline{\xi}_0 - \xi_0, \nht_{z_0^\tau}]\frac{\partial_\alpha}{z_\alpha^\tau} \, d\tau
\end{align*} and so estimating this expression crudely gives the bound $\|\mathcal{K}_{\zeta_0} \Psi_0\|_{L^2} \leq C\|\xi_0\|_{L^\infty}\|v_0\|_{L^2} \leq C\epsilon$, and so $\|\lambda_0 - \overline{\Xi}_0 \circ \zeta_0\|_{L^2} \leq C\epsilon$ as well.  But then we can by Green's Theorem write
\begin{align*}
\left|\int i\lambda_0 \partial_\alpha \overline{\lambda}_0 d\alpha \right| & \leq 2 \left| \int \partial_\alpha \lambda_0 (\overline{\lambda}_0 - \Xi \circ \zeta_0) d\alpha \right| + \left| \int i(\overline{\Xi} \circ \zeta_0)\partial_\alpha(\Xi \circ \zeta_0) d\alpha \right| \\
& \leq C\epsilon^{3/2} + \iint_{\Omega(0)} |\mathbf{v}_0(x)|^2 dx
\end{align*}  Hence if we choose $\|\mathbf{v}_0\|_{L^2(\Omega(0))} \leq C\epsilon^{1/2}$, the lemma follows.
\end{proof}

We can now prove the
\begin{theorem}\label{WeakEulerianResult}
Let $s \geq 6$ and $k > 0$ be given.  Let $B_0 \in H^{s + 7}$, and $\mathscr{T} > 0$ be given.  Denote by $B(X, T)$ the solution of \eqref{NLS} with initial data $B(0) = B_0$, and let $\zeta^{(1)}$ be defined as in \eqref{Zeta1Formula}.  Suppose that the initial interface $\Sigma(0)$ is given by a graph $\{(x, \eta_0(x)) : x \in \mathbb{R}\}$, the initial velocity is $\mathbf{v}_0$,  the trace of the initial velocity, acceleration on $\{(x, \eta_0(x)) : x \in \mathbb{R}\}$ are $\mathfrak{v}_0$, $\mathfrak{w}_0$,  which satisfy the compatibility conditions as stated in Theorem~\ref{zLocalWellPosed}, and $(\eta_0, \mathfrak{v}_0,\mathfrak{w}_0) \in H^{s + 1} \times H^{s + 1} \times H^{s + 1/2}$ with the remainder estimates
\begin{equation}\label{EulerianRemainderConditions}
\|(|D_x|^{1/2}\eta_0, \mathfrak{v}_0,\mathfrak{w}_0) - \epsilon(\Im |D_x|^{1/2}\zeta^{(1)}(0), \zeta_t^{(1)}(0),\zeta_{tt}^{(1)}(0))\|_{H^{s + 1/2} \times H^{s + 1} \times H^{s + 1/2}} \leq C_1\epsilon^{3/2}
\end{equation} along with 
\begin{equation}\label{WeakConditions}
\|\eta_0\|_{L^2} \leq C_1\epsilon^{1/2} \qquad \text{and} \qquad \|\mathbf{v}_0\|_{L^2(\Omega(0))} \leq C_2\epsilon^{1/2}
\end{equation}
Then there exists an $\epsilon_0 = \epsilon_0(\|B_0\|_{H^{s + 7}}, \mathscr{T}, C_1, C_2)$ so that for all $\epsilon < \epsilon_0$ the following holds:  There exists a solution to \eqref{EulerVelocityField} for times $0 \leq t \leq \mathscr{T}\epsilon^{-2}$ for which $\Sigma(t)$ is given by a graph $\{(x, \eta(x, t)) : x \in \mathbb{R}, t \geq 0\}$, the trace of the velocity field on $\{(x, \eta(x, t)) : x \in \mathbb{R}, t \geq 0\}$ is given by $\mathfrak{v}(x,t)$, 
and which satisfies $$\|(\eta_x(t), \mathfrak{v}(t)) - \epsilon(k\Re \zeta^{(1)}(t),  \zeta_t^{(1)}(t) )\|_{H^{s} \times H^{s}} \leq C(\|B_0\|_{H^{s + 7}}, \mathscr{T}, C_1, C_2)\epsilon^{3/2}$$ for all $0 \leq t \leq \mathscr{T}\epsilon^{-2}$.
\end{theorem}

\begin{proof}
First, we will show that the initial data after being reparametrized by $\kappa^{-1}$, is $B_0$-admissible.  Once we do, a solution $z(\alpha, t)$ exists as in Theorem \ref{MainResult}.  We must then show that this solution can be, for possibly smaller $\epsilon_0$, written as a graph, and we must give remainder estimates for this graph corresponding to the remainder estimates of $\zeta$ in Theorem \ref{MainResult}.

We begin by showing that the reparametrized data satisfies the hypotheses of Theorem \ref{MainResult}.  Let $\gamma_0(\alpha, t) = \alpha + i\eta_0(\alpha, t)$.  Let $\zeta_0(\alpha) = (\gamma_0 \circ \kappa_0^{-1})(\alpha)$, where as in \eqref{ChangeOfVariables} we define $$\kappa_0(\alpha) = \overline{\gamma}_0(\alpha) + \frac{1}{2}(I + \nht_{\gamma_0})(I + \mathcal{K}_{\gamma_0})^{-1}(\gamma_0(\alpha) - \overline{\gamma}_0(\alpha))$$ Then if we denote $\xi_0 := \zeta_0 - \alpha$ as usual, we have $(I - \nhtb_{\zeta_0})\xi_0 = 0$.  This implies that $\xi_0 = i(I + \nhtb_{\zeta_0})(I+\mathcal K_{\zeta_0})^{-1}\Im \xi_0$.  By Proposition \ref{WavePacketAntiholProp} we have $\zeta^{(1)} = i(I + \nhtb_0)\Im \zeta^{(1)} + \mathcal{O}(\epsilon^{3/2})$.  For brevity, temporarily denote $\|\cdot\| := \||D_\alpha|^{1/2} \cdot\|_{H^{s + 1/2}}$.  Then by Lemma \ref{ChangeVarBounds} and interpolation we have
\begin{align*}
\|\xi_0 - \epsilon \zeta^{(1)}(0)\| & \leq \|i(I + \nhtb_{\zeta_0})(I+\mathcal K_{\zeta_0})^{-1}\Im \xi_0 - i(I + \nhtb_0)\Im \epsilon\zeta^{(1)}(0)\| + C\epsilon^{3/2} \\
& \leq \|\Im \xi_0 - \Im \epsilon\zeta^{(1)}(0)\| + \|(\nhtb_{\zeta_0} - \nhtb_0)\Im\epsilon\zeta^{(1)}(0)\| + C\epsilon^{3/2} \\
& \leq C\|\eta_0 - \Im \zeta^{(1)}(0)\| + C\epsilon\|\zeta^{(1)}(0) \circ \kappa_0 - \zeta^{(1)}(0)\| + C\epsilon^{3/2}
\end{align*}  Since $\|\eta_0\|_{H^{s+1}} \leq C\epsilon^{1/2}$ by hypothesis, $\|\kappa_0 - \alpha\|_{H^{s+1}} \leq C\epsilon^{1/2}$, and so by the Mean Value Theorem $\|\zeta^{(1)}(0) \circ \kappa_0 - \zeta^{(1)}(0)\| \leq C\epsilon^{1/2}$.  But then $\|\xi_0 - \epsilon \zeta^{(1)}(0)\| \leq C\epsilon^{3/2}$ follows from $\|\eta_0 - \Im \epsilon \zeta^{(1)}(0)\| \leq C\epsilon^{3/2}$.

Let $v_0 = \mathfrak{v}_0 \circ \kappa^{-1}_0$, $w_0 = \mathfrak{w}_0 \circ \kappa_0^{-1}$. By Lemma~\ref{ChangeVarBounds}, we also have $$\| v_0 - \epsilon i\omega \zeta^{(1)}(0)\|_{H^{s + 1}} \leq C\epsilon^{3/2}$$ $$\| w_0 + \epsilon k \zeta^{(1)}(0)\|_{H^{s + 1/2}} \leq C\epsilon^{3/2}$$  This gives $B_0$-admissible initial data, and so by Theorem \ref{MainResult} there exists a solution to the $\zeta$ system with justified asymptotics.

We must now show that we can give Eulerian estimates for the remainders of this solution.  Since $\zeta$ and $z$ parametrize the same interface $\Sigma(t)$, it suffices to write $\zeta = x + iy$, where
\begin{align}\label{ParametricZeta}
x =x(\alpha,t)& = \alpha + \Re \xi(\alpha,t) \notag \\
y=y(\alpha,t) & = \Im \xi(\alpha,t)
\end{align}  For sufficiently small $\epsilon_0$, $\Sigma(t)$ describes a graph by Lemma \ref{RealPartProfileControl}, and so we can invert to solve for $\alpha = \alpha(x, t)$.  Then we wish to justify asymptotics of $\eta(x,t) := y(\alpha(x,t),t)$.

The rigorous justifications of the asymptotics for $\zeta_\alpha - 1$ and $D_t \zeta$ give rise to rigorous justifications of the quantities $\eta_x$ and $\mathfrak{v}$.  The derivations of each are similar, and so we will focus on $\eta_x$.  By Theorem \ref{MainResult}, we have a solution $\zeta = x + iy$ satisfying $$\|y_\alpha(\cdot, t) - k\epsilon \Re \zeta^{(1)}(\cdot, t)\|_{H^s_\alpha} \leq C\epsilon^{3/2}$$  for sufficiently small $\epsilon_0$, and $\epsilon < \epsilon_0$.  Since $x = \alpha(x, t) + \Re \xi(\alpha(x, t), t)$, we have immediately that $\|\alpha_x - 1\|_{H^s_x} \leq C\epsilon^{1/2}$.  Changing variables then gives us $$\|y_\alpha(\alpha(\cdot,t), t) - k \epsilon \Re \zeta^{(1)}(\alpha(\cdot,t), t)\|_{H^s_x} \leq C\epsilon^{3/2}$$    Moreover, since we would like to take advatage of asymptotics for $\alpha_x(x) - 1$, we write
\begin{align*}
\alpha_x(x) - 1 & = -\Re\xi_\alpha(\alpha(x))\alpha_x(x) \\
& = -\Re\xi_\alpha(\alpha(x)) - \Re\xi_\alpha(\alpha(x))(\alpha_x(x) - 1) \\
& = -\Re\tilde{\xi}_\alpha(\alpha(x)) + \left( \Re(\xi_\alpha(\alpha(x))\Re\Tilde{\xi}_\alpha(\alpha(x)) - \Re r_\alpha(\alpha(x))\right) \\
& \quad - \Re \xi_\alpha(\alpha(x))\left(\alpha_x(x) + \Re\tilde{\xi}_\alpha(\alpha(x)) - 1\right)
\end{align*} from which we have the estimate $$\|\alpha_x(\cdot) + \Re\tilde{\xi}_\alpha(\alpha(\cdot)) - 1\|_{H^s_x} \leq C\epsilon^{3/2}$$
Next, we estimate the derivative of the graph $\eta_x$:
\begin{align*}
\|y_\alpha(\alpha(\cdot), t) - \eta_x(\cdot, t)\|_{H^s_x} & = \|y_\alpha(\alpha(\cdot), t)(\alpha_x(\cdot, t) - 1)\|_{H^s_x} \\
& \leq \|y_\alpha(\alpha(\cdot))\|_{H^s_x}\|\alpha_x(\cdot) + \Re\tilde{\xi}_\alpha(\alpha(\cdot)) - 1\|_{H^s_x} \\
& \qquad + C\|y_\alpha(\alpha(\cdot))\|_{H^s}\|\tilde{\xi}_\alpha\|_{W^{s, \infty}_\alpha} \\
& \leq C\epsilon^{3/2}
\end{align*}
By the Mean Value Theorem and Lemma~\ref{RealPartProfileControl}, we have that
\begin{align*}
& \quad\; \|\Re \zeta^{(1)}(\alpha(x), t) - \Re \zeta^{(1)}(x, t)\|_{H^s_x} \\
& \leq \|B(\epsilon\alpha(x) + \epsilon \omega^\prime t, \epsilon^2 t) - B(\epsilon x + \epsilon \omega^\prime t, \epsilon^2 t)\|_{H^s_x} \\
& \qquad + \|B(\epsilon x + \epsilon \omega^\prime t, \epsilon^2 t)\|_{W^{s, \infty}} \|e^{i(k\alpha(x) + \omega t)} - e^{i(k x + \omega t)}\|_{H^s_x} \\
& \leq \|B\|_{W^{s + 1, \infty}}\|\alpha(x) - x\|_{H^s} \\
& \leq C\epsilon^{1/2}
\end{align*}  Thus we have
\begin{align*}
\|\eta_x(\cdot,t) - k \epsilon \Re \zeta^{(1)}(\cdot, t)\|_{H^s_x} & \leq \|\eta_x(\cdot, t) - y_\alpha(\alpha(\cdot), t)\|_{H^s_x} \\
& \quad + \|y_\alpha(\alpha(\cdot), t) - k\epsilon \Re \zeta^{(1)}(\alpha(\cdot), t)\|_{H^s_x} \\
& \quad + \|\epsilon \zeta^{(1)}(\alpha(\cdot), t) - \epsilon \zeta^{(1)}(\cdot, t)\|_{H^s_x} \\
& \leq C\epsilon^{3/2}
\end{align*} and with a similar argument we also have
\begin{equation*}
\|\mathfrak{v}(\cdot,t) - \epsilon  \zeta_t^{(1)}(\cdot, t)\|_{H^{s + 1}_x} \leq C\epsilon^{3/2}
\end{equation*}
\end{proof}

\subsection{The Multiscale Calculation for $\tilde{\Psi}$ and $\tilde{\lambda}$.}

We have two formal calculations to complete.  The first is to derive an expansion for the quantity $\Psi = \psi \circ \kappa^{-1}$ of the form $\tilde\Psi = \epsilon \Psi^{(1)} + \epsilon^2 \Psi^{(2)} + \epsilon^3 \Psi^{(3)}$  so that it satisfies the transformed version of Bernoulli's principle (c.f. (2.14) of \cite{WuAlmostGlobal2D}): \begin{equation}\label{DtPsi}D_t\Psi = -\Im(\zeta) + \frac{1}{2}|D_t\zeta|^2\end{equation}
up to the order $O(\epsilon^4)$. 
The second is to check whether $\tilde\lambda =(I-\tilde {\mathcal H})\tilde\Psi$ satisfies \eqref{LambdaEvolutionEquation} up to the order $O(\epsilon^4)$.  We will repeatedly use the formula \eqref{TildeZetaFormula} for $\tilde{\zeta}$ in the sequel.

\subsubsection{Deriving the expansion of $\Psi$}

The $O(\epsilon)$ terms of \eqref{DtPsi} yield\footnote{Here $\cc$ represents the complex conjugate of the term immediately preceding it.}
\begin{align*}
\Psi^{(1)}_{t_0} & = -\Im(\zeta^{(1)}) \\
& = -\frac{1}{2i}Be^{i\phi} + \cc
\end{align*} from which we have
\begin{equation*}
\Psi^{(1)} = \frac{1}{2\omega}Be^{i\phi} + \cc + C^{(1)}(\alpha_0, \alpha_1, t_1, t_2)
\end{equation*}  Equating the $O(\epsilon^2)$ terms of \eqref{DtPsi} gives
\begin{align*}
\Psi^{(2)}_{t_0} & = -\Psi^{(1)}_{t_1} - \Im(\zeta^{(2)}) + \frac{1}{2}|\zeta^{(1)}_{t_0}|^2 \\
& = -\omega^\prime\frac{1}{2\omega}B_X e^{i\phi} + \cc - C^{(1)}_{t_1} - \Im\left(\frac{1}{2}ik(I - \nhtb_0)|B|^2\right) + \frac{1}{2}k|B|^2 \\
& = -\frac{1}{4k}B_Xe^{i\phi} + \cc - C^{(1)}_{t_1} - \frac{1}{2}k|B|^2 + \frac{1}{2}k|B|^2 \\
& = -\frac{1}{4k}B_Xe^{i\phi} + \cc - C^{(1)}_{t_1}
\end{align*}  To avoid secular terms  we set $C^{(1)} = 0$ and so arrive at the solution 
\begin{equation}\label{Psi2Formula}
\Psi^{(2)} = -\frac{1}{4ik\omega}B_Xe^{i\phi} + \cc + C^{(2)}(\alpha_0, \alpha_1, t_1, t_2)
\end{equation} and hence determine $\Psi^{(1)}$ as
\begin{equation}\label{Psi1Formula}
\Psi^{(1)} = \frac{1}{2\omega}Be^{i\phi} + \cc
\end{equation}
Finally, we collect the $O(\epsilon^3)$ terms of \eqref{DtPsi} together to give the equation
\begin{align*}
\Psi^{(3)}_{t_0} & = -\Psi^{(2)}_{t_1} - \Psi^{(1)}_{t_2} - b_2 \Psi^{(1)}_{\alpha_0} - \Im(\zeta^{(3)}) + \Re\left(\overline{\zeta}^{(1)}_{t_0}(\zeta^{(1)}_{t_1} + \zeta^{(2)}_{t_0})\right) \\
& = -\Psi^{(2)}_{t_1} - \Psi^{(1)}_{t_2} - b_2 \Psi^{(1)}_{\alpha_0} - \Im(\zeta^{(3)}) + \Re(\overline{\zeta}^{(1)}_{t_0}\zeta^{(1)}_{t_1})
\end{align*}  We calculate that $\Psi^{(2)}_{t_1} = -\frac{1}{8ik^2}B_{XX}e^{i\phi} + \cc + C^{(2)}_{t_1}$ and $\Psi^{(1)}_{t_2} = \frac{1}{2\omega}B_T e^{i\phi} + \cc$.  Recalling from \eqref{TildeBFormula} that $b_2 = -k\omega|B|^2$ we also have $b_2\Psi^{(1)}_{\alpha_0} = -\frac{1}{2}ik^2B|B|^2 e^{i\phi} + \cc$  As for the remaining terms, we can write
\begin{align*}
\Im(\zeta^{(3)}) & = \Im\left(-\frac{1}{2}k^2\overline{B}|B|^2e^{-i\phi} + \frac{1}{2}(I - \nhtb_0)\left(\overline{B}B_X\right)\right) \\
& = \frac{1}{4i}k^2B|B|^2e^{i\phi} + \cc + \frac{1}{2}\Im(I - \nhtb_0)(\overline{B}B_X)
\end{align*} as well as $\Re(\overline{\zeta}^{(1)}_{t_0}\zeta^{(1)}_{t_1}) = \Re\left(-\frac{1}{2}i\overline{B}B_X\right) = \Im\left(\frac{1}{2}\overline{B}B_X\right)$, and so 
\begin{align*}
-\Im(\zeta^{(3)}) + Re(\overline{\zeta}^{(1)}_{t_0}\zeta^{(1)}_{t_1}) & = -\frac{1}{4i}k^2B|B|^2e^{i\phi} + \cc - \frac{1}{2}\Im(I - \nhtb_0)(\overline{B}B_X) + \frac{1}{2} \Im(\overline{B}B_X) \\
& =- \frac{1}{4i}k^2B|B|^2e^{i\phi} + \cc + \frac{1}{2}\Im \nhtb_0(\overline{B}B_X)
\end{align*}  Summing these terms now gives
\begin{align*}
\Psi^{(3)}_{t_0} & = -\Psi^{(2)}_{t_1} - \Psi^{(1)}_{t_2} - b_2 \Psi^{(1)}_{\alpha_0} - \Im(\zeta^{(3)}) + \Re(\overline{\zeta}^{(1)}_{t_0}\zeta^{(1)}_{t_1}) \\
& = \frac{1}{8ik^2}B_{XX}e^{i\phi} + \cc - C^{(2)}_{t_1} - \frac{1}{2\omega}B_Te^{i\phi} + \cc + \frac{1}{2}ik^2B|B|^2e^{i\phi} + \cc \\
& \qquad - \frac{1}{4i}k^2B|B|^2e^{i\phi} + \cc + \frac{1}{2}\Im \nhtb_0(\overline{B}B_X) \\
& = \left(-\frac{1}{2\omega}B_T + \frac{1}{8ik^2}B_{XX} + \frac{3}{4}ik^2B|B|^2\right)e^{i\phi} + \cc + \left(-C^{(2)}_{t_1} + \frac{1}{2}\Im \nhtb_0(\overline{B}B_X)\right) \\
& = -\frac{1}{4i\omega}\left(2iB_T - \frac{1}{2k\omega}B_{XX} + 3k^2\omega B|B|^2\right)e^{i\phi} + \cc + \left(-C^{(2)}_{t_1} + \frac{1}{2}\Im \nhtb_0(\overline{B}B_X)\right) \\
& = -\frac{1}{4i\omega}\left(2iB_T + 2\omega^{\prime\prime}B_{XX} +3 k^2\omega B|B|^2\right)e^{i\phi} + \cc + \left(-C^{(2)}_{t_1} + \frac{1}{2}\Im \nhtb_0(\overline{B}B_X)\right)
\end{align*}  We must choose $C^{(2)}$ so that $C^{(2)}_X = \omega \Im \nhtb_0(\overline{B}B_X)$. Therefore
$$C^{(2)}=\frac12\omega i\mathcal H_0(|B|^2).$$
  Since $B$ satisfies the NLS equation $2iB_T - \omega^{\prime \prime}B_{XX} + k^2\omega B|B|^2 = 0$, we have that $$\Psi^{(3)}_{t_0} = -\frac{3\omega^{\prime\prime}}{4i\omega}B_{XX}e^{i\phi}-\frac{k^2}{2i}B|B|^2e^{i\phi}+ \cc = \frac{3}{16ik^2}B_{XX}e^{i\phi} -\frac{k^2}{2i}B|B|^2e^{i\phi}
+ \cc$$ and so we can take as our solution
\begin{equation}\label{Psi3Formula}
\Psi^{(3)} = -\frac{3}{16k^2\omega}B_{XX}e^{i\phi}+\frac{k^2}{2\omega}B|B|^2e^{i\phi}
 + \cc 
\end{equation}

\subsubsection*{Checking the Evolution Equation for $\lambda$}

Now we would like to use our expansion of $\Psi$ to check to see whether \eqref{LambdaEvolutionEquation} is satisfied up to terms of order $O(\epsilon^4)$.  The $O(\epsilon)$ equation that we must verify is 
\begin{align*}
(\partial_{t_0}^2 - i\partial_{\alpha_0})(I - \nht_0)\Psi^{(1)} & = (\partial_{t_0}^2 - i\partial_{\alpha_0})(I - \nht_0)\left(\frac{1}{2\omega}Be^{i\phi} + \cc\right)\\
& = (\partial_{t_0}^2 - i\partial_{\alpha_0})\frac{1}{\omega}Be^{i\phi}\\
& = 0,
\end{align*} as desired.  Similarly, recalling that $\nht^{(1)} f = [\zeta^{(1)}, \nht_0]f_{\alpha_0}$, it is quick to see that the $O(\epsilon^2)$ terms vanish as well:
\begin{align*}
& \;\; (\partial_{t_0}^2 - i\partial_{\alpha_0})(I - \nht_0)\Psi^{(2)} \\
& + (\partial_{t_0}^2 - i\partial_{\alpha_0})(-\nht_1)\Psi^{(1)} \\
& + (2\partial_{t_0}\partial_{t_1} - i\partial_{\alpha_1})(I - \nht_0)\Psi^{(1)} = 0
\end{align*}  For the $O(\epsilon^3)$ terms, we must investigate the sum
\begin{align*}
& \;\; (\partial_{t_0}^2 - i\partial_{\alpha_0})(I - \nht_0)\Psi^{(3)} \\
& + (\partial_{t_0}^2 - i\partial_{\alpha_0})(-\nht_1)\Psi^{(2)} \\
& + (\partial_{t_0}^2 - i\partial_{\alpha_0})(-\nht_2)\Psi^{(1)} \\
& + (2\partial_{t_0}\partial_{t_1} - i\partial_{\alpha_1})(I - \nht_0)\Psi^{(2)} \\
& + (2\partial_{t_0}\partial_{t_1} - i\partial_{\alpha_1})(-\nht_1)\Psi^{(1)} \\
& + (2\partial_{t_0}\partial_{t_2} + \partial_{t_1}^2 + 2b_2\partial_{t_0}\partial_{\alpha_0})(I - \nht_0)\Psi^{(1)} \\
& - G_\lambda^{(3)} \\
& = I_1 + \cdots + I_6 - G_\lambda^{(3)}
\end{align*} where $G_\lambda^{(3)}$ is the third term in the formal expansion of the cubic term $G_\lambda$ in \eqref{LambdaEvolutionEquation}.  We have
\begin{align*}
I_1 & = (\partial_{t_0}^2 - i\partial_{\alpha_0}) \left(-\frac{3}{8k^2\omega}B_{XX}e^{i\phi}+\frac{k^2}{\omega}B|B|^2e^{i\phi}
\right)\\
& = 0
\end{align*} and
\begin{align*}
I_2 & = (\partial_{t_0}^2 - i\partial_{\alpha_0})(-\nht^{(1)})\left(-\frac{1}{4ik\omega}B_Xe^{i\phi} + \cc + C^{(2)}\right) \\
& = (\partial_{t_0}^2 - i\partial_{\alpha_0})\left(\frac{1}{4\omega}(I - \nht_0)B\overline{B}_X\right) \\
& = 0
\end{align*}  We also have
\begin{align*}
I_4 & = (2\partial_{t_0}\partial_{t_1} - i\partial_{\alpha_1})\left(-\frac{1}{2ik\omega}B_Xe^{i\phi} + (I - \nht_0)C^{(2)}\right) \\
& = -i\omega(I - \nht_0) \Im \nhtb_0 (\overline{B}B_X) \\
& = -i\omega(I - \nht_0)\frac{\nhtb_0(\overline{B}B_X) - \nht_0(B\overline{B}_X)}{2i} \\
& = -\frac{1}{2}\omega(I - \nht_0)(\overline{B}B_X + B\overline{B}_X)
\end{align*} and
\begin{align*}
I_5 & = (2\partial_{t_0}\partial_{t_1} - i\partial_{\alpha_1})(-\nht^{(1)})\left(\frac{1}{2\omega}Be^{i\phi} + \cc\right) \\
& = (2\partial_{t_0}\partial_{t_1} - i\partial_{\alpha_1})\frac{1}{2}i\omega(I - \nht_0)|B|^2 \\
& = \frac{1}{2}\omega(I - \nht_0)(B\overline{B}_X + \overline{B}B_X)
\end{align*}  Moreover, since $B$ satisfies the NLS equation \eqref{NLS},
\begin{align*}
I_6 & = (2\partial_{t_0}\partial_{t_2} + \partial_{t_1}^2 -2k\omega|B|^2\partial_{t_0}\partial_{\alpha_0}) \frac{1}{\omega}Be^{i\phi} \\
& = (2iB_t - \omega^{\prime\prime}B_{XX} + 2k^2\omega B|B|^2)e^{i\phi} \\
& = k^2\omega B|B|^2e^{i\phi}
\end{align*}  The remaining terms are more involved.  Recall the multiscale operator $$\nht^{(2)}f = [\zeta^{(1)}, \nht_0]\partial_{\alpha_1}f + [\zeta^{(2)}, \nht_0]\partial_{\alpha_0}f - [\zeta^{(1)}, \nht_0]\zeta^{(1)}_{\alpha_0}\partial_{\alpha_0}f + \frac{1}{2}[\zeta^{(1)}, [\zeta^{(1)}, \nht_0]]\partial_{\alpha_0}^2f$$  Thus we first have
\begin{align*}
\nht^{(2)} \Psi^{(1)} & = [Be^{i\phi}, \nht_0]\left(\frac{1}{2\omega}B_Xe^{i\phi} + \cc\right) \\
& \quad + \left[\frac{1}{2}ik(I - \nhtb_0)|B|^2, \nht_0\right]\left(\frac{1}{2}i\omega Be^{i\phi} + \cc\right) \\
& \quad - [Be^{i\phi}, \nht_0]\left(ikBe^{i\phi}\left(\frac{1}{2}i\omega Be^{i\phi} + \cc\right)\right) \\
& \quad + \frac{1}{2}[Be^{i\phi}, [Be^{i\phi}, \nht_0]]\left(-\frac{1}{2}k\omega Be^{i\phi} + \cc\right) \\
& = [Be^{i\phi}, \nht_0]\left(\frac{1}{2\omega}\overline{B}_Xe^{-i\phi}\right) \\
& \quad - \frac{1}{2}k\omega[Be^{i\phi}, \nht_0]|B|^2 \\
& \quad + \frac{1}{2}[Be^{i\phi}, [Be^{i\phi}, \nht_0]]\left(-\frac{1}{2}k\omega \overline{B}e^{-i\phi}\right) \\
& = \frac{1}{2\omega}(I - \nht_0)(B\overline{B}_X)\\
& \quad -\frac{1}{2}k\omega Be^{i\phi}(I + \nht_0)|B|^2 \\
& + \frac{1}{2}k\omega Be^{i\phi}\nht_0|B|^2 \\
& = -\frac{1}{2}k\omega B|B|^2e^{i\phi} + \frac{1}{2\omega}(I - \nht_0)(B\overline{B}_X)
\end{align*}  But then
\begin{align*}
I_3 & = (\partial_{t_0}^2 - i\partial_{\alpha_0})(-\nht^{(2)})\Psi^{(1)} \\
& = (\partial_{t_0}^2 - i\partial_{\alpha_0})\left(\frac{1}{2}k\omega B|B|^2e^{i\phi} - \frac{1}{2\omega}(I - \nht_0)(B\overline{B}_X)\right) \\
& = 0
\end{align*}  Finally, we turn to calculating $G_\lambda^{(3)}$.  We have by definition that
$$G_\lambda^{(3)} = -\left[D_t\zeta,\nht\frac{1}{\zeta_\alpha} + \nhtb\frac{1}{\overline{\zeta}_\alpha}\right](\overline{\zeta}_\alpha D_t^2\zeta) + [D_t\zeta, \nhtb]\left(D_t\overline{\zeta}\frac{\partial_\alpha D_t\zeta}{\overline{\zeta}_\alpha}\right) + D_t\zeta[D_t\zeta, \nht]\frac{\partial_\alpha D_t\overline{\zeta}}{\zeta_\alpha}$$
$$ - 2[D_t\zeta, \nht]\frac{D_t\zeta \cdot \partial_\alpha D_t \zeta}{\zeta_\alpha} + \frac{1}{\pi i} \int \left(\frac{D_t\zeta(\alpha) - D_t\zeta(\beta)}{\zeta(\alpha) - \zeta(\beta)}\right)^2(D_t\zeta(\beta) \cdot \zeta_\beta(\beta)) d\beta$$
We simplify the formal leading terms of the commutators first.  We have that
\begin{align*}
[\zeta^{(1)}_{t_0}, \nhtb_0](\overline{\zeta}^{(1)}_{t_0}\zeta^{(1)}_{t_0 \alpha_0}) & = k^2 \omega Be^{i\phi} (I - \nhtb_0)|B|^2
\end{align*} and
\begin{align*}
\zeta^{(1)}_{t_0}[\zeta^{(1)}_{t_0}, \nht_0]\overline{\zeta}^{(1)}_{t_0 \alpha_0} & = k^2\omega Be^{i\phi} (I - \nht_0)|B|^2
\end{align*}  Also, since $\zeta^{(1)}_{t_0} \cdot \zeta^{(1)}_{t_0 \alpha_0} = 0$, the third commutator vanishes.  We will write the leading orders of the remaining terms as singular integrals to which we can apply the following formula:
$$\frac{1}{\pi i}\int \frac{(g(\alpha) - g(\beta))(h(\alpha) - h(\beta))}{(\alpha - \beta)^2} f(\beta) d\beta = [g, \nht_0](h_\alpha f) + [h, \nht_0](g_\alpha f) - [g, [h, \nht_0]]f_\alpha$$  Since to leading order, $\zeta^{(1)}_{t_0} \cdot \zeta^{(1)}_{\alpha_0} = \zeta^{(1)}_{t_0} \cdot 1 + O(\epsilon^2) = \Re(\zeta^{(1)}_{t_0}) + O(\epsilon^2) = \frac{1}{2}(\zeta^{(1)}_{t_0} + \overline{\zeta}^{(1)}_{t_0}) + O(\epsilon^2)$, we can rewrite the second singular integral above as
\begin{align*}
\frac{1}{\pi i} \int \left(\frac{\zeta^{(1)}_{t_0}(\alpha) - \zeta^{(1)}_{t_0}(\beta)}{\alpha - \beta}\right)^2 \frac{1}{2}\overline{\zeta}^{(1)}_{t_0} d\beta & = [\zeta^{(1)}_{t_0}, \nht_0](\zeta^{(1)}_{t_0 \alpha_0} \overline{\zeta}^{(1)}_{t_0}) - \frac{1}{2}[\zeta^{(1)}_{t_0}, [\zeta^{(1)}_{t_0}, \nht_0]]\overline{\zeta}^{(1)}_{\alpha_0 t_0}
\end{align*}
Similarly, the leading order of the first singular integral is
\begin{align*}
& -\frac{1}{\pi i} \int \left( \frac{(\zeta^{(1)}_{t_0}(\alpha) - \zeta^{(1)}_{t_0}(\beta))(\overline{\zeta}^{(1)}(\alpha) - \overline{\zeta}^{(1)}(\beta))}{(\alpha - \beta)^2}\right) \zeta^{(1)}_{t_0 t_0}(\beta) d\beta \\
& = -[\zeta^{(1)}_{t_0}, \nht_0](\overline{\zeta}^{(1)}_{\alpha_0}\zeta^{(1)}_{t_0 t_0}) - [\overline{\zeta}^{(1)}, \nht_0](\zeta^{(1)}_{t_0 \alpha_0}\zeta^{(1)}_{t_0 t_0}) + [\zeta^{(1)}_{t_0}, [\overline{\zeta}^{(1)}, \nht_0]]\zeta^{(1)}_{t_0 t_0 \alpha_0}
\end{align*}  By extracting the coefficients resulting from differentiation, the first terms of these two expressions cancel each other.  Therefore we are left with the following expression as the sum of these singular integrals:
\begin{align*}
& - \frac{1}{2}[\zeta^{(1)}_{t_0}, [\zeta^{(1)}_{t_0}, \nht_0]]\overline{\zeta}^{(1)}_{\alpha_0 t_0} + [\zeta^{(1)}_{t_0}, [\overline{\zeta}^{(1)}, \nht_0]]\zeta^{(1)}_{t_0 t_0 \alpha_0} \\
& = k^2\omega Be^{i\phi} \nht_0|B|^2 - k^2 \omega Be^{i\phi} (I + \nht_0)|B|^2 \\
& = -k^2\omega B|B|^2e^{i\phi}
\end{align*}  Therefore, summing these calculations gives
\begin{align*}
G_\lambda^{(3)} & = k^2\omega Be^{i\phi} (I + \nht_0)|B|^2 + k^2 \omega Be^{i\phi}(I - \nht_0)|B|^2 - k^2 \omega B|B|^2 e^{i\phi} \\
& = k^2 \omega B|B|^2 e^{i\phi}
\end{align*}  Therefore we have at last that the $O(\epsilon^3)$ terms sum to
\begin{align*}
& -\frac{1}{2}\omega(I - \nht_0)(\overline{B}B_X + B\overline{B}_X) + \frac{1}{2}\omega(I - \nht_0)(B \overline{B}_X + \overline{B}B_X) + k^2 \omega B|B|^2e^{i\phi} - k^2\omega B|B|^2e^{i\phi}
\end{align*} which exactly cancels.   Thus the development of $\Psi$ indeed satisfies \eqref{LambdaEvolutionEquation} up to $O(\epsilon^4)$.  Define 
\begin{equation}\label{TildePsiFormula}
\tilde{\Psi} = \epsilon \Psi^{(1)} + \epsilon^2 \Psi^{(2)} + \epsilon^3\Psi^{(3)}
\end{equation} as well as
\begin{equation}\label{TildeLambdaFormula}
\tilde{\lambda} = (I - \tilde{\nht})\tilde{\Psi}
\end{equation} so that $\tilde{\mathcal{P}}\tilde{\lambda} - G_\lambda^{(3)} = O(\epsilon^4)$.

\subsection{Estimates of the Remainder of $\lambda$}

Our goal here is to construct an energy from an evolution equation for \begin{equation}\label{LFormula}
l = \lambda - \tilde{\lambda}
\end{equation}  This will enable us to show that the quantity $D_t l = D_t(\lambda - \tilde{\lambda})$ is bounded in $L^2$.  In turn, we will control $r$ in $L^2$ for $O(\epsilon^{-2})$ times.

\subsubsection{Showing that $D_t l$ and $r$ are comparable.}

Following the proof of Lemma \ref{RealPartProfileControl}, we first show that $r$ and $(I - \nht)\Im(r)$ are comparable in $L^2$.  First, since $(I - \nht)\overline{\xi} = 0$ by \eqref{XiIsAntihol}, we have by the multiscale calculation of Section 3.3 and Corollary \ref{DiffHilbertBound} that $$(I - \nht)\overline{r} = -(I - \nht)\overline{\tilde{\xi}} = -(\tilde{\nht} - \nht)\overline{\tilde{\xi}} - (I - \tilde{\nht})\overline{\tilde{\xi}} = \mathcal{O}(\epsilon^{5/2})$$  Hence we have
\begin{align*}
\|r\|_{L^2} & \leq C\|(I - \nht)(r + \overline{r})\|_{L^2} + C\|(I - \nht)\Im(r)\|_{L^2} \\
& \leq C\|(I - \nht)\Im(r)\|_{L^2} + C\epsilon^{5/2},
\end{align*} and so 
\begin{equation}\label{DtLAndRComparable}
\frac{1}{C} \|r\|_{L^2} - C\epsilon^{5/2} \leq \|(I - \nht)\Im(r)\|_{L^2} \leq C\|r\|_{L^2} + C\epsilon^{5/2}
\end{equation} Turning to $D_t l$ and $r$, we expand
\begin{align*}
D_t l & = D_t \lambda - D_t \tilde{\lambda} \\
& = D_t \lambda - \tilde{D}_t\tilde{\lambda} - (D_t - \tilde{D}_t)\tilde{\lambda} \\
& = D_t \lambda - \tilde{D}_t\tilde{\lambda} - (b - \tilde{b})\tilde{\lambda}_\alpha
\end{align*}  Another multiscale calculation confirms that the residual quantity $$\tilde{D}_t\tilde{\lambda} + (I - \tilde{\nht})\Im(\tilde{\zeta}) + \frac{1}{2}[\tilde{D}_t\tilde{\zeta}, \tilde{\nht}]\frac{\overline{\tilde{\zeta}}_\alpha \tilde{D}_t\tilde{\zeta}}{\tilde{\zeta}_\alpha}$$ is of size at most $C\epsilon^{3/2}$ in $L^2$.  Hence, using \eqref{DtLambdaFormula}, we have that $D_t l = -(I - \nht)\Im(r) + \mathcal{O}(\epsilon^{3/2})$.  But then this implies the bound 
\begin{align}\label{DtLandREquiv}
\frac{1}{C} \|r\|_{L^2} - C\epsilon^{3/2} \leq \|D_t l\|_{L^2} \leq C\|r\|_{L^2} + C\epsilon^{3/2}
\end{align}

\subsubsection{The Evolution Equation and Energy Estimates for $l$.}

We can write immediately that
\begin{align*}
\mathcal{P}l & = G_\lambda - (\mathcal{P} - \tilde{\mathcal{P}})\tilde{\lambda} - \tilde{\mathcal{P}}\tilde{\lambda} \\
& = (G_\lambda - G_\lambda^{(3)}) - (\mathcal{P} - \tilde{\mathcal{P}})\tilde{\lambda} - (\tilde{\mathcal{P}}\tilde{\lambda} - G_\lambda^{(3)})
\end{align*} from which we have by the usual decompositions and estimates that $\mathcal{P}l$ is controlled in $H^s$ by $C(E_s^{3/2} + \epsilon E_s + \epsilon^2 E_s^{1/2} + \epsilon^{7/2}) = O(\epsilon^{7/2})$.  We can now construct the energy $$\int \frac{1}{\mathcal{A}}|D_t l|^2 + il\overline{l}_\alpha d\alpha$$ corresponding to the above evolution equation for $l$.  Since $l$ need not be the trace of a holomorphic function in $\Omega(t)^c$, we cannot conclude that this quantity bounds $\|D_t l\|_{L^2}^2$ above.  Hence we decompose $l$ as $$l = \frac{1}{2}(I - \nht)l + \frac{1}{2}(I + \nht)l := l^- + l^+$$  The energy $$\mathscr{L}(t) = \int \frac{1}{\mathcal{A}}|D_t l|^2 + il^-\overline{l}^-_\alpha d\alpha$$ does bound $\|D_t l\|_{L^2}^2$ from above, by Lemma \ref{BasicEnergyInequality}.  We would like to show that $d\mathscr{L}/dt \leq C\epsilon^5$.  To do so, we write $$\mathscr{L}(t) = \int \frac{1}{\mathcal{A}}|D_t l|^2 + il\overline{l}_\alpha d\alpha - i\int l^-\overline{l}^+_\alpha + l^+\overline{l}^-_\alpha + l^+\overline{l}^+_\alpha d\alpha := \mathscr{L}_1(t) + \mathscr{L}_2(t)$$  By Lemma \ref{BasicEnergyInequality} and \eqref{DtLandREquiv}, the time derivative of the first integral is
\begin{equation*}
\frac{d\mathscr{L}_1}{dt} \leq C \epsilon^{7/2}\|D_t l\|_{L^2} + C \epsilon^{2}\|D_t l\|_{L^2}^2
\end{equation*}  We use the usual almost-orthogonality argument to address the terms of $\mathscr{L}_2(t)$.  Observe that with a change of variables we have
\begin{align*}
\frac{d\mathscr{L}_2}{dt} & = \frac{d}{dt} \left( -i \int l^-\overline{l}^+_\alpha + l^+\overline{l}^-_\alpha + l^+\overline{l}^+_\alpha d\alpha \right) \\
& = -i \int D_t l^-\overline{l}^+_\alpha + D_t l^+\overline{l}^-_\alpha + D_t l^+\overline{l}^+_\alpha + l^-\partial_\alpha D_t\overline{l}^+ + l^+\partial_\alpha D_t \overline{l}^- + l^+\partial_\alpha D_t \overline{l}^+ d\alpha \\
& = \frac{1}{i} \int D_t l^-\overline{l}^+_\alpha + D_t l^+\overline{l}^-_\alpha + D_t l^+\overline{l}^+_\alpha - l^-_\alpha D_t\overline{l}^+ - l^+_\alpha D_t \overline{l}^- - l^+_\alpha D_t \overline{l}^+ d\alpha \\
& = 2 \Im \int D_t l^-\overline{l}^+_\alpha + D_t l^+\overline{l}^-_\alpha + D_t l^+\overline{l}^+_\alpha d\alpha
\end{align*}  We calculate that
\begin{align*}
l^+ & = \frac{1}{2}(I + \nht)l \\
& = \frac{1}{2}(I + \nht)(\lambda - \tilde{\lambda}) \\
& = -\frac{1}{2}(I + \nht)(I - \tilde{\nht})\tilde{\Psi} \\
& =- \frac{1}{2}(I + \nht)(\nht - \tilde{\nht})\tilde{\Psi},
\end{align*} from which we have $\|l^+\|_{H^1} \leq C\epsilon^{5/2}$.  Via Corollary \ref{DtHilbertDiffBound} the same formula readily implies that $\|D_t l^+\|_{L^2} \leq C\epsilon^{5/2}$, and so we clearly have $$\int D_t l^+\overline{l}^+_\alpha d\alpha \leq C\epsilon^5$$  The other two terms are controlled by exploiting their almost-orthogonality.  Note that $D_t l^- = \frac{1}{2}(I - \nht)D_t l - \frac{1}{2}[D_t \zeta, \nht]\frac{l_\alpha}{\zeta_\alpha}$ and $\overline{l}_\alpha^+ = \frac{1}{2}(I - \nhtb^*)\overline{l}_\alpha$.  Since we have
\begin{align*}
l_\alpha & = \lambda_\alpha - \tilde{\lambda_\alpha} \\
& = (I - \nht)\Psi_\alpha - [\xi_\alpha, \nht]\frac{\Psi_\alpha}{\zeta_\alpha} - \tilde{\lambda}_\alpha \\
& = (I - \nht)\Re ( \overline{\zeta}_\alpha D_t\zeta) - [\xi_\alpha, \nht]\frac{\Re(\overline{\zeta}_\alpha D_t\zeta )}{\zeta_\alpha} - \tilde{\lambda}_\alpha
\end{align*} we see that the only $O(\epsilon)$ terms contributed are $(I - \nht_0)\Re(\overline{\zeta}^{(1)}_{t_0}) - \partial_{\alpha_0}(I - \nht_0)\Psi^{(1)} = 0$.  Hence $\|l_\alpha\|_{L^2} \leq C\epsilon^{3/2}$.  But then we can rewrite the commutator as a term of third order as follows:
\begin{align*}
[D_t \zeta, \nht]\frac{l_\alpha}{\zeta_\alpha}& = \left[D_t \zeta, \nht\frac{1}{\zeta_\alpha} + \nhtb\frac{1}{\overline{\zeta}_\alpha}\right]l_\alpha - [D_t \zeta, \nhtb]\frac{\partial_\alpha}{\overline{\zeta}_\alpha}l \\
& = \left[D_t \zeta, \nht\frac{1}{\zeta_\alpha} + \nhtb\frac{1}{\overline{\zeta}_\alpha}\right]l_\alpha - [D_t \zeta, \nhtb]\frac{\partial_\alpha}{\overline{\zeta}_\alpha}\left(l^+ - \frac{1}{2}(\nht + \nhtb)l\right)
\end{align*} and so $\|[D_t \zeta, \nht]\frac{l_\alpha}{\zeta_\alpha}\|_{L^2} \leq C\epsilon^{7/2}$.  Since $\|l^+_\alpha\|_{L^2} \leq C\epsilon^{5/2}$ as above, it suffices to estimate the inner product
\begin{align*}
\langle (I - \nht)D_t l, (I - \nhtb^*)\overline{l}_\alpha\rangle & = -\langle (\nht + \nhtb)D_t l, (I - \nhtb^*)\overline{l}_\alpha\rangle \\
& = -2\langle (\nht + \nhtb)D_t l, \overline{l}_\alpha^+\rangle \\
& \le C\epsilon^{7/2}\|D_t l\|_{L^2}
\end{align*}  For the second term, we have that $D_t l^+ = \frac{1}{2}(I + \nht) D_t l + \frac{1}{2}[D_t \zeta, \nht]\frac{l_\alpha}{\zeta_\alpha}$ and $\overline{l}_\alpha^- = \frac{1}{2}(I + \nhtb^*)\overline{l}_\alpha$.  The commutator is estimated by $\|[D_t \zeta, \nht]\frac{l_\alpha}{\zeta_\alpha}\|_{L^2} \leq C\epsilon^{7/2}$ as before.  Hence it suffices to consider
\begin{align*}
\langle (I + \nht)D_t l, (I + \nhtb^*)\overline{l}_\alpha \rangle & = \langle (I + \nht)D_t l, (\nht^* + \nhtb^*)\overline{l}_\alpha \rangle \\
& = \left\langle 2D_t l^+ - [D_t\zeta, \nht]\frac{l_\alpha}{\zeta_\alpha}, (\nht + \nhtb)^*\overline{l}_\alpha \right\rangle \\
&\le C(\epsilon^{7/2}\|D_t l\|_{L^2}+\epsilon^5)
\end{align*}  Summing these estimates, we finally have that
$$\frac{d\mathscr{L}}{dt}(t) \leq C(\epsilon^5 + \epsilon^{7/2}\mathscr L(t)^{1/2}+ \epsilon^{2}\mathscr L(t))\leq C\epsilon^2(\epsilon^3 + \mathscr L(t))$$ whenever $0 \leq t \leq \mathscr{T}\epsilon^{-2}$.  
Therefore
$$\sup_{0\le t\le \mathscr T \epsilon^{-2}}\mathscr L(t)\le C(\mathscr L(0) + \epsilon^3)$$
Consequently
$$\|r\|_{C([0, \mathscr{T}\epsilon^{-2}]: L^2)} \leq C(\mathscr{L}(0)^{1/2} + \epsilon^{3/2}).$$

\subsection{Justifying the Eulerian Asymptotics for the Profile.}

With these preliminaries, we can now prove the

\begin{theorem}\label{StrongEulerianResult}
Suppose the remainder hypotheses \eqref{EulerianRemainderConditions} hold, and moreover that the stronger conditions hold: 
\begin{equation}\label{StrongerConditions}
\|\eta_0 - \epsilon \Im \zeta^{(1)}\|_{L^2} \leq C\epsilon^{3/2} \qquad \text{and} \qquad \|\Phi_0(\alpha + i\eta_0(\alpha)) - \epsilon \omega^{-1} \Re \zeta^{(1)}\|_{\dot{H}^{1/2}} \leq C\epsilon^{3/2}
\end{equation} where $\Phi_0$ is the initial velocity potential.  Then there exists a possibly smaller $\epsilon_0$ so that in addition to the conclusions of Theorem \ref{WeakEulerianResult} holding, the profile $\eta$ satisfies $$\|\eta(t) - \epsilon\Im \zeta^{(1)}(t)\|_{H^{s + 1}} \leq C(\|B_0\|_{H^{s + 7}}, \mathscr{T}, C_1, C_2)\epsilon^{3/2}$$ for all $0 \leq t \leq \mathscr{T}\epsilon^{-2}$.
\end{theorem}

\begin{proof}
As in the proof of Lemma \ref{RealPartProfileControl}, it suffices to derive conditions under which $\mathscr{L}(0) = O(\epsilon^3)$.  We will show that the quantity $$\mathscr{L}(0) = \int \frac{1}{\mathcal{A}_0} |D_t l_0|^2 + i l_0 \partial_\alpha \overline{l}_0 d\alpha$$ is controlled by $\|r_0\|_{L^2}$ and $\|\Phi_0 \circ z_0 - \epsilon \omega^{-1} \Re \zeta^{(1)}\|_{\dot{H}^{1/2}}$.\footnote{Ideally one would prefer, in keeping with the weaker conditions given in Theorem \ref{WeakEulerianResult}, to bound $\mathscr{L}(0)$ by some difference of the initial velocity fields of the true and approximate solution in the square-mean.  However, since these velocity fields are defined in different domains, we instead give this equivalent condition, which is more straightforward to state.}

We can control the first term $$\int \frac{1}{\mathcal{A}_0}|D_t l(0)|^2 \leq C\|D_t l_0\|_{L^2}^2 \leq C(\|r_0\|_{L^2} + \epsilon^{3/2})^2$$ by \eqref{DtLAndRComparable}.

To estimate the other term in $\mathscr{L}(0)$, observe that we can write $l$ in terms of $\Psi - \tilde{\Psi}$ as follows:
\begin{align*}
l & = (I - \nht)\Psi - (I - \tilde{\nht})\tilde{\Psi} \\
& = (I - \nht)(\Psi - \tilde{\Psi}) - (\nht - \tilde{\nht})\tilde{\Psi}
\end{align*} and the latter term is $\mathcal{O}(\epsilon^{5/2})$ by Corollary \ref{DiffHilbertBound}.  Hence we expand our integral as usual:
\begin{align*}
\int i l_0 \partial_\alpha \overline{l}_0 d\alpha & = i \int \left(l_0 - (I - \nht_{z_0})(\Psi_0 - \tilde{\Psi}_0)\right) \partial_\alpha \overline{l}_0 d\alpha \\
& \quad - i \int \partial_\alpha(I - \nht_{z_0})(\Psi_0 - \tilde{\Psi}_0) \overline{\left(l_0 - (I - \nht_{z_0})(\Psi_0 - \tilde{\Psi}_0)\right)} d\alpha \\
& \quad + i \int (I - \nht_{z_0})(\Psi_0 - \tilde{\Psi}_0) \partial_\alpha \overline{(I - \nht_{z_0})(\Psi_0 - \tilde{\Psi}_0)} d\alpha
\end{align*}  The first two of these integrals are $O(\epsilon^4)$, since $\partial_\alpha(\Psi - \tilde{\Psi}) = \Re(\overline{\zeta}_\alpha D_t\zeta) - \tilde{\Psi}_\alpha$ is $\mathcal{O}(\epsilon^{3/2})$.  Therefore since $\nht$ is bounded on $\dot{H}^{1/2}$,\footnote{Since $\nht$ is bounded on $L^2$, this can be checked by showing that $\nht$ is bounded on $\dot{H}^1$ using the identity $\partial_\alpha \nht f = \nht f_\alpha + [z_\alpha, \nht]\frac{f_\alpha}{z_\alpha}$ and then by using complex interpolation.} it follows that $$\left| \int i l_0 \partial_\alpha \overline{l}_0 d\alpha \right| \leq \|(\Phi_0 \circ z_0) - \epsilon \omega^{-1} \Re \zeta^{(1)} \|_{\dot{H}^{1/2}}^2 + C\epsilon^4$$  Hence, if we choose the initial profile and the initial velocity potential $\Phi_0$ to satisfy $$\|r_0\|_{L^2} \leq C\epsilon^{3/2} \qquad \text{and} \qquad \|(\Phi_0 \circ z_0) - \epsilon \omega^{-1} \Re \zeta^{(1)} \|_{\dot{H}^{1/2}} \leq C\epsilon^{3/2}$$ then $\mathscr{L}(0) \leq C\epsilon^3$, and so $\sup_{0\le t\le \mathscr T\epsilon^{-2}}\|r(t)\|_{L^2} \leq C\epsilon^{3/2}$ as well.
\end{proof}

\textit{Acknowledgement:} Part of the work in this paper was done while the authors were visiting at the IMA during the academic year 2009-10. S. Wu would like to thank the IMA for their hospitality, generous support and pleasant academic environment.  N. Totz would like to thank the IMA for its generous travel funding during 2009-10.

\newpage

\appendix

\section{Glossary of Symbols}

We collect the commonly used notations and symbols.  References such as (1.1) refer to the equation in which the symbol was introduced, whereas p. 1 refers to the page number in which the symbol is first used.

\begin{align*}
\mathbb{N} & \qquad \text{The set of nonnegative integers} \\
\mathbb{R} & \qquad \text{The set of real numbers} \\
\mathbb{C} & \qquad \text{The set of complex numbers} \\
\overline{w} & \qquad \text{The complex conjugate of } w \in \mathbb{C} \\
\Re(w) & \qquad \text{The real part of } w \in \mathbb{C} \\
\Im(w) & \qquad \text{The imaginary part of } w \in \mathbb{C} \\
[F, G] & \qquad \text{The commutator }FG - GF\text{ of the operators }F\text{ and }G \\
U_g & \qquad \text{Precomposition by }g\text{, p. 2}\\
\langle \cdot, \cdot\rangle & \qquad \text{The real inner product on }L^2\text{, p. 38} \\
T^* & \qquad \text{The formal real adjoint of a linear operator } T\text{, p. 38} \\
\Omega(t) & \qquad \text{The fluid region associated to }z\text{ at time }t\text{, p. 1} \\
\Sigma(t) & \qquad \text{The boundary of the fluid region associated to }z\text{ at time }t\text{, p. 1} \\
z(\alpha, t) & \qquad \text{The parametrization of }\Sigma(t)\text{ in Lagrangian coordinates }\alpha\text{, p. 1} \\
\mathfrak{a} & \qquad \text{see p. 1}\\
\pv \int & \qquad \text{The principal value integral, p. 1} \\
\oht & \qquad \text{The Hilbert transform associated to }z\text{, p. 1} \\
\mathfrak{K} & \qquad \text{The double layer potential operator associated to }z\text{, p. 4} \\
\kappa & \qquad \text{The change of variables taking }z\text{ to }\zeta\text{, \eqref{ChangeOfVariables}} \\
\zeta & \qquad \text{The new water wave interface variable, \eqref{NewVariableNotation}} \\
\xi & \qquad \text{The perturbation of }\zeta\text{ from the still water solution, p. 5} \\
\nht & \qquad \text{The Hilbert transform associated to }\zeta\text{, p. 3} \\
\mathcal{K} & \qquad \text{The double layer potential operator associated to the curve }\zeta\text{, p. 51} \\
\nht_\gamma & \qquad \text{The Hilbert transform associated to the curve }\gamma\text{, p. 3} \\
\mathcal{K}_\gamma & \qquad \text{The double layer potential operator associated to the curve }\gamma \\
\nht_0 & \qquad \text{The Hilbert transform associated to the curve }\alpha\text{, p. 3} \\
D_t & \qquad \text{The transformed time derivative, \eqref{NewVariableNotation}} \\
\mathcal{P} & \qquad \text{The transformed linear water wave operator, \eqref{NewVariableNotation}} \\
b & \qquad \text{see \eqref{NewVariableNotation}} \\
\mathcal{A} & \qquad \text{see \eqref{NewVariableNotation}} \\
\end{align*}
\begin{align*}
G & \qquad \text{The cubic nonlinearity of the transformed water wave equation, \eqref{GFormula}} \\
\hat{f} & \qquad \text{The Fourier transform of }f\text{, p. 6} \\
H^s & \qquad \text{The }L^2\text{ Sobolev space of index }s\text{, p. 6} \\
\dot{H}^s & \qquad \text{The }L^2\text{ homogeneous Sobolev space of index }s\text{, p. 6} \\
W^{s, \infty} & \qquad \text{The }L^\infty\text{ Sobolev space of index }s\text{, p. 6} \\
C([0, T]; X) & \qquad \text{The Banach space of functions }f \in X \times [0, T]\text{ with} \\
& \qquad \qquad \|f\|_X\text{ varying continuously in }[0, T]\text{, p. 6} \\
\mathfrak{S}(t) & \qquad \text{The supremum of a modified energy of }\zeta\text{, \eqref{ZetaLocalAPrioriBound}} \\
\zeta^{(n)} & \qquad \text{Terms of the formal power expansion of }\zeta\text{ in }\epsilon\text{, \eqref{TildeZetaFormula}} \\
\nht_n & \qquad n\text{th order term of the formal expansion of }\nht\text{, \eqref{HilbertFormulas}} \\
\nht^{(n)} & \qquad \text{Operator at the order }\epsilon^n\text{ of the formal expansion of }\nht \\
& \qquad \qquad \text{ acting on a multiscale function, \eqref{MultiscaleHilbertFormulas}} \\
\tilde{\zeta} & \qquad \text{Formal multiscale approximation of }\zeta\text{, see \eqref{TildeZetaFormula}} \\
k & \qquad \text{The wave number of the wave packet approximation to }\zeta\text{, p. 12} \\
\omega & \qquad \text{The wave frequency of the wave packet approximation to }\zeta\text{, p. 12} \\
\phi & \qquad \text{The phase of the wave packet approximation to }\zeta\text{, p. 12} \\
\omega^\prime & \qquad \text{The group velocity, p. 13} \\
\omega^{\prime \prime} & \qquad \text{The dispersion coefficent, p. 15} \\
B & \qquad \text{The slowly varying envelope of the leading order of }\tilde{\zeta}\text{, \eqref{Zeta1Formula}}\\
\tilde{\xi} & \qquad \text{Perturbation of }\tilde{\zeta}\text{ from the still water solution, p. 17} \\
\tilde{b} & \qquad \text{see \eqref{TildeBFormula}} \\
\tilde{\mathcal{A}} & \qquad \text{see \eqref{TildeAFormula}} \\
\tilde{D}_t & \qquad \text{see \eqref{TildeDtFormulas}} \\
\tilde{\mathcal{P}} & \qquad \text{see \eqref{TildeDtFormulas}} \\
r & \qquad \text{The difference between the true solution }\zeta\text{ and the} \\
& \qquad \qquad \text{approximate solution }\tilde{\zeta}\text{ of the water wave equations, p. 18} \\
\mathcal{O}(\epsilon^n) & \qquad \text{Landau notation for functions in }H^s\text{, p. 18} \\
E_s & \qquad \text{The modified energy of }r\text{, p. 18} \\
\mathcal{E} & \qquad \text{The energy of }r\text{, \eqref{FullEnergyDefinition}} \\
\mathscr{A}^s & \qquad \text{The manifold of admissible initial conditons for }(z, z_t)\text{, p. 41} \\
+\,\cc & \qquad \text{Adds the complex conjugate of the preceding term, p. 52}
\end{align*}

\section{Estimates of Singular Integrals in Sobolev Space}

The purpose of this appendix is to provide bounds for singular integrals of the form 
\begin{equation}\label{S1Formula}
S_1(A, f) = \int \prod_{j = 1}^m \frac{A_j(\alpha) - A_j(\beta)}{\gamma_j(\alpha) - \gamma_j(\beta)} \frac{f(\beta)}{\gamma_0(\alpha) - \gamma_0(\beta)} d\beta
\end{equation} and
\begin{equation}\label{S2Formula}
S_2(A, f) = \int \prod_{j = 1}^m \frac{A_j(\alpha) - A_j(\beta)}{\gamma_j(\alpha) - \gamma_j(\beta)} f_\beta(\beta) d\beta\end{equation} in Sobolev space.  For these singular integrals to be well-defined we insist that the $\gamma_j$ each obey the chord-arc condition \eqref{ChordArcCondition}.  Our starting point is the result of Coifman-Meyer-McIntosh, expanded upon by Wu, which bounds these singular integrals in $L^2$.

\begin{theorem}\label{CMMWEstimates}
(c.f. \cite{CoifmanMeyerMcintoshL2Bounds} and \cite{WuAlmostGlobal2D})  Both $\|S_1(A, f)\|_{L^2}$ and $\|S_2(A, f)\|_{L^2}$ are bounded by $$C \prod_{j = 1}^m \|A_j^\prime\|_{X_j} \|f\|_{X_0},$$ where one of the $X_0, X_1, \ldots, X_n$ is equal to $L^2$ and the rest are $L^\infty$.  The constant $C$ depends $\|\gamma_0^\prime\|_{L^\infty}, \|\gamma_1^\prime\|_{L^\infty}, \ldots, \|\gamma_m^\prime\|_{L^\infty}$.
\end{theorem}

Observe that the kernels of the operators $S_1$ and $S_2$ are functions of differences of the form $F\left(f_1(\alpha) - f_1(\beta), \ldots, f_n(\alpha) - f_n(\beta)\right)$.  When the differential operator $(\partial_\alpha + \partial_\beta)$ acts on such differences of functions, it yields another function of the same kind, e.g., the Chain Rule becomes $$(\partial_\alpha + \partial_\beta)F\left(f_1(\alpha) - f_1(\beta), \ldots, f_n(\alpha) - f_n(\beta)\right) = \sum_{i = 1}^n (\partial_i F) (\partial_\alpha + \partial_\beta)(f_i(\alpha) - f_i(\beta))$$  The other rules of differential calculus hold as well.  Hence acting on kernels of $S_1$ or $S_2$ with $m$ factors by $(\partial_\alpha + \partial_\beta)$ yields another kernel  which is a sum  of the same type with $m + 1$ factors.  This allows us to cleanly prove the following

\begin{proposition}\label{SingIntSobolevEstimates}
Let $n \geq 3$ be given, and suppose that \eqref{ChordArcCondition} holds.  Then $$\|S_2(A, f)\|_{H^n} \leq C \prod_{j = 1}^m \|A'_j\|_{Y_j}\|f\|_Z,$$ where for all $j = 1, \ldots, m$ the Banach spaces $Y_j = H^{n-1} \text{ or } W^{n - 2, \infty}$, $Z= H^n \text{ or } W^{n - 1, \infty}$.
Moreover, the constant $C = C\left( \|\partial_\alpha \gamma_j - 1\|_{H^{n-1}}, j=1,\dots, m\right)$.
\end{proposition}

\begin{proof}
Write $S_2 f = \int K(\alpha, \beta) f_\beta(\beta) d\beta$.  To exploit the observations preceding the theorem, we expand $\partial_\alpha^n S_2 f$ using the Binomial Theorem applied to $((\partial_\alpha + \partial_\beta) - \partial_\beta)^n$:
\begin{align*}
\partial_\alpha^n S_2 f(\alpha) & = \sum_{j=0}^n \binom n j\int   (-1)^{j}(\partial_\alpha + \partial_\beta)^{n-j}\partial_\beta^{j} K(\alpha, \beta) f_\beta(\beta) d\beta \\
& = \sum_{j= 0}^n\binom n j \int (\partial_\alpha + \partial_\beta)^{n-j}K(\alpha, \beta) \partial_\beta^{j} f_\beta(\beta) d\beta
\end{align*}  After applying routine calculus identities, we see that $(\partial_\alpha + \partial_\beta)^{n-j}K(\alpha, \beta)$ yields a sum of terms, each of which is another kernel expressible in the form \eqref{S2Formula}.  Now we apply Theorem \ref{CMMWEstimates} to estimate each term in $L^2$.  

We proceed by cases.  Since $n \geq 3$, it suffices to consider the cases where $j = 0$ and $j = 1$; in all other cases one can estimate however one pleases using Theorem \ref{CMMWEstimates}.  If a difference of the form $A_j^{(n - 1)}(\alpha) - A_j^{(n - 1)}(\beta)$ or $\gamma_j^{(n-1)}(\alpha) - \gamma_j^{(n-1)}(\beta)$ occurs in some kernel, estimate this difference in $L^2$; observe that only one of these can occur in a given singular integral since $n \geq 3$.  If a difference of the form $A_j^{(n)}(\alpha) - A_j^{(n)}(\beta)$ or $\gamma_j^{(n)}(\alpha) - \gamma_j^{(n)}(\beta)$ occurs in some kernel, split the integral into a difference of singular integrals of the form $S_1$ and estimate using Theorem \ref{CMMWEstimates}.
\end{proof}

\bibliography{Mybib}{}
\bibliographystyle{plain}

\end{document}